\documentclass[reqno,centertags,12pt]{amsart}
\usepackage{amsmath,amsthm,amscd,amssymb,latexsym,verbatim}
\usepackage{amssymb}
\usepackage{graphicx,epsf,cite}

\usepackage[shortlabels]{enumitem}

\usepackage{xcolor}

-\marginparwidth \oddsidemargin -4mm \evensidemargin \oddsidemargin

\newtheorem{theorem}{Theorem}[section]

\newtheorem{proposition}[theorem]{Proposition}
\newtheorem{lemma}[theorem]{Lemma}
\newtheorem{corollary}[theorem]{Corollary}
\theoremstyle{definition}
\newtheorem{definition}[theorem]{Definition}
\newtheorem{example}[theorem]{Example}

\newcounter{smalllist}

\DeclareMathOperator{\diam}{diam}

\allowdisplaybreaks
\numberwithin{equation}{section}

\newcommand{\lb}{\label}

\newcommand{\beq}{\begin{equation}}
\newcommand{\eeq}{\end{equation}}

\newcommand{\bal}{\begin{align}}
\newcommand{\eal}{\end{align}}
\newcommand{\bals}{\begin{align*}}
\newcommand{\eals}{\end{align*}}

\newcommand{\bbN}{{\mathbb{N}}}
\newcommand{\bbR}{{\mathbb{R}}}

\newcommand{\bbP}{{\mathbb{P}}}
\newcommand{\bbE}{{\mathbb{E}}}
\newcommand{\bbZ}{{\mathbb{Z}}}

\newcommand{\bbQ}{{\mathbb{Q}}}

\newcommand{\bbS}{{\mathbb{S}}}

\newcommand{\calE}{{\mathcal E}}
\newcommand{\calS}{{\mathcal S}}

\newcommand{\calC}{{\mathcal C}}
\newcommand{\calF}{{\mathcal F}}
\newcommand{\calH}{{\mathcal H}}
\newcommand{\calK}{{\mathcal K}}
\newcommand{\calG}{{\mathcal G}}

\newcommand{\calU}{{\mathcal U}}

\newcommand{\eps}{\varepsilon}

\newcommand{\al}{\alpha}
\newcommand{\be}{\beta}

\newcommand{\izero}{\iota}

\begin{document}
\title[Long Time Dynamics for Combustion in Random Media]
{Long Time Dynamics for Combustion in Random Media}

\author{Yuming Paul Zhang and Andrej Zlato\v s}

\address{\noindent Department of Mathematics \\ University of
California San Diego \\ La Jolla, CA 92093 \newline Email: \tt
zlatos@ucsd.edu}

\address{\noindent Department of Mathematics \\ University of
California San Diego \\ La Jolla, CA 92093 \newline Email: \tt
yzhangpaul@ucsd.edu}


\begin{abstract} 
We study long time dynamics of combustive processes in random media, modeled by reaction-diffusion equations with random ignition reactions.  One expects that under reasonable hypotheses on the randomness, large space-time scale dynamics of solutions to these equations  is almost surely governed by a different effective PDE, which should be a homogeneous Hamilton-Jacobi equation.  While this was previously proved in one dimension as well as for isotropic reactions  in several dimensions (i.e., with radially symmetric laws), we provide here the first proof of this phenomenon in the general non-isotropic multidimensional setting.  Our results hold for reactions that have finite ranges of dependence (i.e., their values are independent at sufficiently distant points in space) as well as for 
some with
infinite ranges of dependence, and are based on proving existence of deterministic front (propagation) speeds in all directions for these reactions.
\end{abstract}

\maketitle

\section{Introduction} 

The reaction-diffusion equation
\beq\lb{1.1}
u_t=\Delta u+f(x,u,\omega),
\eeq
with $(t,x)\in (0,\infty)\times \bbR^d$ and $\omega$ an element of some probability space $(\Omega,\calF,\bbP)$, models a host of physical phenomena occurring in random media.  These phenomena all exhibit diffusion, modeled by the Laplacian, as well as some reactive process, modeled by the non-linear {\it reaction function} $f$.  The nature of the latter process determines the behavior of $f$ in the variable $u$, which models the property under study and will take values between its minimum and maximum, customarily normalized to be 0 and 1.

When $u=0$ is an unstable equilibrium for the $(x,\omega)$-dependent ODE $\dot u=f(x,u,\omega)$ and $u=1$ a stable one (e.g., when $f>0$ for $u\in(0,1)$), the reaction is of the {\it monostable type}.  A special case of this are the {\it Kolmogorov-Petrovskii-Piskunov (KPP)} or {\it Fisher-KPP type}  reactions \cite{Fisher, KPP}, for which the growth rate  $u^{-1}f(x,u,\omega)$ of the reactive process is largest near $u=0$ for each $(x,\omega)$ (e.g., in the case of logistic growth functions $f(x,u,\omega)=g(x,\omega)u(1-u)$).  These reactions are used in, for instance, population dynamics models, with $u$ being the normalized population density and $u^{-1}f(x,u,\omega)$ the sum of the birth and death rates.
When both $u=0$ and $u=1$ are asymptotically stable equilibria for $\dot u=f(x,u,\omega)$ (e.g., when $f(x,u,\omega)=g(x,\omega)u(1-u)(u-h(x,\omega))$ with $h(x,\omega)\in(0,1)$), the reaction is of the {\it bistable type}, used in modeling   phase transition processes.

In this paper we will consider the third main type of reactions, modeling various combustive processes, including forest fires.  Here $u$ is the normalized temperature and $f$ vanishes for all $u$ below some possibly $(x,\omega)$-dependent ignition temperature (so $u=0$ is typically a stable but not asymptotically stable equilibrium), which is why these reactions are of the {\it ignition type}.  Our interest is in the long term dynamics of solutions to \eqref{1.1}.  The PDE typically exhibits ballistic propagation of solutions, which means that the state $u\sim 1$ invades the region where initially $u\sim 0$ at a linear-in-time rate.    
If the medium is sufficiently random, one expects this invasion to acquire a deterministic  asymptotic speed as $t\to\infty$, which may depend on the invading direction but not on the position (or $\omega$), due to averaging of the variations in the medium over long distances.

This phenomenon is called {\it homogenization}, because over large space-time scales, solutions behave as if the medium were possibly non-isotropic but homogeneous (i.e., direction- but not position-dependent).  One can study solutions on these scales by rescaling them via the transformation
\beq\lb{1.3}
u_\eps(t,x,\omega):=u\left({\eps}^{-1} t,{\eps}^{-1} x,\omega\right),
\eeq
with $\eps>0$ small, which  turns \eqref{1.1} into
\beq\lb{1.4}
( u_\eps)_t=\eps\Delta u_\eps+{\eps}^{-1}f\left({\eps}^{-1} x, u_\eps,\omega\right).
\eeq
If we now take $\eps\to 0$, the hope is to recover some (almost surely) $\omega$-independent limit $u_\eps\to\bar u$, in an appropriate sense and for appropriate initial data $u_\eps(0,\cdot,\omega)$, that should ideally also satisfy some limiting {\it effective PDE}.

However, unlike in typical homogenization scenarios, the limiting PDE for reaction-diffusion equations cannot be another reaction-diffusion equation, or even another second order parabolic PDE.  The reason for this is that one expects solutions to exhibit uniformly bounded in time width of the regions where transition between values $u\sim 0$ and $u\sim 1$ happens, which means that this width becomes zero in the scaling from \eqref{1.3} as $\eps\to 0$, and any limiting function $\bar u$ takes only values 0 and 1.  For instance, in the homogeneous deterministic reaction case $f(x,u,\omega)=f(u)$, the simplest solutions are {\it traveling fronts}, which are of the form $u(t,x)=U(x\cdot e-ct)$ for some vector $e\in\bbS^{d-1}$, where the front profile and speed $(U,c)$ solve the ODE $U''+cU'+f(U)=0$ with boundary values $U(-\infty)=1$ and $U(\infty)=0$.  Clearly, the region where $u(t,\cdot)\in[\eta,1-\eta]$ for any fixed $\eta>0$ is a slab of a constant-in-$t$ width that shrinks to zero as we take $\eps\to 0$ in \eqref{1.3}.  But then the limiting solution will be the (discontinuous) characteristic function of the half-space-time $\{x\cdot e<ct\}$,
which does not solve a second order parabolic PDE.

This suggests that any effective equation should be of the first order, with any limiting function $\bar u$ being its discontinuous solution, taking only values 0 and 1.  The expectation of the effective (asymptotic) propagation speeds being direction- but not position-dependent then suggests that the effective PDE should be the {\it Hamilton-Jacobi equation}
\beq\lb{1.5}
\bar{u}_t = c^*\left(- \frac{\nabla \overline{u}}{|\nabla \overline{u}|} \right)|\nabla \bar{u}|,
\eeq
with $c^*(e)$ being the $(x,\omega)$-independent effective propagation speed in direction $e\in \bbS^{d-1}$.  Moreover, the traveling front solutions above suggest that in the deterministic homogeneous reaction case, the speed $c^*(e)$ should be precisely the traveling front speed $c$ (which is also direction-independent in that case).  One may therefore hope that in the general random case, it is also possible to find some front-like solutions in all directions $e\in \bbS^{d-1}$, and that each of these has an associated speed $c^*(e)$ in some sense.

Unfortunately, there are some serious obstacles to realizing this hope.  The first is that its basic premise, that the width of the transition region where $u(t,\cdot,\omega)\in[\eta,1-\eta]$ stays uniformly bounded  in time (or at least $o(t)$) for any fixed $\eta>0$, may not be true in some media.  The second author in fact showed that this need not happen for bistable reactions, even for periodic ones in one dimension \cite{ZlaBist}, where solutions can develop linearly-in-time growing intervals on which they are close to periodic functions with values strictly away from 0 and 1.  As a result, there may be no analog of a traveling front for such reactions, and hence no homogenization as described above.   

Recalling pictures of forest fires, which are usually actively burning only along the margins of the already burnt area, one may hope that such issues do not occur for ignition reactions.  The second author showed that this is indeed the case in dimensions $d\le 3$ \cite{ZlaInhomog}, where the widths of the transition regions (properly defined, as these regions may have  complicated geometries in heterogeneous media; see \eqref{d.2.1} below) indeed remain uniformly bounded in time, by constants depending on $\eta$ above and some bounds on the reaction.  However, he also showed in \cite{ZlaInhomog} that this need not be the case in dimensions $d\ge 4$, where these widths may grow linearly in time as in the above bistable example.  Nevertheless, the relevant examples have a special structure and it is not clear to what extent they indicate possible almost sure behaviors of solutions for various stochastic reactions (in particular, those with finite ranges of spatial dependence).

All this demonstrates the difficulties associated with even the question whether solutions to \eqref{1.1} have some basic properties required for one to be able to initiate the study of homogenization for \eqref{1.1}.  This is the reason for relatively little progress in this area, until recently, particularly in the multi-dimensional case $d\ge 2$.  In the one-dimensional setting, there are only two directions of propagation of solutions, and homogenization simply refers to showing that solutions starting from large enough compactly supported initial data propagate almost surely with some deterministic asymptotic speeds $c_+$ (to the right) and $c_-$ (to the left).  Moreover,  the transition regions (which are intervals) have trivial geometries.  This allowed several authors to obtain such ``homogenization'' results in this setting for all three types of stationary ergodic reactions --- KPP \cite{GF}, ignition \cite{NolRyz, ZlaGenfronts}, and bistable \cite{NolRyz, VakVol, ZlaBist} --- although with some non-trivial limitations in the latter case, due to the counterexamples from \cite{ZlaBist} mentioned above.  There are also a number of 1D and quasi-1D results concerning related models and/or periodic reactions, which we do not discuss here.

Once we move to higher dimensions, the geometry of the level sets of solutions becomes much more complicated, and relatively little is known.  One previous  result appears in the paper \cite{LioSou} by Lions and Souganidis, which studies homogenization for viscous Hamilton-Jacobi equations.   Their Theorem 9.3 states that homogenization also holds for general stationary ergodic KPP reactions in any dimension.  (While it is indicated in \cite{LioSou} that its proof can be obtained via methods from \cite{LioSou,38majda1994} and two other papers, a proof is not provided there.)
The reason why Hamilton-Jacobi homogenization techniques should be applicable to KPP reaction-diffusion equations is that the dynamics of solutions for these reactions is determined, to the leading order, by the linearization of \eqref{1.1} (i.e., of $f$) at $u=0$.   This linear PDE can then be turned into a viscous Hamilton-Jacobi equation with a convex Hamiltonian  via the Hopf-Cole transformation.  

This linearization approach can only work for KPP reactions, and is not applicable to other types, including other monostable ones.  In particular, it cannot be used in ignition-reaction-based models of combustion, where one has to work with the original non-linear PDE.  Because of this complication, so far there has only been a single result proving homogenization for (non-KPP) stationary ergodic reactions in several dimensions.  This is a conditional result by Lin and the second author \cite{zlatos2019}, who proved homogenization for ignition reactions whose Wulff shapes exist and have no corners (a Wulff shape for \eqref{1.1}, if it exists, is an open set $\calS\subseteq\bbR^d$ such that solutions starting from any large enough compactly supported initial data converge to $\chi_\calS$ as $t\to\infty$, after being scaled down by $t$ in space).  They also showed that these properties hold for isotropic ignition reactions in dimensions $d\le 3$, with  the dimension limitation being used to show that  the Wulff shape exists (recall the above-mentioned examples of solutions with linearly growing widths of transition regions  in dimensions $d\ge 4$ from \cite{ZlaInhomog})
and  isotropy then guaranteeing that the Wulff shape is a (corner-less) ball centered at the origin.  We also note that it follows from a result of Caffarelli, Lee, and Mellet \cite{caffarelli2006} that  Wulff shapes can have corners, even for periodic ignition reactions in two dimensions.

In fact, even homogenization for periodic  reactions in several dimensions has seen fairly limited progress until recently, despite many results concerning existence of pulsating fronts and Wulff shapes for such reactions (see \cite{16berestycki2002, zlatos2019, weinberger2002, 49xin1992} and references therein).   While Theorem~9.3 in \cite{LioSou} applies to periodic KPP  reactions (and is based in part on methods from \cite{38majda1994}, applicable to KPP reactions in periodic media),  homogenization for  periodic non-KPP  reactions in several dimensions has only recently been obtained for ignition reactions as a byproduct of the method in \cite{zlatos2019}, as well as  for  monostable reactions  by Alfaro and Giletti \cite{AlfGil} (for initial data with smooth convex supports, later extended to general convex supports in \cite{zlatos2019}).

In this paper we prove for the first time unconditional stochastic homogenization for ignition reactions, without assuming the reaction to be isotropic.  Our Theorems \ref{T.1.1} and \ref{T.1.2} below are valid for random pure ignition reactions (see Definition \ref{D.2}) in dimensions $d\le 3$ that either have a finite range of dependence (see Definition \ref{D.1}) or can be uniformly approximated by such reactions.  We also extend these results in Theorems \ref{T.1.3} and \ref{T.1.7} to ignition reactions in any dimension, provided some a priori assumptions on the dynamics of certain special solutions to \eqref{1.1} are satisfied.  

Our proof uses a result from \cite{zlatos2019}, which shows that to prove homogenization, it suffices to show that the above-mentioned propagation speeds $c^*(e)$ (called {\it front speeds}) exist for all directions $e\in\bbS^{d-1}$, are almost-surely $\omega$-independent, and also {\it exclusive} (see Definition \ref{D.5.0}).  This is, however, a difficult problem in general, and \cite{zlatos2019} was only able to show existence of a deterministic front speed in direction $e$ when the reaction has a Wulff shape with outer normal vector $e$ at some point (this is where the absence of corners is needed), because then the expanding Wulff shape can be used at large times to locally approximate a front-like solution propagating in direction $e$.

To show existence of deterministic front speeds, we apply a method modeled on the one employed by Armstrong and Cardaliaguet \cite{armstrong2015} in their proof of homogenization for Hamilton-Jacobi equations with $\alpha$-homogeneous (for $\alpha\ge 1$) non-convex (in $\nabla u$) Hamiltonians with finite ranges of dependence.  This was the first proof of stochastic Hamilton-Jacobi homogenization for non-convex Hamiltonians in several dimensions without special structural hypotheses (such as $H(x,\nabla u,\omega)=H(\nabla u)+V(x,\omega)$).  While there are many homogenization results for convex and level-set-convex Hamiltonians, including in the paper \cite{armstrong2014} by Armstrong, Cardaliaguet, and Souganidis where the method used in \cite{armstrong2015} originated, non-convexity of the Hamiltonian presents serious issues.  In fact, similarly to our reaction-diffusion setting, there are examples when homogenization  does not happen  for Hamilton-Jacobi equations with stationary ergodic non-convex  Hamiltonians \cite{FelSou, Ziliotto}, even in one dimension.
The approach in \cite{armstrong2015} overcomes these problems by leveraging the finite range of dependence hypothesis (which is akin to an i.i.d.~medium setting) and the resulting mixing properties of the environment to obtain strong quantitative estimates on the solutions where a soft approach via ergodic theorems does not appear to work.
%
These estimates involve
 fluctuations of the values of solutions to the so-called {\it metric problem} for any compact set $S\subseteq\bbR^d$  (which is an appropriate time-independent Hamilton-Jacobi PDE on $\bbR^d\setminus S$) with a smooth enough boundary.  These estimates improve at an exponential rate as the distance from $S$ increases, and were then upgraded to similar estimates for $S$ being any half-space.  
 
 Here we apply this strategy to reaction-diffusion equations, with the relevant estimates involving fluctuations of ``arrival times'' at any point $x\in\bbR^d$ for solutions initially approximating $\chi_S$ (we only need to consider $S=B_k(0)$ for any $k\in\bbN$).  We still obtain an exponentially-in-$d(x,S)$ decaying estimate (see Proposition \ref{P.3.4} below), albeit at a slower rate.  However, we are also able to extend it to some reactions with infinite ranges of dependence (see Proposition \ref{P.3.1}) by carefully tracking the dependence of this rate on the  range of dependence of $f$ when the latter is finite, something that was described in \cite{armstrong2015} as completely open in the Hamilton-Jacobi setting (and appears to remain such at this time)!  
 
 After we upgrade this estimate  from balls to half-spaces, we are able to  prove existence of deterministic exclusive front speeds in all directions, and thus homogenization after using results from \cite{zlatos2019}.
  We note that while the effective equations in Hamilton-Jacobi homogenization are still Hamilton-Jacobi PDE (although some of their terms can disappear in the homogenization process), and the limiting functions are their  continuous solutions, our limiting functions are {\it discontinuous} viscosity solutions to \eqref{1.5}, which causes extra difficulties in the analysis.  For a more thorough discussion of similarities and differences between Hamilton-Jacobi homogenization for non-convex Hamiltonians and reaction-diffusion homogenization, as well as for further references, we refer the reader to the introduction of \cite{zlatos2019}. 

\subsection{Hypotheses and Main Results}  
Let us now turn to our main results.  Our goal is to show that as $\eps\to 0$, solutions to \eqref{1.4} with initial data approximating $\chi_A$ for any open set $A\subseteq\bbR^d$ converge to the unique (discontinuous viscosity) solution to \eqref{1.5} with initial data $\chi_A$.  Here, of course, $c^*(e)$ are the deterministic front speeds discussed above, and establishing their existence forms the bulk of our work.

One can show that if $c^*:\bbS^{d-1}\to(0,\infty)$ is Lipschitz (which will be our case), then
for any open $A\subseteq\bbR^d$, there is an open set $\Theta^{A, c^*}\subseteq (0,\infty)\times\bbR^d$ such that  the unique 
solution to \eqref{1.5} with initial data $\chi_A$ is $\bar u:=\chi_{\Theta^{A, c^*}}$.  In fact, this set can also be found from the formula
\beq \lb{1.11}
\Theta^{A, c^*}:= \big\{ (t,x)\in(0,\infty)\times\bbR^d \,\big|\, v(t,x)>0 \big\},
\end{equation}
where $v_{0}:\bbR^d\to\bbR$  is any Lipschitz function satisfying $v_0>0$ on $A$ and $v_0<0$ on $\bbR^d\setminus  \overline A$, and $v$ is the unique (continuous) viscosity solution  to \eqref{1.5} with $v(0,\cdot)=v_{0}$.  The open set $\Theta^{A, c^{*}}$ is then independent of the choice of $v_0$ as above, and $\partial \Theta^{A, c^{*}}$ has zero measure.    

All these claims are contained in Theorem 5.3 in \cite{zlatos2019}, which is a
combination of results by Barles, Soner, and Souganidis \cite{13barles1993}, Crandall, Ishii, and Lions \cite{crandall1992user}, Souganidis \cite{47souganidis}, and Soravia \cite{soravia1994generalized}.  The reader can also consult Definition 5.1 in \cite{zlatos2019} for the definition of viscosity solutions to initial value problems for \eqref{1.5}.  

We also note that it was shown in the proof of Theorem 1.4(iii) in \cite{zlatos2019} that for any convex open $A\subseteq\bbR^d$ we have the explicit formula
\[ 
\Theta^{A, c^*}= \bigcap_{e\in\bbS^{d-1}} \left\{ (t,x)\in(0,\infty)\times\bbR^d\,\bigg|\, x\cdot e < \sup_{y\in\partial A} y\cdot e+c^*(e)t\,\right\}.
\]
In particular, if $A=\{x\in\bbR^d\,|\,x\cdot e<0\}$ is the half-space with outer normal $e$, then we obviously have $\Theta^{A, c^*}=\{(t,x)\in(0,\infty)\times\bbR^d\,|\,x\cdot e< c^*(e)t\}$.  This also shows that if we let $(\Theta^{A, c^*})_t$ be the spatial slice of $\Theta^{A, c^*}$ at the time $t>0$, then for any open bounded $A$ we have
\[
\lim_{t\to\infty} \frac { (\Theta^{A, c^*})_t}t =  \bigcap_{e\in\bbS^{d-1}} \left\{ y\in \bbR^d\,\big|\, y\cdot e < c^*(e)\,\right\}
\]
(e.g., in the sense of Hausdorff distances of boundaries of sets).  Hence the set on the right-hand side is the {\it Wulff shape} for \eqref{1.5}, and therefore also for \eqref{1.1} if homogenization holds.  

We will consider here stationary ignition reactions that either have finite ranges of dependence, or can be uniformly approximated by such reactions (see Example \ref{E.1.5} below for a simple example of the latter).  These properties are summarized in the following definition and in hypothesis \textbf{(H1)} below.


\begin{definition}\lb{D.1}
Consider a probability space $(\Omega,\calF,\bbP)$ that is endowed with a group of measure-preserving bijections $\{{\Upsilon_y:\Omega\to\Omega}\}_{y\in\bbR^d}$ such that for all $y,z\in \bbR^d$ we have
\[
\Upsilon_y\circ\Upsilon_z=\Upsilon_{y+z}.
\]
A reaction function $f:\bbR^d\times [0,1]\times\Omega\to [0,\infty)$, uniformly continuous in the first two arguments and with the random variables $X_{x,u} :=f(x,u,\cdot)$ being $\calF$-measurable for all $(x,u)\in \bbR^d\times [0,1]$, is called \textit{stationary} if for each $(x,y,u,\omega)\in \bbR^{2d}\times [0,1]\times\Omega$ we have
\[
f(x,u,\Upsilon_y\omega) = f(x+y,u,\omega).
\] 
The \textit{range of dependence} of such  $f$ is the infimum of all $r\in\bbR^+\cup\{\infty\}$ such that 
\[
\calE(U)\text{ and }\calE(V) \text{ are $\bbP$-independent}
\]
for any $U,V\subseteq \bbR^d$ with $d(U,V)\geq r$,
where $\calE(U)$ is the $\sigma$-algebra generated by 
the family of random variables 
$
\{X_{x,u}  \,|\, (x,u)\in U\times [0,1]\}.
$
\end{definition}

\noindent {\it Remark.}
While stationary reactions with finite ranges of dependence are also stationary ergodic, we will not need to use this property here due to our quantitative approach.  We note that although the main results in \cite{zlatos2019} apply to stationary ergodic reactions, that assumption is only needed to prove that all the deterministic (exclusive) front speeds for \eqref{1.1} exist and are strong (see Definition \ref{D.5.0} below),
which we instead  prove in Sections \ref{S3}--\ref{S6}.
\smallskip


We will consider here stationary reaction functions $f:\bbR^d\times [0,1]\times\Omega\to [0,\infty)$, and extend them to $\bbR^d\times \bbR\times\Omega$ by $0$ whenever we need to evaluate them with $u\notin[0,1]$.  Additionally, our reactions will be of the {\it ignition type}.  That is, we will assume the following hypothesis. 

\smallskip

\begin{itemize}
    \item[\textbf{(H1)}] The reaction  $f$ is stationary, Lipschitz in both $x$ and $u$ with constant ${M}\geq 1$, 
    and there are $\theta_1\in (0,\frac 12)$, $m_1>1$, and $\al_1>0$ such that
    $f(\cdot,u,\cdot)\equiv 0$ for $u\in [0,\theta_1]\cup\{1\}$, $f(\cdot,u,\cdot)\ge \al_1 (1-u)^{m_1}$ for $u\in [1-\theta_1,1)$, and  $f$ is non-increasing in $u\in [1-\theta_1,1)$.
%
%
%
\end{itemize}
 \smallskip
    




It is not difficult to see that one cannot hope for general reactions satisfying \textbf{(H1)} to lead to homogenization for  \eqref{1.1}, even if $f$ is independent of $(x,\omega)$  (see, e.g.,  \cite{ZlaInhomog, ZlaBist}).  Indeed, if $f$ is allowed to vanish at some intermediate value $\theta'\in(\theta_1,1-\theta_1)$ and is also sufficiently large for some $u\in(\theta_1,\theta')$, solutions could easily form ``plateaus'' with values near $\theta'$ (or another intermediate value) whose widths grow linearly in time.  And if that happens, the widths of these plateaus will not vanish even after the scaling from \eqref{1.3} is applied.

To avoid this scenario, one should assume that as the argument $u$ grows from 0 to 1 (for any fixed $(x,\omega)$),  the reaction $f$ cannot become arbitrarily small  (except near $u=1$) once it has become large enough.  This is expressed in Definition \ref{D.1.2} below, which was used in \cite{ZlaInhomog} to show that not only solutions to \eqref{1.1} do not develop such plateaus, but the transition from values $u\sim 0$ to values $u\sim 1$ in fact occurs over uniformly-in-time bounded distances in space (see, e.g.,  Lemma \ref{T.2.4} below).  Our most general results apply in this setting, as well as when one instead only assumes at most $O(t^\alpha)$ growth of the above transition distances, with $\alpha<1$
(see hypothesis \textbf{(H2')} below).

However,  for the sake of simplicity, in our first two results we will consider the case where the reaction does not become arbitrarily small (except near $u=1$) after it has become just {\it positive}.  That is, once $u$ has exceeded the {\it ignition temperature} 
\[
\theta_{x,\omega} := \sup \{\theta\ge 0 \,|\, f(x,u,\omega)=0 \text{ for all $u\in[0,\theta]$}\} \qquad (\in [\theta_1,1-\theta_1)).
\]
Of course, this is the case for any realistic model of combustion, where the reaction rate is positive at all temperatures above the ignition temperature (its vanishing at $u=1$ is due to fuel exhaustion in systems of equations for temperature and concentration of the reactant, which  in certain regimes simplify to \eqref{1.1} with $f(\cdot,1,\cdot)\equiv 0$).

\begin{definition}\lb{D.2}
A reaction $f$ satisfying  \textbf{(H1)} is a stationary {\it pure ignition} reaction if for each $\eta>0$ we have
\[ 
\inf_{\substack{(x,\omega)\in\bbR^d\times\Omega \\ \theta_{x,\omega} + \eta < 1-\theta_1 }} f(x,\theta_{x,\omega}+\eta,\omega)>0.
\]
\end{definition}

\noindent {\it Remark.}  This definition (with the bound for $u\in[1-\theta_1,1)$ being $\inf_{(x,\omega)} f(x,u,\omega)>0$) is from \cite{ZlaBist}.  Note that it is trivially satisfied, for instance, when  $f(x,u,\omega)=g(x,\omega)F_0(u)$, with $g$ bounded away from 0 and $\infty$, Lipschitz in $x$, and stationary in $\omega$, and with Lipschitz $F_0:[0,1]\to[0,\infty)$   such that $F_0=0$ on $[0,\theta_0]\cup \{1\}$ and $F_0>0$ on $(\theta_0,1)$ for some $\theta_0\in(0,1)$, and $F_0$ is non-increasing and bounded below by $\al_1 (1-u)^{m_1}$ near 1 (for some $m_1,\al_1$).
\smallskip


We will therefore start by assuming the following hypothesis.

\smallskip

\begin{itemize}
    \item[\textbf{(H2)}]  $f$ is a stationary pure ignition  reaction and $d\leq 3$.

\end{itemize}

\smallskip


The additional restriction $d\le 3$ is necessitated by the above-mentioned surprising result from \cite{ZlaInhomog}, where the second author showed that even for pure ignition reactions, transition from values $u= \eta$ to values $u= 1-\eta$  may only occur over linearly-in-time growing distances for solutions to \eqref{1.1} and all small $\eta>0$  in dimensions $d\ge 4$ (while these distances remain bounded in dimensions $d\le 3$).  

%

We are now ready to state our first main homogenization result. In it and later  we use the notation $B_r(A):=A+(B_r(0)\cup\{0\})$ and $A^0_r:=A\backslash\overline{B_r(\partial A)}$ for $A\subseteq\bbR^d$ and $r\ge 0$.  For the sake of generality, we also allow $O(1)$ shifts  and $o(1)$ errors in initial data as $\eps\to 0$ in \eqref{1.4}.

\begin{theorem}\lb{T.1.1}
If $f$ satisfying \textbf{(H2)} has a finite range of dependence, then there is Lipschitz $c^*:\bbS^{d-1}\to (0,\infty)$ such that the following holds for any open $A\subseteq\bbR^d$ and $\Theta^{A, c^*}$ from \eqref{1.11}. 
If $\Lambda>0$, and for all $\omega\in\Omega$ and $\eps>0$, the function $u_\eps(\cdot,\cdot,\omega)$ solves \eqref{1.4} and satisfies
\beq \lb{1.77}
(1-\theta_1)\chi_{A^0_{\psi(\eps)}}\leq {u_\eps}(0,\cdot +y_\eps,\omega)\leq \chi_{B_{\psi(\eps)}(A)}+\psi(\eps)\chi_{\bbR^d\backslash B_{\psi(\eps)}(A)}
\eeq
for some $y_\eps\in B_{\Lambda}(0)$ and some $\psi$ with $\lim_{\eps\to 0} \psi(\eps)=0$ (when $y_\eps=0$ and $\psi(\eps)= 0$, this becomes just $(1-\theta_1)\chi_{A}\le {u_\eps}(0,\cdot,\omega) \le \chi_{A}$), then for almost all $\omega\in\Omega$ we have
\[
\lim_{\eps\to 0}u_\eps(\cdot,\cdot+y_\eps,\omega)= \chi_{\Theta^{A, c^*}}
\]
locally uniformly on  $([0,\infty)\times\bbR^d)\setminus \partial \Theta^{A, c^*}$.
\end{theorem}


\noindent {\it Remark.}  Our proofs use results from \cite{zlatos2019} which in fact show that in all our main results, $1-\theta_1$ in \eqref{1.77} can be replaced by any $\theta$ satisfying $\inf_{(x,u,\omega)\in\bbR^d\times[\theta,1-\theta_1]\times\Omega} f(x,u,\omega)>0$.
\smallskip

We next extend this  to the case of reactions with infinite ranges of dependence that are  uniform limits of reactions with finite ranges of dependence.  Here we will also require some uniform decay of $f$ near $u=1$.  This is the content of the next two hypotheses.

\smallskip

\begin{itemize}
\item[\textbf{(H3)}] There are $m_3\geq 1$ and $\al_3>0$ such that for all $\eta\in (0,\frac 12\theta_1]$ we have 
\[
\inf_{\substack{
(x,\omega)\in\bbR^d\times\Omega\\
    u\in [1- \theta_1/2,1]}}\left(f(x,u-\eta,\omega)-f(x,u,\omega)\right)\geq \al_3 \eta^{m_3}.
    \]
\end{itemize} 

\smallskip



  
\smallskip

\begin{itemize}
\item[\textbf{(H4)}] 
There are $m_4,n_4,\al_4>0$ such that for each $n\geq n_4$, there exists a stationary reaction $f_n$
with range of dependence $\le n$ and $\|f_n-f\|_\infty \leq \al_4 {n^{-m_4}}$.
%
%
\end{itemize}


\begin{theorem}\lb{T.1.2}
Theorem \ref{T.1.1} holds for any $f$ satisfying \textbf{(H2)--(H4)}.
\end{theorem}

While this result does not cover all interesting  pure ignition reactions in dimensions $d\le 3$ with  correlations of $f(x,u,\cdot)$ and $f(y,v,\cdot)$ decreasing as $|x-y|\to\infty$ (for all $u,v\in[0,1]$), it does apply to many of them.  
Here is a simple such example.

\begin{example} \lb{E.1.5}
Let $d\le 3$, 
 $F_0:[0,1]\to[0,\infty)$  be Lipschitz with $F_0=0$ on $[0,\theta_0]\cup \{1\}$ and $F_0>0$ on $(\theta_0,1)$ for some $\theta_0\in(0,1)$, and $F_0'(u)\le -(1-u)^m$ near $u=1$ for some $m$.  
Also pick some Lipschitz $g:\bbR^d\to [0,\infty)$ 
with $\sup_{x\in\bbR^d} |x|^{m'}g(x)<\infty$ for some $m'>0$,
and some Lebesgue measurable  $a:[0,1]\to[0,1]$.  
Consider the product probability space $\Omega=[0,1]^{\bbZ^d}$, with the Lebesgue measure on each copy of $[0,1]$, and
for any $k\in\bbZ^d$, denote by $\omega_k$ the $k^{\rm th}$ coordinate of $\omega\in \Omega$ (note that these are i.i.d. random variables).  Let $\Upsilon_y:\Omega\to\Omega$ for $y\in\bbZ^d$ be given by $(\Upsilon_y\omega)_k:=\omega_{y+k}$ for all $k\in\bbZ^d$.    Then 
\[
f(x,u,\omega):= \left( 1+\sup_{k\in\bbZ^d} a(\omega_k) g(x-k) \right) F_0(u)
\]
 satisfies \textbf{(H2)}--\textbf{(H4)} (see next paragraph for stationarity), with $f_n$ defined as $f$ but with $g$ replaced by $g_n(x):=g(x)\min\{1,d(x,\bbR^d\setminus B_{n/2}(0))\}$.
 Hence Theorem~\ref{T.1.2} applies.  
Note that $f$ may have infinite range of dependence when $g$ is not compactly supported.

Note also that while this  $f$ is  stationary only with respect to integer shifts (i.e., $y\in\bbZ^d$ in Definition \ref{D.1}),
such settings can be easily transformed to the case considered in the present paper by letting $\tilde\Omega:=\Omega\times[0,1)^d$ with the product measure, $\tilde f(x,u,(\omega,z)):= f(x+z,u,\omega)$, and $\tilde \Upsilon_y(\omega,z):=(\Upsilon_{\lfloor y+z \rfloor}\omega, \{y+z\})$ for $y\in\bbR^d$.
Since inclusion of $y_\eps$ in \eqref{1.77} shows that all our main results continue to hold if we replace the identified full-measure set $\tilde \Omega'\subseteq\tilde \Omega$ by $\bigcup_{y\in\bbR^d} \tilde \Upsilon_y\tilde \Omega'$, which is of the form $\Omega'\times[0,1)^d$, they then also apply in integer-shift settings.
\end{example}

In the above example and in Theorem \ref{T.1.2}, reactions $f$ with infinite ranges of dependence are uniform limits of those with finite ranges of dependence.  The next example is a natural situation when this need not be the case (it is an analog of the setting where sticks of random unbounded lengths are randomly positioned in  $\bbR^d$).  While Theorem \ref{T.1.2} does not apply here, one can instead use its generalization, Theorem \ref{T.1.3} below, which allows this.

\begin{example} \lb{E.1.5'}
Consider the setting from Example \ref{E.1.5}, without the functions $g$ and $a$.  Instead pick some uniformly bounded and uniformly Lipschitz   $g_j:\bbR^d\to [0,\infty)$ ($j\in\bbN$) that
satisfy $\sup_{j\in\bbN} \sup_{|x|> j} |x|^{m'}g_j(x)<\infty$ for some $m'>0$,
and some Lebesgue measurable $a:[0,1]\to\bbN$ with $|a^{-1}(j)|\le j^{-\gamma}$ for some $\gamma>3d+2$ and all $j\in\bbN$.  Then 
\[
f(x,u,\omega):= \left( 1+\sup_{k\in\bbZ^d}  g_{a(\omega_k)}(x-k) \right) F_0(u)
\]
 satisfies \textbf{(H3)} and \textbf{(H4')} below (see Example \ref{E.1.5} for stationarity), with $f_n$ defined as $f$ but with $g_j$ replaced by $g_{j,n}(x):=g_j(x) \min\{1,2d(x,\bbR^d\setminus B_{n/2}(0))\}$.  This uses the fact that
 \[
 \sum_{j\ge n/2} (n^{1+m_4'} +j)^d |a^{-1}(j)| = o(n^{-(2d+1+m_4')})
 \]  
 for some $m_4'>0$ (because $\gamma>3d+2$), with the left-hand side being (up to a constant factor) an upper bound on the probability that $a(\omega_k)\ge \max\{\frac n2, |k|-n^{1+m_4'}\}$ for at least one  $k\in\bbZ^d$.
 Hence Theorem~\ref{T.1.3} below applies.  
 Note that $f$ need not be a uniform limit of reactions with finite ranges of dependence when the functions $g_j$ do not decay uniformly to 0 as $x\to\infty$.
\end{example}

\subsection{Generalizations}
As we indicated above, it is not clear whether the limitation on the dimension in \textbf{(H2)} is necessary to obtain a sufficiently general result.  However, since both conditions in \textbf{(H2)}  are only needed to guarantee certain estimates for some special solutions to \eqref{1.1}  (see Lemma \ref{T.2.4} below), including that  the transition from values $u\sim 0$ to values $u\sim 1$   occurs over spatial distances that grow only sub-linearly in time, as we mentioned above (Lemma \ref{T.2.4} shows that in the case of \textbf{(H2)} these distances are in fact uniformly bounded), we can extend our results to more general settings as long as these estimates still hold there.  
In particular, 
this might be the case for  stationary ignition reactions in dimensions $d\ge 4$.

In order to state this alternative to hypothesis \textbf{(H2)}, let us define for any $0<\eta<\theta<1$ the {\it width of the transition zone} from $\eta$ to $\theta$  for a  solution $u:[0,\infty)\times \bbR^d\to [0,1]$ to \eqref{1.1} at some time $t\ge 0$ to be (see  \cite{ZlaInhomog})
\beq\lb{d.2.1}
L_{u,\eta,\theta}(t):=\inf \Big\{L>0\, \Big| \,\{x\in\bbR^d \,|\, u(t,x)\geq\eta\}\subseteq B_L\left(\left\{x\in\bbR^d \,|\, u(t,x)\geq \theta\right\} \right) \Big\}.
\eeq
The special solutions for which we need to assume certain bounds on these quantities will be essentially those evolving from  characteristic functions of 
the balls $B_k(0)\subseteq \bbR^d$, with $k\in\bbN$.  

It will however be more convenient to work with approximations $u_{0,k}$ of these characteristic functions that have two useful properties.  First, they are close to 1 on $B_k(0)$ but are strictly below 1 (which will allow us to treat general initial data from \eqref{1.77}), and are supported on $B_{k+R_0}(0)$ for some fixed $R_0$.  Specifically, we will require that 
\beq\lb{2.7'}
(1-\theta^*)\chi_{B_k(0)}\leq u_{0,k}\leq (1-\theta^*)\chi_{B_{k+R_0}(0)}
\eeq
holds 
with $\theta^*>0$ from \eqref{2.10} below.  We note that we could in fact replace $1-\theta^*$ in \eqref{2.7'} by 
any $\theta<1$ satisfying $\inf_{(x,u,\omega)\in\bbR^d\times[\theta,1-\theta_1]\times\Omega} f(x,u,\omega)>0$,
but we make our choice for the sake of convenience (Lemma~\ref{L.2.2} shows that  solutions $u:(0,\infty)\times\bbR^d\to[0,1]$ with $u(t,x)\ge 1-\theta^*$ for some $(t,x)\in [1,\infty)\times\bbR^d$ converge locally uniformly to 1).
The second property is that the corresponding solutions to \eqref{1.1} satisfy $u_t>0$.  For this, it suffices to have
\beq\lb{2.6}
\Delta u_{0,k}+{F}(u_{0,k})\geq 0
\eeq
with $F(u):=\inf_{(x,\omega)\in\bbR^d\times\Omega} f(x,u,\omega)$,
which yields
$\Delta u_{0,k}+f(\cdot, u_{0,k},\omega)\geq 0$
 for any $\omega\in\Omega$.  Then $u_t> 0$ follows for the corresponding solution $u$ at all positive times because $v:=u_t$ solves the linear equation $v_t=\Delta v+f_u(x,u(t,x),\omega)v$ with $v(0,\cdot)\geq 0$ and $v(0,\cdot)\not \equiv 0$ (due to \eqref{2.7'} and $F(1-\theta^*)>0$). 
 
It is  easy to construct radial functions satisfying \eqref{2.7'} and \eqref{2.6}, since then \eqref{2.6} becomes a simple ordinary differential inequality.
(This is in fact possible for any set $S\subseteq\bbR^d$, without radial symmetry but still with a uniform $R_0$, and we do so in Lemma \ref{L.2.4} below.)
Let us now pick one such $u_{0,k}$ for each $k\in\bbN$ (any one can be chosen), and denote by $\calU_f$ the set of all solutions $u$ to \eqref{1.1} obtained by choosing any $\omega\in\Omega$ and initial data $u(0,\cdot)=u_{0,k}$ for any $k\in\bbN$.
We can now replace \textbf{(H2)} by the following hypothesis.

\smallskip

\begin{itemize}
    \item[\textbf{(H2')}] $f$  satisfies \textbf{(H1)} and there are $\al_2<1$ and ${m_2}>0$ such that 
\beq\lb{2.2}
\begin{aligned}
&\limsup_{t\to\infty}\sup_{u\in\calU_{f}}\sup_{\eta>0} \,\frac{L_{u,\eta,1-\theta^*}(t)}{t^{\al_2} \eta^{-{m_2}}}  <\infty ,\\
&\liminf_{t\to\infty} \inf_{u\in \calU_{f}} \inf_{ u(t,x)\in [\theta^*,1-\theta^*]} \, u_t(t,x)  >0.
\end{aligned}
\eeq
Here $\calU_f$ is as above, with some $u_{0,k}$ satisfying \eqref{2.7'} and \eqref{2.6} for each $k\in\bbN$, and $\theta^*=\theta^*(M,\theta_1,m_1,\al_1)$ from \eqref{2.10} and $R_0=R_0(M,\theta_1,m_1,\al_1)$ are independent of $k$.
\end{itemize}
 
\smallskip

\noindent {\it Remarks.} 
1. The first statement in  \eqref{2.2} allows $L_{u,\eta,1-\theta^*}(t)$ to grow algebraically in both $\eta\to0$ and $t\to\infty$ (note that $\alpha_2<1$ is critical here because the scaling from \eqref{1.3} yields $L_{u_{\eps},\eta,1-\theta^*}(t)=\eps L_{u,\eta,1-\theta^*}\left(\eps^{-1} {t}\right)$, which will then vanish on any bounded time interval as we take $\eps\to 0$).  We note that Lemma~\ref{T.2.4} below shows that in the case of \textbf{(H2)}, the former growth is only logarithmic while the latter is non-existent.  
\smallskip

2.  Lemma \ref{T.2.4} shows that the second statement in \eqref{2.2}  holds as well if one assumes \textbf{(H2)}  (recall also that all $u\in \calU_f$ satisfy $u_t> 0$).  Nevertheless, we will further weaken this hypothesis in Theorem \ref{T.1.7} below.
\smallskip

3.  We could also replace $B_k(0)$ and $B_{k+R_0}(0)$ in \eqref{2.7'} by $B_{r_k}(0)$ and $B_{r_k+R_0}(0)$ for any sequence $r_k\to\infty$, without any change to our results.
\smallskip

%
%
%
%
%

After replacing \textbf{(H2)} by \textbf{(H2')}, we must also adjust \textbf{(H4)} in the extension of Theorem~\ref{T.1.2}, in order to ensure that the reactions $f_n$ will satisfy \textbf{(H2')} with uniform constants.  Note that when \textbf{(H2)} holds, we will show in Corollary \ref{C.2.5} that $f_n$ from \textbf{(H4)} can be perturbed so that this is the case, but we do not know whether this remains true when we only assume \textbf{(H2')}. 

In addition, we also state this new version of \textbf{(H4)} so that it applies to some $f$ that are not uniform limits of reactions with finite ranges of dependence (see Example \ref{E.1.5'} above).

\smallskip

\begin{itemize}
\item[\textbf{(H4')}] 
There are $m_4,m_4',n_4,\al_4>0$ such that for each $n\geq n_4$, there exists a stationary reaction $f_n$
with range of dependence $\le n$ and 
\[
\bbP\left( \sup_{|x|< n^{1+m_4'}} \, \sup_{u\in [0,1]} |f_n(x,u,\cdot) - f(x,u,\cdot) |  > \al_4 {n^{-m_4}}\right) \leq n^{-(2d+1+m_4')}.
\]
Moreover, \textbf{(H2')} holds uniformly in $n$
(i.e., reactions $f_n$ satisfy \textbf{(H1)} with the same $M,\theta_1,m_1,\al_1$, and \eqref{2.2} with $\calU_{f}$  replaced by $\bigcup_{n\ge n_4}\calU_{f_n}$).
\end{itemize}
\smallskip

We note that the initial data $u_{0,k}$ used in the definition of $\calU_{f_n}$ are in principle  allowed to be different for distinct $n$ (but $\theta^*$ and $R_0$ are uniform in $n$; also $r_k$ in Remark 3 above).

\begin{theorem}\lb{T.1.3}
Theorem \ref{T.1.1} holds for any $f$ that either satisfies \textbf{(H2')} and has a finite range of dependence, or satisfies \textbf{(H3)} and  \textbf{(H4')}.
\end{theorem}

Finally, we show that one can also allow a power decay in time in the second statement in \eqref{2.2}, at the expense of either having to extend this assumption to a slightly larger family of special solutions or obtaining the result for a smaller family of initial data.  

For each $a\in [0,\frac 12\theta^*]$, let $\calU_{f,a}$ be defined as $\calU_f$ above, but with \eqref{2.7'} replaced by
\beq\lb{7.3}
(1-a)(1- {\theta^*} ) \chi_{B_k(0)}+a\leq u_{0,k,a}\leq (1-a)(1- {\theta^*} ) \chi_{B_{k+R_0}(0)}+a
\eeq
for initial data denoted $u_{0,k,a}$ instead of $u_{0,k}$ (so now $u_{0,k,a}-a$ is supported in $B_{k+R_0}(0)$).
Obviously $\calU_{f,0}=\calU_f$, and one can find such initial data (for any $S\subseteq\bbR^d$ and with $R_0$ uniform in $a$) via Lemma~\ref{L.2.4} 
with $\alpha_1$ replaced by $\alpha_1(1-\frac18{\theta_1})^{m_1-1}$
(since $a\le\frac 18\theta_1$) and then applying the scaling $u_{0,S,a}:=(1-a)u_{0,S}+a$.

We can now replace \textbf{(H2')} and \textbf{(H4')} by the following hypotheses.

\smallskip

\begin{itemize}
    \item[\textbf{(H2'')}] $f$  satisfies \textbf{(H1)} and there are $\al_2<1$, $m_2> 0$, $a_2\in[0,\frac 12\theta^*]$, and $\alpha_2'<\min\{ \frac 1{m_1-1}, \frac {1-\alpha_2}{m_2} \}$
 such that
\beq \lb{7.000}
\begin{aligned}
&\limsup_{t\to\infty} \sup_{a\in[0,a_2]} \sup_{u\in\calU_{f,a}} \sup_{\eta>0}\frac{L_{u,\eta+a,1-\theta^*}(t)}{t^{\al_2} \eta^{-{m_2}}}<\infty, \\
&    \liminf_{t\to\infty} \inf_{a\in[0,a_2]} \inf_{u\in\calU_{f,a}} \inf_{ u(t,x)\in [\theta^*,1-\theta^*]} \, u_t(t,x)\,t^{\al_2'}>0.
\end{aligned}
\eeq
Here $\calU_{f,a}$ is as above, with some $u_{0,k,a}$ satisfying \eqref{7.3} and \eqref{2.6} for each $k\in\bbN$, and $\theta^*=\theta^*(M,\theta_1,m_1,\al_1)$ from \eqref{2.10} and $R_0=R_0(M,\theta_1,m_1,\al_1)$ are independent of $k$.
\end{itemize}

\smallskip

%

\begin{itemize}
    \item[\textbf{(H4'')}] 
$f$ satisfies \textbf{(H4')}  with \textbf{(H2'')} in place of \textbf{(H2')},
and also $\alpha_2'<\frac{m_4}{m_3}$. 
\end{itemize}

\smallskip

Of course, these hypotheses coincide with \textbf{(H2')} and \textbf{(H4')} when $a_2=0=\al_2'$.  
With them, we can now state our second generalization of Theorems \ref{T.1.1} and \ref{T.1.2}.

\begin{theorem}\lb{T.1.7}
Assume that $f$ either satisfies \textbf{(H2'')}  and has a finite range of dependence, or satisfies \textbf{(H3)} and  \textbf{(H4'')}.

(i) If $a_2>0$, then Theorem \ref{T.1.1} holds for such $f$.

(ii) If $a_2=0$, then Theorem \ref{T.1.1} holds for such $f$ with $A$ convex
 and \eqref{1.77} replaced by
\[
(1-\theta_1)\chi_{A^0_{\psi(\eps)}}\leq {u_\eps}(0,\cdot+y_\eps,\omega)\leq (1-\Lambda^{-1})\chi_{B_{\psi(\eps)}(A)}\,.
\]
\end{theorem}

%

\subsection{Organization of the Paper and Acknowledgements}
In Section \ref{S2} we collect most important notation and prove several preliminary results.  These include Corollary \ref{C.2.5}, which shows that Theorems \ref{T.1.1} and \ref{T.1.2} follow from Theorem \ref{T.1.3}.  It will therefore suffice to prove Theorems \ref{T.1.3} and \ref{T.1.7}.  We prove the first one in Section \ref{S6}, after obtaining crucial quantitative  estimates on long-time dynamics of solutions to \eqref{1.1} in Sections \ref{S3}--\ref{S5} (specifically, Propositions \ref{P.3.4}, \ref{P.3.1}, and \ref{P.4.5}, with the first two of these being essentially the same result but assuming \textbf{(H2')}+finite range  in the first and \textbf{(H3)}+\textbf{(H4')} in the second).  In Section~\ref{S7} we then show how to extend all these results to the cases considered in Theorem \ref{T.1.7}.

The authors thank Scott Armstrong and Jessica Lin for illuminating discussions.
AZ also acknowledges partial support by  NSF grants DMS-1652284 and DMS-1900943.

\section{Preliminaries and Notation}\lb{S2}

In this section we collect some previous results and preliminary lemmas, all of which hold uniformly in $\omega$ and without needing to assume stationarity of the reaction.  We will therefore use the following hypothesis.
\smallskip

\begin{itemize}
\item[\textbf{(H1')}] 
$f$ satisfies \textbf{(H1)} except possibly the stationarity hypothesis.
\end{itemize}
\smallskip

  At the end of the section we also collect all the important notations in one place.

 Let us start with a basic lower bound (see, e.g., \cite[Lemma~3.1]{ZlaInhomog}), which shows that general solutions to \eqref{1.1} propagate with speed no less than some $c_0=c_0(M,\theta_1,m_1,\al_1)>0$.  We will choose this to be the {\it unique front speed} for the homogeneous reaction $F_0:[0,1]\to[0,\infty)$ defined to be the largest $M$-Lipschitz function with $F_0(u)\le\al_1(1-u)^{m_1}\chi_{[1-\theta_1,1]}(u)$ (so clearly $F_0\le F$).
Hence $c_0$ is the unique number such that the PDE $u_t=u_{xx} + F_0(u)$ in one space dimension has a {\it traveling front} solution $u(t,x)=U(x-c_0t)$ with $U(-\infty)=1$ and $U(\infty)=0$.  

\begin{lemma}\lb{L.2.2}
There exists $\theta_2=\theta_2({M},\theta_1,m_1,\al_1)<1$ 
such that for each $c<c_0$ and $\theta<1$, there is $\kappa_0=\kappa_0({M},\theta_1,m_1,\al_1,c,\theta)\geq 1$ such that the following holds. If $u:(0,\infty)\times \bbR^d\to [0,1]$ is a solution to \eqref{1.1} with $f$ satisfying \textbf{(H1')} 
and with some $\omega\in\Omega$, 
and if $u(t_0,y)\geq \theta_2$ for some $t_0\ge 1$ and $y\in\bbR^d$, then for all $t\geq t_0+\kappa_0$, 
\[
\inf_{|x-y|\leq c(t-t_0) }u(t,x) \geq \theta.
\]
If also $u_t\ge 0$, then this clearly holds with any $t_0\ge 0$ (and $\kappa_0$ increased by 1).
\end{lemma}


Let now 
\beq\lb{2.10}
\theta^*:=\frac{1}{4}\min\{1-\theta_2,\,\theta_1\},
\eeq
where $\theta_2=\theta_2({M},\frac 12\theta_1,m_1,\al_1(1-\frac18{\theta_1})^{m_1-1})<1$.

\smallskip
\noindent {\it Remark.} Addition of the factors $\frac 12$ and  $(1-\frac18{\theta_1})^{m_1-1}$ here is due to the scaling $u\mapsto (1-a)u+a$ mentioned before \textbf{(H2'')}, as we shall see in Section \ref{S7}.  All arguments before Proposition \ref{P.7.6} 
will only require $\theta_2=\theta_2({M},\theta_1,m_1,\al_1)$ here, and also only that $\theta^*\le \frac{1}{2}\min\{1-\theta_2,\,\theta_1\}$.  So we could define $\theta^*$ this way in Theorem \ref{T.1.3} and its proof.  
\smallskip

In the rest of the paper we will primarily use Lemma \ref{L.2.2} with $c=\frac{c_0}{2}$ and $\theta=1-\theta^*$, 
and we will therefore define
\beq\lb{2.10a}
\kappa_0:=\kappa_0 \left({M},\theta_1,m_1,\al_1, \frac {c_0}2,1-\theta^* \right).
\eeq


Having defined this $\theta^*$, let us next  construct the initial data $u_{0,S}$ from the introduction, which are perturbations of the functions $\chi_S$ that also satisfy \eqref{2.6}.

\begin{lemma}\lb{L.2.4}
There is $R_0=R_0(M,\theta_1,m_1,\al_1)\ge 1$ such that for any $f$ satisfying \textbf{(H1')}  and $S\subseteq \bbR^d$, there is a smooth function $u_{0,S}$ satisfying \eqref{2.6} and
\beq\lb{2.7}
(1-\theta^*)\chi_{S}\leq u_{0,S}\leq (1-\theta^*)\chi_{B_{R_0}(S)}.
\eeq
\end{lemma}

\noindent {\it Remark.}
Recall that \eqref{2.6} implies that the relevant solutions to \eqref{1.1} satisfy $u_t>0$.  Moreover, \eqref{2.7} yields the uniform bound $L_{u,\eta,1-\theta^*}(0)\leq R_0$ for all $\eta\in(0,1-\theta^*)$ and $S\subseteq \bbR^d$, which is relevant for the next result.  
\smallskip

Lemma \ref{L.2.4} is proved in Appendix A.

Let us now turn to the consequences of \textbf{(H2)} obtained in \cite{ZlaInhomog}.  In fact, these results hold for the following more general classes of functions.

\begin{definition} \lb{D.1.2}
For $M,\theta_1,m_1,\al_1$ from \textbf{(H1)} (and $F_0$ defined above), and for any $\zeta,\xi>0$, let $\calF(F_0,M,\theta_1,\zeta,\xi)$ be the class of all $f$ satisfying \textbf{(H1')} such that
 \[
 \inf_{\substack{{(x,\omega)\in \bbR^d\times\Omega}\\ u\in [\gamma_f(x,\omega;\zeta),1-\theta_1]}} f(x,u,\omega)\geq\xi,
 \]
    where  (with the convention $\inf\emptyset=\infty$)
\[
    \gamma_f(x,\omega;\zeta):=\inf\{u\geq 0\, | \, f(x,u,\omega)>\zeta u\}.
\]
\end{definition}

\noindent {\it Remarks.} 1.  Although we could instead write $\calF(M,\theta_1,m_1,\al_1,\zeta,\xi)$, we use notation  from \cite{ZlaInhomog}.\smallskip

2. Note that pure ignition reactions  belong to $\bigcup_{\xi>0} \calF(F_0,M,\theta_1,\zeta,\xi)$ for any $\zeta>0$.\smallskip

It was shown in  \cite{ZlaInhomog} that if $d\le 3$ and $f$ is from the class $\calF(F_0,M,\theta_1,\zeta,\xi)$ for some $F_0,M,\theta_1,\xi$ and $\zeta<c_0^2/4$, then transitions from values $u\sim 0$ to values $u\sim 1$ for fairly general solutions to \eqref{1.1}  occur over uniformly-in-time bounded distances.  In view of our interest in solutions from $\calU_f$,  with initial data satisfying Lemma \ref{L.2.4}, the following result will be relevant. 

\begin{lemma}\lb{T.2.4}
Let $d\le 3$, let $F_0,M,\theta_1$ be as in Definition \ref{D.1.2} and $\theta^*$ from \eqref{2.10}, and consider any $\xi>0$ and $\zeta<c_0^2/4$.
There is $\Lambda>0$
and for any $\eta\in(0,\frac 12)$ there are $\mu_\eta,\kappa_\eta>0$
such that if $f\in\calF(F_0,M,\theta_1,\zeta,\xi)$ and  
$u$ solves \eqref{1.1} with some $\omega\in\Omega$ and initial data satisfying Lemma~\ref{L.2.4} for some $S\subseteq\bbR^d$, then
\beq\lb{2.5}
\sup_{t\geq0 \,\&\, \eta\in (0,1-\theta^*)} \frac{ L_{u,\eta,1-\theta^*}(t) } { 1+ |\ln \eta| } \le \Lambda
\eeq
and for any $\eta>0$ we have
\beq \lb{2.5a}
    \inf_{\substack{(t,x)\in(\kappa_\eta,\infty)\times \bbR^d\\
    u(t,x)\in [\eta,1-\eta]}}u_t(t,x)\geq \mu_\eta.
\eeq
\end{lemma}

\begin{proof}
Recall that we have $L_{u,\eta,1-\theta^*}(0)\leq R_0$ for all 
$\eta\in(0,1-\theta^*)$.  We then obtain \eqref{2.5a} from \cite[Theorem 2.5(i)]{ZlaInhomog} with $(\eps',\eps)=(1-\theta^*,\eta)$ (we can choose there $\eps_0=\theta^*$, and then $\eps'=1-\eps_0$), with  independence of $\kappa_\eta$ on 
$u$ due to Remark 1 after the theorem.

To obtain \eqref{2.5}, we instead use (4.14) in \cite{ZlaInhomog} with $(\eps_0,h,t_0)=(\theta^*,0,0)$ (where $\eps_0$ was used to define the function $Z_y$ in (4.14)).  Then the bound $Z_y(t_0)-Y_y^h(t_0)\le R_0$ for all $y\in\bbR^d$ follows from $L_{u,\eta,1-\theta^*}(0)\leq R_0$ for all  $\eta\in(0,1-\theta^*)$, so (4.14) yields $Z_y(t)-Y_y^h(t)\le \lambda$ for all $(t,y)\in[0,\infty)\times \bbR^d$ and some $\lambda=\lambda(F_0,M,\theta_1,\zeta,\xi)$.  The definition of $Y_y^h(t)$ shows that if $u(t,y)\ge \eta$ for some $(t,y)$, then $Y_y^h(t)\le \psi^{-1}(\frac 1\eta)$, with $\psi''(r)+\frac{d-1}r \psi'(r)=\zeta\psi(r)$ on $(0,\infty)$ and $\psi(0)=1$ (hence we have $\lim_{r\to\infty} r^{(d-1)/2} e^{-\sqrt\zeta r} \psi(r)\in(0,\infty)$).  Therefore $\psi(r)\ge e^{\sqrt\zeta r/2}$ for all large enough $r$, so $Z_y(t)\le \lambda + \frac 2{\sqrt\zeta} |\ln \eta|$ for all small enough $\eta$.  But since $Z_y(t)$ is the distance from $y$ to the nearest point $x$ with $u(t,x)\ge 1-\eps_0=1-\theta^*$, we obtain $L_{u,\eta,1-\theta^*}(t)\le \lambda + \frac 2{\sqrt\zeta} |\ln \eta|$ for all 
$t\ge 0$ and all small enough $\eta>0$.  This yields \eqref{2.5} with some $\Lambda=\Lambda(F_0,M,\theta_1,\zeta,\xi)$.
%
%
%
\end{proof}

The next result is a counterpart of Lemma \ref{L.2.2} (see \cite[Lemma 2.2]{zlatos2019}).  It shows that the speed of propagation of perturbations of solutions to \eqref{1.1} is bounded above by $2 \sqrt{Md\,}$ (in fact, the bound $2\sqrt M$ works as well, but we will not need it here).

\begin{lemma}\lb{L.2.1}
Let  $r>0$ and $y\in\bbR^d$, and let $u_1:[0,\infty)\times B_r(y)\to (-\infty,1]$ be a subsolution and $u_2:[0,\infty)\times B_r(y)\to [0,\infty)$ a supersolution to \eqref{1.1} with some $f$ satisfying \textbf{(H1')} and some $\omega\in\Omega$.  If $u_1(0,\cdot)\leq u_2(0,\cdot)$ on $B_r(y)$, then for all $(t,x)\in[0,\infty)\times B_r(y)$ we have
\[
u_1(t,x)\leq u_2(t,x)+2d\, e^{ \sqrt{M/{d}\,}\left(|x-y|-r+2 \sqrt{Md\,}\,t\right)}.
\]
\end{lemma}

\noindent {\it Remark.}  This was stated in \cite{zlatos2019} with $u_1,u_2$ having values in $[0,1]$ only, but the proof applies to our case without change.
\smallskip

This yields the following corollary (as above, $c_1$ in this result could be just $2\sqrt{M\,}$).

\begin{corollary}\lb{C.2.1}
If $u:[0,\infty)\times\bbR^d\to [0,1]$ solves \eqref{1.1} with some $f$ satisfying \textbf{(H1')} and some $\omega\in\Omega$, then for any $t\geq 0$
we have 
\[
\{x\in\bbR^d\,|\, u(t,x)\geq 1-\theta_1\}\subseteq  B_{c_1t+\kappa_1} \big( \{ x\in\bbR^d\,|\, u(0,x)\geq \theta_1\} \big),
 \]
where
\[
c_1:=2\sqrt{Md\,}>c_0\qquad \text{ and } \qquad\kappa_1:=1+ \sqrt{d/{M}\,} \ln \frac{2d}{1-2\theta_1}.
\]
\end{corollary}

\begin{proof}
The claim $c_1>c_0$ follows from the well-known estimate $c_0<2\sqrt M$ for any $M$-Lipschitz ignition reaction $F_0$. 
If $t\ge 0$ and $y\notin B_{c_1t+\kappa_1} ( \{ x\,|\, u(0,x)\geq \theta_1\} )$, then  Lemma \ref{L.2.1} with $u_1=u$, $u_2 \equiv \theta_1$, and $r=c_1t+\kappa_1$ yields
\[
    u(t,y)\leq 
    \theta_1+2d\,e^{-\sqrt{{M}/{d}\,}\,\kappa_1}<1-\theta_1,
\]
finishing the proof.
\end{proof}

We can now use these results to show that Theorems \ref{T.1.1} and \ref{T.1.2} follow from Theorem \ref{T.1.3}, so it will suffice to prove the latter (and then Theorem \ref{T.1.7}).  This also means that we will  assume either  \textbf{(H2')} or  \textbf{(H4')} in the rest of Sections \ref{S2}--\ref{S6}.  Hence, then there will be
$\mu_*,\kappa_*>0$ such that for any $u\in\calU_f$ or $u\in\bigcup_{n\ge n_4}\calU_{f_n}$, respectively, we have
\beq
\begin{aligned}\lb{2.1}
    \sup_{t\ge0 \,\&\,\eta>0} \frac{L_{u,\eta,1-\theta^*}(t)}{(1+t^{\al_2}) \eta^{-{m_2}}} & \le \mu_*^{-1},
\\    \inf_{\substack{(t,x)\in[\kappa_*,\infty)\times \bbR^d\\
    u(t,x)\in [{\theta^*},1-{\theta^*}]}}  u_t(t,x) &\geq \mu_*.
    \end{aligned}
\eeq
Notice that while \eqref{2.2} only allows us to state the first of these claims for $t\ge\kappa_*$, one can extend this to all $t\ge 0$ (with a different $\mu_*$).   This is because Lemma \ref{L.2.1} with $x=y$ shows that if $u(t,y)\ge\eta$ for some $u\in \calU_f$ (or $u\in \bigcup_{n\ge n_4}\calU_{f_n}$), $(t,y)\in[0,\kappa^*]\times\bbR^d$, and $\eta>0$, then
\[
d(y,B_k(0))\le R_0+ \sqrt{Md\,} \left(2\kappa_* + M^{-1}\ln \frac{2d}\eta \right)
\]
(when $u(0,\cdot)=u_{0,k}$), so $L_{u,\eta,1-\theta^*}(t)$ satisfies the same upper bound because $u_t\ge 0$.  We note that we could  also include $t\in[0,\kappa_*)$ in the second claim of \eqref{2.1}, at the expense of some extra work, but this 
would not be as useful.


\begin{corollary} \lb{C.2.5}
 \textbf{(H2)} implies \textbf{(H2')}.  Also,  \textbf{(H2)} and \textbf{(H4)} imply \textbf{(H4')} for some sequence of reactions $g_n$ in place of $f_n$, with possibly different $M,\theta_1,n_4,\al_4$ and with $m_4'=\infty$.
 \end{corollary}

\begin{proof}
Remark 2 after Definition \ref{D.1.2} and Lemma \ref{T.2.4} show that  \textbf{(H2)} implies \textbf{(H2')} 
with $\alpha_2=0$ and any $m_2>0$.


Let us now assume \textbf{(H2)} and \textbf{(H4)},  and with the convention that $[a,b]=\emptyset$ if $a>b$, let
\[
\tilde f_n(x,u,\omega):=\min\{f_n(x,v,\omega)\,|\, v\in \{u\}\cup [1-\theta_1,u]\}.
\]
These reactions are non-increasing in $u\in [1-\theta_1,1]$, and  still satisfy $\|\tilde f_n-f\|_\infty \leq {\al_4 n^{-m_4}}$, because $f$ is non-increasing in $u\in [1-\theta_1,1]$.  Also, each $\tilde f_n$ is obviously stationary with range of dependence no greater than that of $f_n$.

Next, let $\phi:\bbR^{d+1}\to [0,\infty)$ be a smooth function supported in $B_1(0)$ and with integral over $B_1(0)$ equal to 1.  Then let  
$\phi_n(x,u):=n^{(d+1)m_4}\phi(n^{m_4} x,n^{m_4} u)$, and consider $\tilde{f}_n*\phi_n$ (with the convolution in $(x,u)$; recall that all reactions are extended by 0 to $u\notin[0,1]$). 
With $\nabla$ being either $\nabla_x$ or $\partial_u$, we then have
\[
    \|\nabla \tilde{f}_n*\phi_n\|_\infty \leq \|\nabla f*\phi_n\|_\infty+\|(\tilde{f}_n-{f})*\nabla \phi_n\|_\infty 
    \leq M+\al_4 n^{-m_4} \|\nabla\phi_n\|_1
    =M+\al_4 \|\nabla\phi\|_1.
\]
Finally, recall $F$ from  \eqref{2.6} and let $\bar F(u):=\sup_{(x,\omega)\in\bbR^d\times \Omega} f(x,u,\omega)$, and 
\[
\tilde g_n(x,u,\omega):=\min\{\max\{(\tilde{f}_n*\phi_n)(x,u,\omega),F(u)\},\bar F(u)\}.
\]
It is not hard to see that for all large enough $n$, these functions satisfy \textbf{(H1)} with $(M,\theta_1)$ replaced by $(M+ \al_4 \|\nabla_x\phi\|_1,\frac 12 {\theta_1})$.  We also have $\|\tilde g_n-f\|_\infty \leq (2M+\al_4) n^{-m_4}$ for all large enough $n$ because $\|\tilde f_n-f\|_\infty \leq\al_4  {n^{-m_4}}$ and $F\le f(x,\cdot,\omega)\le \bar F$ for all $(x,\omega)$, and the range of dependence of $\tilde g_n$ is at most $n+2n^{-m_4}\le n+2$. 

Since $f\in\calF(F_0,M,\theta_1,\zeta,\xi)$ for some  $\zeta<{c_0^2}/{4}$ and $\xi>0$,  $\|\tilde g_n-f\|_\infty \leq (2M+\al_4) n^{-m_4}$ yields $\tilde g_n\in\calF(F_0,M+\al_4 \|\nabla_x\phi\|_1,\frac 12\theta_1,\frac 18(4\zeta+c_0^2),\frac 12\xi)$ for all large enough $n$.  Therefore, as at the start of this proof, we obtain that the $\tilde g_n$ satisfy \textbf{(H2')} uniformly in $n$, for all large enough $n$.
Hence, the reactions $g_n:=\tilde g_{n-2}$ satisfy \textbf{(H4')} with possibly different $M,\theta_1,n_4,\al_4$ and with $m_4'=\infty$.
%
\end{proof}

%
%

The next result uses Lemma \ref{L.2.1}  and the bound $f(\cdot,u,\cdot)\ge \alpha_1(1-u)^{m_1}$ for $u$ near 1 to essentially obtain an upper bound on $\kappa_0(M,\theta_1,m_1,\al_1,\frac{c_0}4,\theta)$ from Lemma \ref{L.2.2} as $\theta\to 1$ (we could similarly do this for any $c<c_0$ in place of $\frac{c_0}4$).  

\begin{lemma}\label{Cor.2.1}
Let $u:[0,\infty)\times\bbR^d\to [0,1]$ solve \eqref{1.1} with $f$ satisfying \textbf{(H1')} and some $\omega\in\Omega$.  There is $D_1=D_1(M,\theta_1,m_1,\al_1)$ such that if $u(t_0,y)\geq 1-\theta^*$ for some $t_0\ge 1$ and $y\in\bbR^d$, then for any  $\theta\in[1-\theta^*,1)$ and $t\geq t_0 + D_1 (1-\theta)^{1-m_1}$  we have
\[
\inf_{|x-y|\le c_0(t-t_0)/4} u(t,x)\geq \theta .
\]
\end{lemma}


\begin{proof}
Lemma \ref{L.2.2} shows that with $\kappa_0$ from \eqref{2.10a} we have
\[
\inf_{|x-y|\leq c_0(t-t_0)/2 }u(t,x) \geq 1-\theta^*
\]
for any $t\ge t_0+\kappa_0$.  Since
\[
U(s)=1-((\theta^*)^{1-{m_1}}+({m_1}-1) \al_1 s)^{-{1}/{({m_1}-1)}}
\]
solves the ODE
$U'= \al_1 (1-U)^{m_1}$
with $U(0)=1-\theta^*$,
it follows from Lemma \ref{L.2.1} that 
\[
u(t+s,x)\ge U(s) - 2de^{\sqrt{M/d\,} \left( |x-y|-c_0(t-t_0)/2 + 2\sqrt{Md\,}s \right)}
\]
for all $(s,x)\in [0,\infty)\times\bbR^d$.  Picking $s:=\frac 1{(m_1-1)\al_1} [(\frac{1-\theta} 2)^{1-m_1} - (\theta^*)^{1-m_1}]$ makes $U(s)\ge \frac {1+\theta}2$.  The second term on the right will be no more than $\frac{1-\theta}2$ when $|x-y|\le \frac{c_0}4 (t+s-t_0)$, provided
\[
\frac{c_0(t+s-t_0)}4 - \frac{c_0(t-t_0)}2 +2\sqrt{Md\,}s\le \sqrt{d/M\,}\ln\frac{1-\theta}{4d}, 
\]
which holds as long as $t\ge t_0 +D_1(1-\theta)^{1-m_1}$ for some $D_1=D_1(M,\theta_1,m_1,\al_1)\ge \kappa_0$.
 Replacing $D_1$ by $D_1+\frac 1{(m_1-1)\al_1} 2^{m_1-1}$ now yields the claim for $t+s$ in place of $t$ whenever $t+s\ge t_0 +D_1(1-\theta)^{1-m_1}$.
\end{proof}

Finally, we will need  two results providing estimates on how much solutions to \eqref{1.1} may change when the reaction $f$ is perturbed.
The first one concerns the case when the reaction can change only where the solution is initially close to $1$.  In it, we will also use the definition
\[
T_u(x):=\inf\{t\geq 0\, | \, u(t,x)\geq 1-\theta^*\}.
\]

\begin{lemma}
\label{L.3.1}
Let  $f_1$ satisfy \textbf{(H2')} and $f_2$ satisfy \textbf{(H1')},  and let $M_*:=\frac{1+M}{\mu_*}$, with $\mu_*,\kappa_*$ from \eqref{2.1} for all $u\in\calU_{f_1}$.
Fix some $\omega\in\Omega$ and let $u_1,u_2:[0,\infty)\times\bbR^d\to [0,1]$ solve \eqref{1.1} with   $f_1,f_2$ in place of $f$, respectively.    If $u_1\in\calU_{f_1}$, $t_0\ge 0$, and for some $\eta\in [0,\frac{1}{2}\min\{{\theta^*},M_*^{-1}\} ]$ we have
\beq\lb{2.4}
f_1(x,u,\omega)=f_2(x,u,\omega)\qquad \text{whenever }u_1(t_0,x)<1-\eta\text{ and }u\in [0,1],
\eeq
then
\[
    u_+(t,x) :=u_1((1+{M_*}\eta)t+t_0,x)+\eta
\]
is a supersolution to \eqref{1.1} with $f_2$ in place of $f$ on $(\kappa_*,\infty)\times \bbR^d$, and
\[
u_{-}(t,x) :=u_1((1-{M_*}\eta)t+t_0,x)-\eta
\]
is a subsolution to \eqref{1.1} with $f_2$ in place of $f$ on $(2\kappa_*,\infty)\times \{ x\in\bbR^d \,|\,u_1(t_0,x)<1-\eta\}$.

Moreover, there is $D_2=D_2(M,\theta_1,m_1,\al_1)\ge 1$ 
such that if also
\[
\sup_{x\in B_R(y)} (u_2(0,x)- u_1(t_0,x))\le \eta 
\]
for some $y\in\bbR^d$ and
$R\ge D_2(1+{T}_{u_2}(y)),$
then
\beq \lb{2.888}
{T}_{u_2}(y)\ge  (1+ {M_*}\eta)^{-1} (T_{u_1}(y)-t_0- 2\kappa_*-\kappa_0).
\eeq
\end{lemma}



The proof of Lemma \ref{L.3.1} appears in Appendix B.

Our last preliminary lemma concerns the case when the reaction may be perturbed anywhere, although not by a lot. 

\begin{lemma}\lb{L.4.3}
Let  $f_1$ satisfy \textbf{(H2')} and $f_2$ satisfy \textbf{(H1')}, with at least one of them satisfying \textbf{(H3)} with $\al_3\le 1$, and let $M_*,D_2$ be from Lemma \ref{L.3.1}. 
Fix some $\omega\in\Omega$ and let $u_1,u_2:[0,\infty)\times\bbR^d\to [0,1]$ solve \eqref{1.1} with $f_1,f_2$ in place of $f$, respectively. If $u_1\in\calU_{f_1}$, for some  $y\in\bbR^d$, $R\ge D_2(1+{T}_{u_2}(y))$, and  $\eta\in [0,\frac{1}{2}\min\{{\theta^*},M_*^{-1}\} ]$ we have 
\beq\lb{3.3} 
f_1(x,u,\omega)\geq f_2(x,u,\omega)-\al_3 \eta^{m_3}\quad \text{ for all } (x,u)\in B_R(y)\times [0,1],
\eeq
and $u_2(0,x)\leq u_1(t_0,x)$ for some $t_0\ge 0$ and all $x\in B_R(y)$, 
then \eqref{2.888} holds.
\end{lemma}

The proof of Lemma \ref{L.4.3} appears in Appendix C.

\subsection{Notations}
Since Sections \ref{S3}--\ref{S6} are  just the proof of Theorem \ref{T.1.3}, we will assume either \textbf{(H2')} or \textbf{(H3)}+\textbf{(H4')} in them.
Any constants $C,C_0,C_\delta,C(\delta),\dots$ may depend on
\beq\lb{const}
{M},\theta_1,m_1,\al_1,{m_2},\al_2,m_3,\al_3,m_4, m_4',\al_4,\mu_*,\kappa_*
\eeq
(except for $m_3,\al_3,m_4,m_4',\al_4$ when \textbf{(H2')} is assumed; dependence on $d$ is implicit in the whole paper).  Any other dependence will be explicitly declared, for instance, ``for some $C_\delta$'' or ``for some $C=C(\delta)$'' will mean that this constant depends on $\delta$ as well as \eqref{const}.  These constants may also vary from line to line. We recall that  $M,\theta_1,m_1,\al_1$ are from \textbf{(H1)}; $m_2,\al_2$ from \textbf{(H2')}; $m_3,\al_3$ from \textbf{(H3)};
$m_4,m_4',\al_4$ from \textbf{(H4')}; and $\mu_*,\kappa_*$ from \eqref{2.1}. 

The constants 
\[
\theta_2,\theta^*,c_0,\kappa_0,c_1,\kappa_1,D_1,D_2,
R_0, M_*
\]
also {only depend on subsets of \eqref{const}}, with 
$\theta_2,c_0,\kappa_0$ from Lemma \ref{L.2.2}; $\theta^*$ from \eqref{2.10}; $R_0$ from Lemma~\ref{L.2.4}; $c_1,\kappa_1$ from Corollary \ref{C.2.1};    $D_1$ from Lemma \ref{Cor.2.1}; and $M_*,D_2$ from Lemma \ref{L.3.1}.

For $A\subseteq\bbR^d$ and $r\ge 0$ we let $B_r(A):=A+(B_r(0)\cup\{0\})$ and $A^0_r:=A\backslash\overline{B_r(\partial A)}$ (so $A^0_0$ is the interior of $A$).
For sets $U,V\subseteq \bbR^d$, we let $d(U,V):=\inf_{x\in U\,\&\,y\in V} |x-y|$ and
\[
{d}_H(U,V):=\max\left\{\sup_{x\in U}d(x,V),\,\sup_{y\in V}d(y,U)\right\}
\]
be their standard and Hausdorff distances. 

Widths $L_{u,\eta,\theta}(t)$ of transition zones of solutions are defined in \eqref{d.2.1}, and the special sets $\calU_f$ of solutions  evolving from approximate characteristic functions of balls are defined in \textbf{(H2')}.

Finally, we recall that $\calE(U)$ is the $\sigma$-algebra generated by 
the family of random variables 
$
\{f(x,u,\cdot)  \,|\, (x,u)\in U\times [0,1]\}.
$
Further important notation related to the dependence of reactions and solutions on $\omega$ appears below, particularly early in Sections \ref{S3} and \ref{S5}.





\section{Fluctuations for Reactions with Finite Ranges of Dependence}\lb{S3}

The proof of Theorem \ref{T.1.3} will be based on the analysis of the dynamics of special solutions, starting from the approximate characteristic functions $u_{0,S}$ of sets $S\subseteq \bbR^d$ satisfying Lemma~\ref{L.2.4}.  
For each $S$ and $\omega\in\Omega$, we therefore let $u(\cdot,\cdot,\omega;S)$ be the solution to
\begin{equation} \lb{e.3.1}
\begin{aligned}
  &u_t=\Delta u+f(x,u,\omega)\qquad && \text{ on }(0,\infty)\times \bbR^d,\\
& u(0,\cdot,\omega;S)=u_{0,S} \qquad && \text{ on } \bbR^d.
\end{aligned}
\end{equation}
Recall that $u_t(\cdot,\cdot,\omega;S)\ge 0$ because $u_{0,S}$ satisfies \eqref{2.6}.  
Let us also define for any $x\in \bbR^d$,
\beq\lb{d.3.2}
T(x,\omega;S):=\inf\{t\geq 0\,|\, u(t,x,\omega;S)\geq 1-\theta^*\} \ge 0,
\eeq
with $\theta^*$ from \eqref{2.10}, which can be thought of as the time when this solution reaches $x$.

Our goal is now to estimate the stochastic fluctuations of $T(x,\omega;S)$.  
In this section we will consider the first case in Theorem \ref{T.1.3}, when the reaction satisfies \textbf{(H2')} and has a finite range of dependence.  
We will only need to treat the cases when $S$ is either a ball or a half-space. 
We start with $S$ being a ball, when the solution in \eqref{e.3.1} will be precisely  the one from $\calU_f$ corresponding to $(S,\omega)$.  Hence, below always $u(\cdot,\cdot,\omega;S)\in\calU_f$ and \eqref{2.1} holds for it.
\smallskip

\noindent {\it Remark.}  We note that if we enlarge $\calU_f$ (or each $\calU_{f_n}$) to include solutions with initial data $u_{0,S}$ from Lemma \ref{L.2.4} for all $S$ from some family $\calS$ of bounded subsets of $\bbR^d$, and \eqref{2.1} still continues to hold with some $\mu_*,\kappa_*>0$, then everything in this section and the next holds without change (and with uniform constants) for either all $S\in\calS$ (the results involving balls only) or for all local limits in Hausdorff distance of translations of sets from $\calS$ (the results involving half-spaces).  In particular, Remark 2 after Definition \ref{D.1.2} and Lemma \ref{T.2.4} show that if we assume \textbf{(H2)}, then we can let $\calS$ be the family of all bounded subsets of $\bbR^d$.

\begin{proposition}\lb{P.3.3}
Let $f$  satisfy \textbf{(H2')} and let
\beq\lb{3.13}
\be_1:=\max\left\{\frac{m_1-1}{m_1},\,\frac{m_2+\al_2}{m_2+1}\right\}  \qquad(\in (0,1)).
\eeq
There is $C_0\geq 1$ such that if $f$ has range of dependence at most $\rho\in [1,\infty)$ and 
$S=B_k(0)$ for some $k\in\bbN$,
then for all $x\in\bbR^d$ and $\lambda>0$ we have
\beq\lb{3.4}
\bbP\left( \big| T(x,\cdot\,;S)-\bbE [T(x,\cdot\,;S)] \big| \geq \lambda\right)\leq 2\exp \left( \frac{-\lambda^2}{C_0(1+ d(x,S))(\rho+d(x,S)^{\be_1})} \right).
\eeq
\end{proposition}

\noindent {\it Remarks.} 1.  The point here is that by choosing $\lambda\sim d(x,S)^\gamma$ for some $\gamma\in(\frac{1+\be_1}2,1)$, one obtains a fast-decreasing bound on the probability of $O(d(x,S)^\gamma)$ fluctuations of $T(x,\cdot\,;S)$ from its mean (the latter is of course $\sim d(x,S)$ by Lemma \ref{L.2.2} and Corollary \ref{C.2.1}). 
\smallskip

2.  Note that Remark 2 after Definition \ref{D.1.2} and Lemma \ref{T.2.4} show that  \textbf{(H2)} implies \textbf{(H2')} with $\alpha_2=0$ and any $m_2>0$, so then $\be_1=\frac{m_1-1}{m_1}$.  If there is also $\theta<1$ such that $\inf_{(x,u,\omega)\in\bbR^d\times(\theta,1)\times\Omega} (1-u)^{-1} f(x,u,\omega)>0$, then this means that $\be_1>0$ can be made arbitrarily small.
\smallskip

The rest of this section is devoted to proving Proposition \ref{P.3.3} (and then extending it from balls to 
half-spaces in Proposition \ref{P.3.4}). We will therefore assume \textbf{(H2')} and the range of dependence of $f$ being at most $\rho\in[1,\infty)$.  We will also fix $S=B_k(0)$ and drop it from the notation (so the functions from \eqref{e.3.1} and \eqref{d.3.2} are $u(t,x,\omega)$ and $T(x,\omega)$, respectively) but all estimates will be independent of $S$ (i.e., of $k$).  Recall that all constants with $C$ in them 
 depend on \eqref{const}, with any other dependence explicitly stated, and may vary from line to line.

\subsection{Construction of a martingale}

Let $\calK$ be the set of all non-empty compact subsets of $\bbR^d$, and endow it with Hausdorff distance ${d}_H$.  For each  $a=(a_1,...,a_d)\in \bbZ^d$, let
\[
\calC_a:=[a_1,a_1+1)\times ... \times[a_d,a_d+1),
\]
and for each finite $\emptyset \neq A\subseteq \bbZ^d$, let
\[
P_A:=\left\{K\in\calK\,\bigg|\,K\cap \frac{1}{\sqrt{d}}\calC_a\neq \emptyset \text{ if and only if }a\in A\right\}.
\]
Let us label all such $A$ as $A_0,A_1,\dots$, and denote $P_i:=P_{A_i}$.   Then $\{{P}_i\}_{i\in\bbN_0}$ is a (pairwise disjoint) partition of the metric space $(\calK,{d}_H)$, and $\diam_H({P}_i)= 1$ for each $i\in\bbN_0$.

For each $i$, let
\[
 K_i:=\bigcup_{a\in A_i}\frac{1}{\sqrt{d}}{\calC_a}=\bigcup_{K\in P_i}K
\]
(note that $K_i$ is not compact), so that for each $K\in P_i$ we have
\[
K\subseteq K_i\subseteq B_1(K).
\]
Notice also that for each $i,j\in\bbN_0$ we have
\beq\lb{3.12}
\text{if $P_i \ni K\subseteq K'\in P_j$, then $K_i\subseteq K_j$.}
\eeq

Next, for any $(t,\omega)\in[0,\infty)\times \Omega$ and $\theta\in(0,1)$, we let
\[
\Gamma_{u,\theta}(t,\omega):=\{x\in\bbR^d\,|\,u(t,x,\omega)\geq \theta\}.
\]
Since $S$ is bounded, Lemma \ref{L.2.1} shows that all these sets are compact. 
Then for all $i\in\bbN_0$ and $t\ge 0$ we let
\[
E_i(t):=\{\omega\in\Omega\,|\, \Gamma_{u,{{\theta^*}}}(t,\omega)\in P_i\} \subseteq\Omega,
\]
with ${{\theta^*}}$ from \eqref{2.10}.  Then $\{E_i(t)\}_{i\in\bbN_0}$ is a pairwise disjoint partition of $\Omega$ for each $t\ge 0$ because $\Gamma_{u,{{\theta^*}}}(t,\omega)\neq\emptyset$ (due to $\theta^*\le 1-\theta^*$).  Also note that $\alpha_2<1$ allows us to only track the evolution of one of the sets $\Gamma_{u,\theta}$ (see the remark after \textbf{(H2')}), and in this section we choose it to be $\Gamma_{u,\theta^*}$.  Finally, for $(t,\omega)\in[0,\infty)\times \Omega$, let
\[
{\izero}(t,\omega):=i\qquad  \text{when $\omega\in E_i(t)$},
\]
and for each $(t,x)\in[0,\infty)\times \bbR^d$ let
\[
F_{t,x}:=\{\omega\in\Omega \,|\, x\in K_{{\izero}(t,\omega)} \}.
\]
The latter is a slightly different version of the set of $\omega$ for which the solution $u$ has reached $x$ by time $t$ (we have $F_{t,x} \supseteq \{\omega\in\Omega\,|\, T(x,\omega)\le t\} \supseteq F_{t-C,x}$ for some $C>0$ and all large enough $t$, due to \eqref{2.1}).
Let us also pick $A_0$ so that ${\izero}(0,\omega)=0$ for each $\omega\in\Omega$ (note that ${\izero}(0,\omega)$ does not depend on $\omega$).

%
%

Since $u_t\ge 0$ and $f(\cdot,u,\cdot)\equiv 0$ for $u\in[0,\theta^*]$, the solution $u(\cdot,\cdot,\omega)$  only depends on the reaction inside $K_{{\izero}(t,\omega)}$ up to time $t$. Hence, we have the following lemma.

\begin{lemma}\lb{L.3.7}
For any $(t,x)\in[0,\infty)\times \bbR^d$ and $s\in[0,t]$, the set $E_i(t)$ and the function
$u(s,x,\cdot)\chi_{E_i(t)}(\cdot)$ are $\calE(K_i)$-measurable for each  $i\in\bbN_0$. 
\end{lemma}

\begin{proof}
Fix any  $i\in\bbN_0$, and let $g$ be any reaction satisfying \textbf{(H1')} and
\[
    f(x,u,\omega)=g(x,u,\omega)\qquad \text{for all $(x,u,\omega)\in K_i\times [0,1]\times\Omega$}.
\]
Then for each $\omega\in\Omega$, let $v$ be the solution to  
\[
  v_t=\Delta v+g(x,v,\omega)
\]
with the same initial data $v(0,\cdot,\omega)=u_{0,S}$ as $u$, and fix any $t\ge 0$. 
Since $u_t,v_t\ge 0$ and $f(\cdot,u,\cdot)\equiv 0$ for $u\in[0,\theta^*]$, it follows that  $v(s,\cdot,\omega)=u(s,\cdot,\omega)$ whenever $\omega\in E_i(t)$ and $s\le t$.  Similarly, we have $v(s,\cdot,\omega)=u(s,\cdot,\omega)$ whenever $\Gamma_{v,{{\theta^*}}}(t,\omega)\in P_i$ and $s\le t$ (with $\Gamma_{v,\theta}$ defined analogously to $\Gamma_{u,\theta}$).  
In particular, $\omega\in E_i(t)$ if and only if $\Gamma_{v,{{\theta^*}}}(t,\omega)\in P_i$.

Since this holds for each $g$ as above, it follows that $E_i(t)\in \calE(K_i)$, and then also that $u(s,x,\cdot)\chi_{E_i(t)}(\cdot)$ is $\calE(K_i)$-measurable for any $(s,x)\in[0,t]\times \bbR^d$.
%
%
%
%
%
\end{proof}

For each $t\geq 0$, let $\calG_t$ be the $\sigma$-algebra on $\Omega$ generated by
\[
\bigcup_{i\in\bbN_0\,\&\, s\in [0,t]}\calE(K_i)|_{E_i(s)}.
\]
That is, 
$\calG_t$ is generated by all events  $F\cap E_i(s)$ with $i\in\bbN_0, F\in\calE(K_i), s\in[0,t]$.  
Then $\{\calG_t\}_{t\geq 0}$ is a filtration on $(\Omega,\calF)$, and $\calG_0=\calE(K_{0})$ because $E_0(0)=\Omega$ and $E_i(0)=\emptyset$ for $i>0$.
Lemma \ref{L.3.7} then shows that  $u(s,x,\cdot)$ is $\calG_t$-measurable for any  $(t,x)\in[0,\infty)\times \bbR^d$ and $s\le t$.
Since $E_i(s)\in \calE(K_i)$ for all $i\in\bbN_0$ and $s\ge 0$ (so also $F\cap E_i(s)\in \calE(K_i)$ above), ${\izero}(s,\cdot)$ and $F_{s,x}$ are also $\calG_t$-measurable for all $(t,x)\in[0,\infty)\times\bbR^d$ and $s\le t$.

We note that $\calG_t$ is actually simpler than its definition suggests, and for each $t\ge 0$ and $j\in\bbN_0$ we in fact have
\beq \lb{3.333}
\calG_t |_{E_j(t)} = \calE(K_j)|_{E_j(t)}
\eeq
(recall also that $\{E_j(t)\}_{j\in\bbN_0}$ is a partition of $\Omega$ for each $t\ge 0$, and note that only finitely many of these sets are non-empty due to Lemma \ref{L.2.1}).
Indeed, let us consider any $i\in\bbN_0$ and $s\in[0,t]$ such that $E_i(s)\cap E_j(t)\neq \emptyset$.  Then there is $\omega\in\Omega$ such that
\[
{\izero}(t,\omega)=j \qquad \text{ and }\qquad {\izero}(s,\omega)=i.
\]
From \eqref{3.12} and 
\[
\Gamma_{u,{{\theta^*}}}(r,\omega)\subseteq \Gamma_{u,{{\theta^*}}}(s,\omega)
\]
for any such $\omega$, we obtain $K_i\subseteq K_j$.  But then $E_i(s)\in \calE(K_i)\subseteq \calE(K_j)$ and
\[
 \calE(K_i) |_{E_i(s)\cap E_j(t)} \subseteq  \calE(K_j) |_{E_i(s)\cap E_j(t)} \subseteq \calE(K_j) |_{E_j(t)}.
 \]
Since this clearly also holds when $E_i(s)\cap E_j(t)= \emptyset$, the definition of $\calG_t$ proves \eqref{3.333}.

Similarly to \cite{armstrong2014,armstrong2015}, we will prove Proposition \ref{P.3.3} by studying the $\calG_t$-adapted martingale $\{X_t:=\bbE\left[T(x,\cdot)\,|\,\calG_t\right]\}_{t\ge 0}$ for any $x\in\bbR^d$ and estimating its increments, which will then allow us to apply Azuma's inequality to bound the fluctuations of $T(x,\cdot)$.

\begin{lemma}[\hskip0.01mm Azuma's inequality]\lb{L.3.5} 
Let $\left\{X_{k}\right\}_{k\in\bbN_0}$  be a martingale on $(\Omega,\calF,\bbP)$. If for each $k\in\bbN$ there is $c_k\ge 0$ such that $|X_k-X_{k-1}|\leq c_k$ almost surely, 
then for all $\lambda>0$ and $n\in\bbN$,
\[
 \bbP[|X_{n}-X_{0}|\geq \lambda]\leq 2\exp \left(-{\frac {\lambda^{2}}{2\sum _{k=1}^{n}c_{k}^{2}}}\right).
 \]
\end{lemma}

We first show that for any 
$\omega\in F_{t,x}$
the difference 
$\bbE[T(x,\cdot)\,|\,\calG_t](\omega)- T(x,\omega)$ is uniformly bounded, and 0 if $\omega\in F_{s,x}$ when $t-s$ is large enough (note that $F_{s,x}\subseteq F_{t,x}$ if $s\le t$ because $u_t\ge 0$).
This is due to $F_{t,x}\in\calG_t$ and the above-mentioned relationship of $F_{t,x}$ and $\{\omega\in\Omega\,|\, T(x,\omega)\le t\}$.


\begin{lemma}\lb{L.3.3}
 There is $C>0$ such that for each
 $(t,x)\in[0,\infty)\times \bbR^d$ we have
\[
   \left|\bbE[T(x,\cdot)\chi_{F_{t,x}}\,|\,\calG_{t}]-T(x,\cdot)\chi_{F_{t,x}}\right|\leq C \qquad \text{ on }\Omega,
\]
and if also $s\in[0,t-C]$, then
\beq
\lb{3.17}
\bbE[T(x,\cdot)\chi_{ F_{s,x}} \,|\,\calG_{t}]=T(x,\cdot)\chi_{ F_{s,x}} \qquad \text{ on }\Omega.
\eeq

\end{lemma}

\begin{proof}
For any $(x,\omega)\in\bbR^d\times\Omega$, let
\[
\tau(x,\omega):=\inf \left\{t\geq 0\,|\,  x\in B_1(\Gamma_{u,{{\theta^*}}}(t,\omega))
\right\} \le T(x,\omega).
\]
If $\omega\in F_{t,x}$
for some $t\ge 0$, then $\tau(x,\omega)\leq t$ due to $K_{{\izero}(t,\omega)}\subseteq B_1(\Gamma_{u,\theta^*}(t,\omega))$.  
And since  $u(s,x,\cdot)$ is $\calG_t$-measurable for all $s\le t$, we see that $\tau(x,\cdot)\chi_{ F_{t,x}}(\cdot)$ is $\calG_{t}$-measurable.

For each $(x,\omega)\in\bbR^d\times\Omega$, there is
$y\in \overline{B_1(x)}$ with $u(\tau(x,\omega),y,\omega)\geq \theta^*$,
so \eqref{2.1} shows that
\beq\lb{e.3.4}
u\left( \tau(x,\omega) +\kappa_*+ \mu_*^{-1},y,\omega \right)\geq 1-\theta^*,
\eeq
and then Lemma \ref{L.2.2} yields
\[
u(\tau(x,\omega) +\kappa_*+ \mu_*^{-1}+\kappa_0+2c_0^{-1},x,\omega)\geq 1-\theta^*.
\]
Therefore there is $C$ such that $T(x,\omega)\leq  \tau(x,\omega)+C$, and hence $|T(x,\omega)- \tau(x,\omega)|\leq C$.
%
Hence for any $(t,x)\in[0,\infty)\times \bbR^d$ we obtain using
$\calG_{t}$-measurability of $\tau(x,\cdot)\chi_{ F_{t,x}}$, 
\[
  \left|\bbE[T(x,\cdot)\chi_{ F_{t,x}}\,|\,\calG_{t}]-T(x,\cdot)\chi_{ F_{t,x}}\right|
\leq \left|\bbE[\tau(x,\cdot)\chi_{ F_{t,x}}\,|\,\calG_{t}]-\tau(x,\cdot)\chi_{ F_{t,x}}\right|+2C
= 2C,
\]
yielding the first claim.
If also $s\le t-C$, then for all $\omega\in F_{s,x}$ we have
\[
T(x,\omega)\le \tau(x,\omega)+C\leq s +C\le t.
\]
Since $u(t,x,\cdot)$ is $\calG_t$-measurable, this shows that so is $T(x,\cdot)\chi_{ F_{s,x}}$, and \eqref{3.17} follows.
\end{proof}

When $\omega\in\Omega\setminus F_{t,x}$ (that is, essentially, when the solution $u$ has not yet reached $x$ by the time $t$), we will obtain a different kind of estimate.

Let $\rho\geq 1$ be from Proposition \ref{P.3.3}, and for each $i\in\bbN_0$ let
\begin{align*}
    g_i(x,u,\omega):=\, \psi_i(x)\bbE[f(x,u,\cdot)]+(1-\psi_i(x))f(x,u,\omega),
\end{align*}
where $0\le \psi_i\le 1$ is Lipschitz with a uniform-in-$i$ constant, with $\psi_i\equiv 1$ on $B_\rho(K_i)$ and $\psi_i \equiv 0$ on $\bbR^d\setminus B_{\rho+1}(K_i)$. 
Then $g_i$ is Lipschitz in $(x,u)$ (with a uniform $M$-dependent constant $\ge M$, which we will call $M$ from now on), and 
$g_i(x,u,\cdot)$ is independent of $\calE(K_i)$ for all $(x,u)\in\bbR^d\times [0,1]$ because $f$ has range of dependence at most $\rho$.
Of course, 
\beq\lb{3.24}
g_i\equiv f\quad \text{ on }(\bbR^d \setminus B_{\rho+1} (K_i))\times [0,1]\times\Omega.
\eeq
For each $\omega\in\Omega$, let now $v_i$ be the solution to
\[ 
\begin{aligned}
& (v_i)_t=\Delta v_i+g_i(x,v_i,\omega)\qquad && \text{ on }(0,\infty)\times \bbR^d,\\
& v_i(0,\cdot,\omega)=u_{0,K_i}\qquad && \text{ on } \bbR^d, 
\end{aligned}
\]
with $u_{0,K_i}$ from Lemma \ref{L.2.4}.
Then $v_i(t,x,\cdot)$ is independent of $\calE(K_i)$ for each $(t,x)$, and so is
\[
{T}_i(x,\cdot):=\inf\{t\ge 0\,|\, v_i(t,x,\cdot)\geq 1-\theta^*\}.
\]

\begin{proposition}\lb{P.3.0}
There is $C\geq 1$ such that for each
 $(t,x)\in[0,\infty)\times \bbR^d$ and $\omega\in\Omega\setminus F_{t,x}$ we have
\[
\left| T(x,\omega) - t -{T}_{{\izero}(t,\omega)}(x,\omega)\right|\leq C(\rho+ d (x,S)^{\be_1}).
\]
\end{proposition}

\noindent {\it Remark.}  This shows that the difference of the time it takes to reach $x$ from $S$ and the sum of any smaller time $t$ and the time it takes to reach $x$ from $K_{{\izero}(t,\omega)}$ (which approximates $\Gamma_{u,{{\theta^*}}}(t,\omega)$) is sublinear in $d (x,S)$.  Hence, this result yields a certain additive structure (up to lower order errors) for the arrival times of solutions with initial data from Lemma \ref{L.2.4}.
\smallskip

\begin{proof}
We will use $(t_0,x_0)$ in place of $(t,x)$ in the proof.  
Fix any $\omega\in\Omega\setminus F_{t_0,x_0}$  and  let $i:={\izero}(t_0,\omega)$.
Note that since $\omega\in E_i(t_0)$, we have
\beq\lb{3.8}
 \Gamma_{u,{{\theta^*}}}(t_0,\omega)\subseteq K_i\subseteq B_1(\Gamma_{u,{{\theta^*}}}(t_0,\omega)).
\eeq
Moreover, since $x_0 \notin K_i$,
Lemma~\ref{L.2.2} shows that $t_0\leq \kappa_0+ \frac 2{c_0}d(x_0,S)$.

Let us first show that
\[
T(x_0,\omega)-t_0 -{T}_{i}(x_0,\omega) \leq C(\rho+ d (x_0,S)^{\be_1}).
\]
 From \eqref{2.1} we know that  with $L_1:=\mu_*^{-1} (1+t_0^{\al_2}) (\theta^*)^{-m_2}+1$ we have
\[
\Gamma_{u,{{\theta^*}}}(t_0,\omega) \subseteq B_{L_1-1} (\Gamma_{u,1-\theta^*}(t_0,\omega)).
\]
Hence we obtain
\[
K_i \subseteq {B}_{L_1} (\Gamma_{u,1-\theta^*}(t_0,\omega)).
\]
This and Lemma \ref{L.2.2}  now show
\beq\lb{3.33}
B_{R_0}(K_i)\subseteq \Gamma_{u,1-\theta^*}(t_0+t_1,\omega)
\eeq
for
\[
t_1:=\kappa_0+2c_0^{-1}(R_0+{L_1}),
\]
hence
\beq\lb{3.28}
u(t,\cdot,\omega)\geq u_{0,K_i}=v_i(0,\cdot,\omega)
\eeq
for any $t\ge t_0+t_1$.
Since $t_0\leq \kappa_0+\frac 2{c_0}d(x_0,S)$, there is $C>0$ such that
\[
t_1\leq C( 1+d(x_0,S)^{\al_2}).
\]

Now take
\beq\lb{3.26}
\eta:= \min\left\{\frac{\theta^*}{2},\,\frac{1}{2M_*},\, d(x_0,S)^{-\gamma}\right\}>0,
\eeq
where 
\[
\gamma:=\min\left\{ \frac{1}{m_1} , \frac{1-{\al_2}}{m_2+1} \right\}\in(0,1).
\]
Note that \eqref{3.13} shows that
\beq \lb{3.13'}
\max\{\gamma({m_1}-1), {\al_2}+\gamma m_2,1-\gamma \} = \be_1 <1.
\eeq

It follows from \eqref{3.33} and Lemma \ref{Cor.2.1} that for $t_2:=\frac 4{c_0}(\rho+1)+ D_1\eta^{1-{m_1}}$ we have 
\beq
\label{3.21}u(t_0+t_1+t_2,\cdot ,\omega)\geq 1-\eta\qquad \text{on } B_{\rho+1}(K_i).
\eeq
Moreover, 
from \eqref{3.26} and \eqref{3.13'} we see that there is $C>0$ such that
\beq\lb{3.2}
    t_3:=t_1+t_2\leq  C( \rho+d (x_0,S)^{\al_2}+d(x_0,S)^{ \gamma ({m_1}-1)}) \leq C(\rho+d(x_0,S)^{\be_1}).
\eeq

We now apply Lemma \ref{L.3.1} with $(f,g_i,u,v_i,t_0+t_3,\infty)$ in place of
$(f_1, f_2,u_1, u_2,t_0,R)$.  Its hypotheses are satisfied due to \eqref{3.24}, \eqref{3.28} and \eqref{3.21}, and it yields
\[
T(x_0,\omega)-t_0-t_3\leq (1+{M_*}\eta){T}_i(x_0,\omega) +2\kappa_*+\kappa_0.
\]
This, \eqref{3.2}, and ${T}_i(x_0,\omega)\leq C(1+d(x_0,K_i))\leq C(1+d(x_0,S))$ (which follows from Lemma~\ref{L.2.2}, with some $C>0$) show that there is indeed some $C>0$ such that
\beq\lb{3.9}
T(x_0,\omega)- t_0 -{T}_i(x_0,\omega) \leq 
    {M_*}\eta {T}_i(x_0,\omega) +2\kappa_*+\kappa_0+t_3
    \leq C(\rho+d(x_0,S)^{\be_1}).
\eeq

Let us now turn to 
\[
t_0 + {T}_{i}(x_0,\omega) - T(x_0,\omega) \leq C(\rho+ d (x_0,S)^{\be_1}),
\]
and let $\eta$ be again from \eqref{3.26}. We will now need to estimate $u$ from above in terms of some time-shift of $v_i$.
It follows from Lemma \ref{L.3.1} that
\[
u_-(t,x,\omega):= u((1-M_*\eta)t+t_0,x,\omega)-\eta \qquad (\le 1-\eta)
\]
is a subsolution to \eqref{1.1} with reaction $f$ on $(2\kappa_*,\infty)\times (\bbR^d\setminus \Gamma_{u,1-\eta}(t_0,\omega))$.  We also know that $v_i$ is a solution of the same equation on $(0,\infty)\times (\bbR^d\setminus B_{\rho+1}(K_i))$.  In order to be able to compare them, we need some more estimates involving these sets.


Since $\theta^*\le\theta_1$, Corollary \ref{C.2.1} and \eqref{3.8} yield
\[
\Gamma_{u,1-\theta^*}(t_0+2\kappa_*,\omega)\subseteq B_{2c_1\kappa_*+\kappa_1}(\Gamma_{u,\theta^*}(t_0,\omega)) \subseteq B_{2c_1\kappa_*+\kappa_1}(K_i).
\]
This and \eqref{2.1} show that with $L_2:= \mu_{*}^{-1}(1+(t_0+2\kappa_*)^{\al_2})\eta^{-{m_2}}$ and
\[
L_3:=\max\{L_2+2c_1\kappa_*+\kappa_1,\rho+1\},
\]
we have
\beq\lb{3.31}
\Gamma_{u_-,0}(2\kappa_*,\omega)\subseteq\Gamma_{u,\eta}(t_0+2\kappa_*,\omega)\subseteq B_{L_2}(\Gamma_{u,1-\theta^*}(t_0+2\kappa_*,\omega))\subseteq B_{L_3}(K_i).
\eeq
We note that this also implies
\beq\lb{3.32}
\Gamma_{u,1-\eta}(t_0,\omega)\subseteq\Gamma_{u,\eta}(t_0+2\kappa_*,\omega)\subseteq B_{L_3}(K_i).
\eeq
Moreover,  Lemma \ref{Cor.2.1} shows that
\beq \lb{3.20}
B_{L_3}(K_i)\subseteq \Gamma_{v_i,1-\eta}(t_4,\omega),
\eeq
with $t_4:= \frac 4{c_0}L_3 +D_1\eta^{1-{m_1}}$.
 Then \eqref{3.26}, \eqref{3.13'}, and  $t_0\leq \kappa_0+\frac 2{c_0}d(x_0,S)$ show that there is $C>0$ such that
\beq \lb{3.2a}
     t_4 \leq C(\rho+t_0^{\al_2}\eta^{-m_2}+\eta^{1-m_1})
    \leq C(\rho+d(x_0,S)^{\be_1}).
\eeq

Using  \eqref{3.20} and \eqref{3.31}, we find that
\beq\lb{3.5}
\begin{aligned}
    v_i(t_4+\cdot,\cdot,\omega) &\geq 1-\eta\geq u_-(2\kappa_*+\cdot,\cdot,\omega)\qquad &&\text{on }(0,\infty)\times B_{L_3}(K_i),\\
   v_i(t_4,\cdot,\omega) &\geq 0\geq u_-(2\kappa_*,\cdot,\omega)\qquad &&\text{on }\bbR^d\setminus B_{L_3}(K_i).
\end{aligned}
\eeq
From \eqref{3.24}, \eqref{3.32}, and Lemma \ref{L.3.1} we also see that $v_i$ and $u_-$ are, respectively, a solution and a subsolution to \eqref{1.1} with reaction $f$ on $(2\kappa_*,\infty)\times (\bbR^d\setminus B_{L_3}(K_i))$. So the second claim in \eqref{3.5} and the comparison principle yield
\[
u_-(2\kappa_*+\cdot,\cdot,\omega)\leq v_i(t_4+\cdot,\cdot,\omega) \qquad\text{on }[0,\infty)\times (\bbR^d\setminus B_{L_3}(K_i)).
\] 

If $x_0\notin B_{L_3}(K_i)$, then this 
shows that
\[
T(x_0,\omega)\geq (1-{M_*}\eta)({T}_i(x_0,\omega)-t_4+2\kappa_*)+t_0.
\]
Using again ${T}_i(x_0,\omega)\leq C(1+d(x_0,S))$ (as we did above) and \eqref{3.2a}, we obtain
\beq\lb{3.10}
  {T}_i(x_0,\omega)+t_0-T(x_0,\omega) \leq t_4+{M_*}\eta  {T}_i(x_0,\omega)
\leq   C(\rho+ d (x_0,S)^{\be_1})
\eeq
for some $C>0$.
If instead $x_0\in B_{L_3}(K_i)$,  the first claim in \eqref{3.5} and \eqref{3.2a} again yield
\[
 {T}_i(x_0,\omega)+t_0-T(x_0,\omega)\leq {T}_i(x_0,\omega)\le t_4 \le C(\rho+ d (x_0,S)^{\be_1})
\]
because $x_0 \notin K_i$ (and hence $T(x_0,\omega)\ge t_0$).
\end{proof}

The last ingredient in the proof of Proposition \ref{P.3.3} is an estimate on the difference of ${T}_{{\izero}(t,\omega)}(x,\omega)$ for two different times $t$.


\begin{lemma}\lb{L.3.6}
There is $C>0$ such that for all $(x,\omega)\in\bbR^d\times\Omega$ and $t_0,t_1\ge 0$ we have
\[
\left| {T}_{{\izero}(t_1,\omega)}(x,\omega)- {T}_{{\izero}(t_0,\omega)}(x,\omega) \right| \leq C\left(\rho+|t_1-t_0|+d(x,S)^{\be_1}\right).
\]
\end{lemma}

\begin{proof}
Let $i_0:={\izero}(t_0,\omega)$ and $i_1:={\izero}(t_1,\omega)$, and then  $\omega\in E_{i_0}(t_0)\cap E_{i_1}(t_1)$.
Without loss of generality, let us assume $t_1>t_0$.  Then $u_t\ge 0$ and  \eqref{3.12} show that $K_{i_0}\subseteq K_{i_1}$.

If $x\notin  K_{i_1}$, Proposition \ref{P.3.0} yields
\[
  \left|{T}_{i_1}(x,\omega)- {T}_{i_0}(x,\omega)\right| 
  \leq |t_1-t_0|+C(\rho+d(x,S)^{\be_1}).
\]
If $x\in K_{i_0}$, then ${T}_{i_1}(x,\omega)={T}_{i_0}(x,\omega)=0$.  The result follows in either case.

Let us now assume that $x\in  K_{i_1}\backslash K_{i_0}$.  Then ${T}_{i_1}(x,\omega)=0$, while  Lemma \ref{L.2.2} shows that
\beq\lb{3.29}
{T}_{i_0}(x,\omega)\leq \kappa_0+2c_0^{-1}d(x,K_{i_0})\leq  C(1+{d}_H(K_{i_0},K_{i_1}))
\eeq
for some $C>0$. From $t_1>t_0$ and  \eqref{2.1} we also have
\[
\Gamma_{u,{{\theta^*}}}(t_0,\omega)\subseteq \Gamma_{u,{{\theta^*}}}(t_1,\omega)\subseteq B_{ \mu_*^{-1}(\theta^*)^{-m_2}(1+t_1^{\al_2})}( \Gamma_{u,1-\theta^*}(t_1,\omega) ),
\]
and Corollary \ref{C.2.1} yields
\[
\Gamma_{u,1-\theta^*}(t_1,\omega)\subseteq B_{ c_1(t_1-t_0)+\kappa_1} (\Gamma_{u,{{\theta^*}}}(t_0,\omega)).
\]
Hence there is $C>0$ such that
\[
{d}_H(\Gamma_{u,{{\theta^*}}}(t_0,\omega),\,\Gamma_{u,{{\theta^*}}}(t_1,\omega))\leq C(1+t_1^{\al_2}+t_1-t_0).
\]
Since also 
$ {d}_H(K_{{\izero}(t,\omega)},\Gamma_{u,{{\theta^*}}}(t,\omega))\leq 1$ for all $t\ge 0$ (because $\Gamma_{u,{{\theta^*}}}(t,\omega)\in P_{{\izero}(t,\omega)}$), this implies
\begin{align*}
    {d}_H(K_{i_0},K_{i_1})&\leq {d}_H(K_{i_0},\Gamma_{u,{{\theta^*}}}(t_0,\omega))+{d}_H(\Gamma_{u,{{\theta^*}}}(t_0,\omega),\Gamma_{u,{{\theta^*}}}(t_1,\omega))+{d}_H(\Gamma_{u,{{\theta^*}}}(t_1,\omega),K_{i_1})\\
    &\leq C(1+t_1^{\al_2}+t_1-t_0).
\end{align*}
This and \eqref{3.29} yields the claim.
\end{proof}

\subsection{Proof of Proposition \ref{P.3.3}}

If $d(x,S)\le \rho$, then \eqref{3.4} holds for all $\lambda>0$ as long as $C_0\ge 2 (\kappa_0+\frac {2}{c_0})^2 (\ln 2)^{-1}$.
This is because $T(x,\cdot)\le \kappa_0+\frac {2}{c_0}d(x,S)$ by Lemma \ref{L.2.2}, so one only needs to consider $\lambda\le (\kappa_0+\frac {2}{c_0})(1+d(x,S))$, for which the right-hand side of \eqref{3.4} with this $C_0$ is at least 1 due to $\rho\ge \frac 12(1+d(x,S))$.  


It therefore suffices to consider the case $d(x,S)> \rho$.  In particular, we have $x\notin K_0$ due to $\rho\ge 1$.
Let us fix any such $x$
and consider the $\calG_t$-adapted martingale $\{X_t\}_{t\geq 0}$ defined by
\[
X_t=X_t(\omega):=\bbE[T(x,\cdot\,)\,|\,\calG_t](\omega).
\]
We want to apply Azuma's inequality to it, which means that we need to obtain a suitable $\omega$-independent bound on $|X_t-X_s|$ for any  $t> s>0$ (which we fix).
Using
\begin{align*}
    &X_t=\bbE[T(x,\cdot\,)\chi_{ F_{s,x}}\,|\,\calG_t]+\bbE[T(x,\cdot\,)\chi_{ F_{s,x}^c}\,|\,\calG_t],\\
    &X_s=\bbE[T(x,\cdot\,)\chi_{ F_{s,x}}\,|\,\calG_s]+\bbE[T(x,\cdot\,)\chi_{ F_{s,x}^c}\,|\,\calG_s],
\end{align*}
we find from Lemma \ref{L.3.3} (recall that $F_{s,x}\subseteq F_{t,x}$) that there is $C>0$ such that
\beq\lb{3.22}
\begin{aligned}
    |X_t-X_s|&\leq  |\bbE[T(x,\cdot\,)\chi_{ F_{s,x}^c}\,|\,\calG_t]-\bbE[T(x,\cdot\,)\chi_{ F_{s,x}^c}\,|\,\calG_s]|+C\\
    &=  \left|\Sigma_{x\notin K_i}\bbE[T(x,\cdot\,)\chi_{E_i(s)} \,|\,\calG_t]-\Sigma_{x\notin K_i}\bbE[T(x,\cdot\,)\chi_{E_i(s)}\,|\,\calG_s]\right|+C.
\end{aligned}
\eeq
Here we used that $\omega\notin F_{s,x}$ precisely when $x\notin K_{{\izero}(s,\omega)}$, and the sums are over all $i\in\bbN_0$ such that $x\notin K_i$.
From Proposition \ref{P.3.0} with $s$ in place of $t$ we have
\[
\begin{aligned}
& \left|\Sigma_{x\notin K_i}\bbE[T(x,\cdot\,)\chi_{E_i(s)}\,|\,\calG_t]-\Sigma_{x\notin K_i}\bbE[T(x,\cdot\,)\chi_{E_i(s)}\,|\,\calG_s]\right|\\
& \leq \left|\Sigma_{x\notin K_i}\bbE[{T}_i(x,\cdot\,)\chi_{E_i(s)}\,|\,\calG_t]-\Sigma_{x\notin K_i}\bbE[{T}_i(x,\cdot\,)\chi_{E_i(s)}\,|\,\calG_s]\right|+C\left(\rho+d(x,S)^{\be_1}\right)\\
& = \left|\Sigma_{i\in\bbN_0}\bbE[{T}_i(x,\cdot\,)\chi_{E_i(s)}\,|\,\calG_t]-\Sigma_{i\in\bbN_0}\bbE[{T}_i(x,\cdot\,)\chi_{E_i(s)}\,|\,\calG_s]\right|+C\left(\rho+d(x,S)^{\be_1}\right).
\end{aligned}
\]
The last equality holds because ${T}_i(x,\omega)=0$ when $x\in K_i$. 
Since
\[
\bbE[{T}_i(x,\cdot\,)\chi_{E_i(s)}\,|\,\calG_t]=\Sigma_{j\in\bbN_0} \bbE[{T}_i(x,\cdot\,)\chi_{E_i(s)\cap E_j(t)}\,|\,\calG_t]
\]
(recall that $E_i(s),E_j(t)\in \calG_t$) and  Lemma \ref{L.3.6} yields
\[
   |\Sigma_{i,j}\bbE[{T}_i(x,\cdot\,)\chi_{E_i(s)\cap E_j(t)}\,|\,\calG_t]-\Sigma_{i,j}\bbE[{T}_j(x,\cdot\,)\chi_{E_i(s)\cap E_j(t)}\,|\,\calG_t]|
    \leq C(\rho+|t-s|+d(x,S)^{\be_1}),
\]
it follows that with some $C>0$ and $C_{t,s,x}^\rho:=C(\rho+|t-s|+d(x,S)^{\be_1})$, we have
\beq\label{3.30}
\begin{aligned}
    |X_t-X_s|
    &\leq \left|\Sigma_{i,j}\bbE[{T}_j(x,\cdot\,)\chi_{E_i(s)\cap E_j(t)}\,|\,\calG_t]-\Sigma_{i}\bbE[{T}_i(x,\cdot\,)\chi_{ E_i(s)}\,|\,\calG_s]\right|+C_{t,s,x}^\rho\\
    &= \left|\Sigma_{j}\bbE[{T}_j(x,\cdot\,)\chi_{ E_j(t)}\,|\,\calG_t]-\Sigma_{i}\bbE[{T}_i(x,\cdot\,)\chi_{ E_i(s)}\,|\,\calG_s]\right|+C_{t,s,x}^\rho.
\end{aligned}
\eeq

We now claim that for any $i\in\bbN_0$ we have 
\beq\lb{3.15}
\bbE[{T}_i(x,\cdot\,)\chi_{ E_i(s)}\,|\,\calG_s]=\bbE[{T}_i(x,\cdot\,)]\chi_{ E_i(s)}.
\eeq
Since $ E_i(s)\in \calG_s$, to prove this,
we only need to show that
\beq\lb{3.16}
\bbE[{T}_i(x,\cdot\,)\chi_{A}]=\bbE[{T}_i(x,\cdot\,)]\bbP(A)
\eeq
for each $A\in\calG_s$ such that $A\subseteq E_i(s)$.  But then $A\in \calE(K_i)$ by \eqref{3.333}, so \eqref{3.16} follows from   ${T}_i(x,\cdot)$ being independent of $\calE(K_i)$.
 
Similarly to \eqref{3.15},  we also have
\[
\bbE[{T}_j(x,\cdot\,)\chi_{ E_j(t)}\,|\,\calG_t]=\bbE[{T}_j(x,\cdot\,)]\chi_{ E_j(t)}
\]
for any $j\in\bbN_0$. Then \eqref{3.30} becomes
\begin{align*}
|X_t-X_s|&\leq  \left|\Sigma_{j}\bbE[{T}_j(x,\cdot\,)]\chi_{ E_j(t)}-\Sigma_{i}\bbE[{T}_i(x,\cdot\,)]\chi_{ E_i(s)}\right|+C_{t,s,x}^\rho\\
&\leq \Sigma_{i,j}\left|\bbE[{T}_j(x,\cdot\,)]-\bbE[{T}_i(x,\cdot\,)]\right|\chi_{ E_j(t)\cap E_i(s)}+C_{t,s,x}^\rho.
\end{align*}
Lemma \ref{L.3.6} now shows that there is $C>0$ such that for all $\omega\in\Omega$ we have
\beq\lb{3.18}
|X_t(\omega)-X_s(\omega)|\leq C(\rho+|t-s|+d(x,S)^{\be_1}).
\eeq

By Lemma \ref{L.2.2}, we have $F_{\tau_x,x}=\Omega$ when $\tau_x:=\kappa_0+\frac{2}{c_0}d(x,S)$. So with $C$ from Lemma~\ref{L.3.3},
\beq\lb{3.18a}
X_{t}=T(x,\cdot)  \qquad\text{for all $t\geq \tau_x+C$}.
\eeq
Let 
$
\tau:=\rho+d(x,S)^{\be_1}$ and let $N$ be the smallest integer such that $N\tau\ge \tau_x+C$.
Then there 
is $C_1>0$ such that $N\le  C_1 d(x,S)(\rho+d(x,S)^{\be_1})^{-1}$ (recall that $d(x,S)>\rho\ge 1$).
It follows from \eqref{3.18} that for  $k=0,\dots, N-1$ we have (uniformly in $\omega\in\Omega$)
\beq \lb{3.555}
    |X_{(k+1)\tau}-X_{k\tau}|\leq C(\rho+d(x,S)^{\be_1}).
\eeq

Now Azuma's inequality (Lemma \ref{L.3.5}), $X_0=\bbE[T(x,\cdot\,)\,|\,\calG_0]$, and  \eqref{3.18a} with $t=N\tau$ yield for any $\lambda\geq 0$ (with $C$ changing from line to line),
\begin{equation}\label{e.3.7}
\begin{aligned}
    \bbP \left[ \big| T(x,\cdot)-\bbE[T(x,\cdot\,)\,|\,\calG_0 ] \big| \geq \lambda \right]
    &\leq 2\exp\left(\frac{-\lambda^2}{2CN (\rho+d(x,S)^{\be_1})^2}\right)\\
    &\leq 2\exp\left(\frac{-\lambda^2}{C d(x,S) (\rho+d(x,S)^{\be_1})}\right).
\end{aligned}
\end{equation}
Since $F_{0,x}=\emptyset$ (because $x\notin K_0$),
Proposition \ref{P.3.0} with $t=0$ (and ${\izero}(0,\cdot)\equiv 0$) yields
\[
|T(x,\omega)-{T}_0(x,\omega)|\leq C'(\rho+d(x_0,S)^{\be_1})
\]
for some $C'>0$ and all $\omega\in\Omega$.  This and ${T}_0(x,\cdot)$ being independent of  $\calE(K_0)=\calG_0$ yield
\begin{align*}
    \left|\bbE[T(x,\cdot\,)\,|\,\calG_0]-\bbE[T(x,\cdot\,)]\right|
    &\leq \left|\bbE[{T}_0(x,\cdot\,)\,|\,\calG_0]-\bbE[{T}_0(x,\cdot\,)]\right|+C'(\rho+d(x,S)^{\be_1})\\
    &=C'(\rho+d(x,S)^{\be_1}).
\end{align*}
Hence, from \eqref{e.3.7} we obtain for any $\lambda> 0$,
\[
\bbP  \left[ \big| T(x,\cdot\,)-\bbE[T(x,\cdot\,)] \big| \geq\lambda + C'(\rho+d(x,S)^{\be_1}) \right]  
  \leq 2\exp  \left(\frac{-\lambda^2}{C d(x,S)(\rho+d(x,S)^{\be_1})}\right).
\]
So for all $\lambda\geq C'(\rho+d(x,S)^{\be_1})$ we have
\[
\bbP\left[ \big| T(x,\cdot\,)-\bbE[T(x,\cdot\,)]\big|\geq 2\lambda\right]
    \leq 2\exp  \left(\frac{-\lambda^2}{C d(x,S)(\rho+d(x,S)^{\be_1})}\right),
\]
which yields \eqref{3.4} for all $\lambda\geq 2C'(\rho+d(x,S)^{\be_1})$ as long as $C_0\ge 4C$. But \eqref{3.4} also holds for all $\lambda\leq 2C'(\rho+d(x,S)^{\be_1})$ as long as $C_0\ge 8(C')^2(\ln 2)^{-1}$ because $d(x,S)>\rho\ge 1>\be_1$ (and so the right-hand side of  \eqref{3.4} is $\ge 1$).
This finishes the proof.

\subsection{Extension to half-spaces.}  We now extend Proposition \ref{P.3.3} to half-spaces, denoting  
\[
\calH_e^-:=\{x\in\bbR^d\, | \,x\cdot e\le  0\}
\] 
for $e\in \bbS^{d-1}$.
This means that we need to enlarge $\calU_f$ to include solutions initially approximating characteristic functions of half-spaces, with \textbf{(H2')} extending as well.

\begin{lemma}\lb{L.4.0'}
\textbf{(H2')} implies \textbf{(H2')} with $\calU_f'$ in place of $\calU_f$, with unchanged values of all the $\sup_{u\in\calU_f}$  and $\inf_{u\in\calU_f}$ in \eqref{2.2}, and with
$\calU_f'$ defined as $\calU_f$ but including the initial functions $u_{0,k}$ from $\calU_f$ as well as all locally uniform limits of their translations.
(These are then functions $u_{0,S}$  satisfying Lemma \ref{L.2.4} for all balls $S=B_k(y)$ and all half-spaces $S=\calH_e^-+le$, due to well-known elliptic regularity estimates.)

In particular, Lemmas \ref{L.3.1} and  \ref{L.4.3} hold with $\calU_{f_1}$ replaced by $\calU_{f_1}'$.
\end{lemma}

\begin{proof}
Stationarity of $f$ again shows that adding translations of the $u_{0,k}$ to $\calU_f$ does not change any of the $\sup$ or $\inf$.
Well known parabolic regularity estimates now show that the $\sup$ and $\inf$ also remain unchanged when we add locally uniform limits of these translations  to $\calU_f$.
The proofs of Lemmas \ref{L.3.1} and  \ref{L.4.3} then extend to $\calU_{f_1}'$ in place of $\calU_{f_1}$ without change.

We note that  the elliptic and parabolic regularity  (Krylov-Safonov and Schauder) estimates used here can be found in \cite[Theorem 4.6]{gilbarg2015elliptic},  \cite[Theorem 4.1]{krylov1980certain}, and  \cite[Theorem 8.6.1]{krylov1996lectures}.
\end{proof}

\begin{proposition}\lb{P.3.4}
Proposition \ref{P.3.3} holds for
$S$ being either any ball $B_k(y)$ with $(k,y)\in\bbN\times\bbR^d$ or any half-space $\calH_e^-+le$ with $(e,l)\in \bbS^{d-1}\times\bbR$ (with the functions $u(\cdot,\cdot,\cdot;S)$ from $\calU_f'$).
\end{proposition}


\begin{proof}
The claim for balls is immediate from stationarity.  For the same reason, in the half-space case we only need to consider $l=0$.  Hence let $S=\calH_e^-$ for some $e\in \bbS^{d-1}$.

For each $k\in\bbN$, let $S_k:=B_k(-ke)$.  Then $\lim_{k\to\infty} d(x,S_k)= d(x,S)$ for each $x\in\bbR^d$,
so by Proposition \ref{P.3.3}  (with $C_0$ independent of $S$), it suffices to show
\[
\limsup_{k\to\infty}\, \sup_{\omega\in \Omega} |T(x,\omega;S)-T(x,\omega;S_{k})|\leq C
\]
for some $C>0$ (depending only on \eqref{const}, as always) and any $x\in\bbR^d$. 

Let $u_\omega:=u(\cdot,\cdot,\omega;S)$ and $u_{k,\omega}:=u(\cdot,\cdot,\omega;S_k)$ for each $(k,\omega)\in\bbN\times\Omega$, and $C:=\kappa_0+ \frac{2R_0}{ c_0}$.   We then have $u_\omega(C,\cdot)\ge u_{k,\omega} (0,\cdot)$ by Lemma \ref{L.2.2} and \eqref{2.7}, so $T(\cdot,\omega;S)\leq T(\cdot,\omega;S_k) + C$.


Similarly, we have $u_{k,\omega}(C+1,\cdot)\ge u_{\omega} (0,\cdot)$ in $B_{k^{1/2}}(0)$.  Then from the last claim in Lemma~\ref{L.3.1} with $(\eta,R)=(0,k^{1/2}-|x|)$ we obtain $T(x,\omega;S_k)\leq T(x,\omega;S) + C+1+2\kappa_*+\kappa_0$ whenever
$k^{1/2}\ge |x|+ D_2(1+\kappa_0+ \frac 2{ c_0} (d(x,S)+1))$,
because then $d(x,S_k)\le d(x,S)+1$ due to $|x|\le k^{1/2}$, so $k^{1/2}-|x|\ge D_2(1+T(x,\omega;S_n))$ by Lemma \ref{L.2.2}.
\end{proof}

\section{Fluctuations for General Reactions}\lb{S4}

First, we show that \textbf{(H4')} implies  \textbf{(H2')}.

\begin{lemma}\lb{L.4.0}
\textbf{(H4')} implies that $f$ also satisfies \textbf{(H2')}  (possibly after removing from $\Omega$ a measure-zero set that is invariant with respect to the group $\{\Upsilon_y\}_{y\in\bbR^d}$, which we then do).  
\end{lemma}

\begin{proof}
For each  $k\in\bbN$, let $u_{0,k,n}$ be the initial datum for $B_k(0)$ that enters in the definition of $\calU_{f_n}$ (these might in principle be different for distinct $n$, as we do not assume them to be those from the proof of Lemma~\ref{L.2.4}).  By \eqref{2.6}, \eqref{2.7'},  and elliptic regularity estimates, there is a subsequence $\{n_j\}_{j\in\bbN}$ such that $u_{0,k,n_j}$ converge uniformly to some $u_{0,k}$ satisfying \eqref{2.6} and \eqref{2.7'} as $j\to\infty$.  Since also \textbf{(H2')} holds uniformly for $f_{n}$,
and the Borel-Cantelli Lemma shows that some subsequence of $f_{n_j}$ converges to $f$ locally uniformly on $\bbR^d$ for almost every $\omega\in\Omega$ (and if this holds for some $\omega$, then it obviously also holds for $\Upsilon_y \omega$ with any $y\in\bbR^d$),
it follows that \textbf{(H2')} also holds for $f$, with the above $u_{0,k}$ for each $k\in\bbN$ and  after removal of a $\{\Upsilon_y\}$-invariant measure-zero set from $\Omega$.
\end{proof}

We now extend Proposition \ref{P.3.4} to reactions $f$ satisfying \textbf{(H3)} and \textbf{(H4')}.
Recall \eqref{d.3.2} and that constants with $C$ in them only depend on \eqref{const} unless explicitly stated otherwise.

\begin{proposition}\lb{P.3.1}
Let $f$ satisfy \textbf{(H3)} and \textbf{(H4')}, and with $\be_1$ from \eqref{3.13} let
\beq\lb{3.1}
\be_3:=\max\left\{ \be_1 , \, \frac{m_3}{m_3+2m_4} ,\frac{2d+2}{2d+2+m_4'} \right\}
 \qquad (\in(0,1))
\eeq
and
\beq\lb{3.1q}
\nu(a):= \sup_{n\ge \max\{ \lceil a \rceil,n_4\}} \bbP\left( \sup_{|z|< n^{1+m_4'}} \, \sup_{u\in [0,1]} |f_{n}(z,u,\cdot) - f(z,u,\cdot) |  > \al_4 n^{-m_4} \right)
\eeq
for all $a\ge 0$   (then $\nu(a)\le \max\{ \lceil a \rceil,n_4\}^{-(2d+1+m_4')}$ by \textbf{(H4')}).
There is $C_0'\geq 1$ such that  for any $S$ (and $u$) from Proposition \ref{P.3.4}, $x\in\bbR^d$, and $\lambda>0$ we have
\[
\bbP\left( |T(x,\cdot\,;S)-\bbE [T(x,\cdot\,;S)]|\geq \lambda\right)\leq 2\exp \left( \frac{ -\lambda^2}{C_0'(1+d(x,S)^{1+\be_3})} \right) 
+ \nu \left(d(x,S)^{\beta_3} \right).
\]
\end{proposition}

\noindent {\it Remark.}  Having this  and Lemma \ref{L.4.0}, we will not need to use \textbf{(H3)} and \textbf{(H4')} again.

\begin{proof}
Let us assume without loss that $\al_3\le 1$ in \textbf{(H3)}.
Lemma \ref{L.4.0} shows that $f$ also satisfies \textbf{(H2')}, and then Lemma \ref{L.4.0'} shows that \textbf{(H2')} holds 
with $\calU_f'\cup\bigcup_{n\ge n_4} \calU_{f_n}'$ in place of $\calU_f$.

For each $\omega\in\Omega$ and $S$ either a ball or a half-space, let $u(\cdot,\cdot,\omega,S)$ be the solution from $\calU_f'$ corresponding to $(\omega,S)$, and for each $n\geq n_4$, let $u_n(\cdot,\cdot,\omega,S)$ be the analogous solution from $\calU_{f_n}'$. (Note that we do not require the initial data for $u$ and $u_n$ to be the same, although they may be.)
Also let
\[
T_n(x,\omega;S):=\inf\{t\geq 0\,|\, u_n(t,x,\omega;S)\geq 1-\theta^*\}.
\]
Then \textbf{(H4')} and Proposition \ref{P.3.4} show for all $n\ge n_4$,
\begin{equation}
    \label{4.7}\bbP\left( |T_n(x,\cdot\,;S)-\bbE [T_n(x,\cdot\,;S)]|\geq \lambda\right)\leq 2\exp \left( -\frac{\lambda^2}{C_0 (1+d(x,S)) (n+d(x,S)^{\be_1})} \right).
\end{equation}
Let 
$C_1\geq n_4$ be such that $(\frac{\al_4}{\al_3})^{1/m_3}C_1^{-m_4/m_3}\leq \frac{1}{2}\min\{\theta^*,M_*^{-1}\}$.

Stationarity of $f,f_n$ shows that the definition of $\nu(a)$ above is unchanged when in it we replace $|z|$ by $|z-x|$, for any $x\in\bbR^d$.  So for each $x$ and $n\ge n_4$, there is $\Omega_{n,x}\subseteq\Omega$ such that $\bbP(\Omega_{n,x})\ge  1-\nu(n)$ and 
\[
|f(z,u,\omega)-f_n(z,u,\omega)|\leq \al_4 n^{-m_4}\qquad\text{for all }(z,u,\omega)\in B_{n^{1+m_4'}}(x)\times[0,1]\times \Omega_{n,x}.
\]
Since $f$ satisfies \textbf{(H3)}, and $u(\cdot,\cdot,\omega;S)\in \calU_f$ and $u_n(\cdot,\cdot,\omega;S)\in \calU_{f_n}$  for each $\omega\in\Omega_{n,x}$, Lemma~\ref{L.4.3} (see Lemma \ref{L.4.0'}) applied twice with $\eta=(\frac{\al_4}{\al_3} n^{-m_4})^{1/m_3}$,  $y=x$, $R=n^{1+m_4'}$, $t_0=\kappa_0+2c_0^{-1}R_0$, and $(f_1,f_2,u_1,u_2)$ being  
\[
(f,f_n,u(\cdot,\cdot,\omega;S),u_n(\cdot,\cdot,\omega;S))
\qquad
\text{ and }\qquad
(f_n,f,u_n(\cdot,\cdot,\omega;S),u(\cdot,\cdot,\omega;S)),
\] 
respectively, 
yields for some $C>0$ and  all $n\ge \max\{C_1, [D_2(1+\kappa_0+\frac{2}{c_0}d(x,S))]^{1/(1+m_4')}\}$,
\beq\lb{4.40}
|T_n(x,\cdot;S)-T(x,\cdot;S)|\leq Cn^{-\frac{m_4}{m_3}}\left(1+d(x,S)\right) + C.
\eeq
Here we also used Lemma \ref{L.2.2} to show that $T_n(x,\cdot;S)$ and $T(x,\cdot;S)$ are at most $\kappa_0 + \frac{2}{c_0}d(x,S)$.

Let now
 $n$ be the smallest integer such that 
\beq\lb{4.333}
n\geq \max\left\{C_1,\,d(x,S)^{\be_3}, [D_2(1+\kappa_0+2c_0^{-1}d(x,S))]^{1/(1+m_4')} \right\}.
\eeq
Since $\be_3\ge\frac 1{1+m_4'}$,  there is $C_2>0$ such that with $C$ from \eqref{4.40} we have (uniformly in $x,S$)
\begin{align*}
(1+d(x,S))(n+ d(x,S)^{\be_1}) &\leq C_2(1+d(x,S)^{1+\be_3}),\\
4C{n}^{-\frac{m_4}{m_3}} (1+d(x,S)) + 4C&\leq C_2(1+d(x,S)^{1-\be_3\frac{m_4}{m_3}})\leq C_2(1+d(x,S)^{\frac{1+\be_3}2}).
\end{align*}
If now $\lambda\geq C_2(1+d(x,S)^{(1+\be_3)/2})$, then \eqref{4.7} and \eqref{4.40} imply
\begin{align*}
    \bbP\left( |T(x,\cdot\,;S)-\bbE [T(x,\cdot\,;S)]|\geq\lambda\right)&\leq\bbP\left( |T_n(x,\cdot\,;S)-\bbE [T_n(x,\cdot\,;S)]|\geq \frac{\lambda}{2}\right) + \bbP(\Omega\setminus \Omega_{n,x}) \\
    &\leq  2\exp \left(\frac{-\lambda^2}{4C_0C_2(1+d(x,S)^{1+\be_3})} \right) + \nu \left(d(x,S)^{\be_3} \right).
 \end{align*}
Hence the result holds with $C_0':=\max\{ 4C_0C_2, 2C_2^2 (\ln 2)^{-1}\}$, because then it obviously also holds for any 
$\lambda\leq C_2(1+d(x,S)^{(1+\be_3)/2})$.
\end{proof}
   
\section{Convergence of the Mean Propagation Speeds} \lb{S5}

We now consider \eqref{e.3.1} with $S=\calH_e^-+le$ for any $e\in \bbS^{d-1}$ and $l\in\bbR$, that is,
\begin{equation} \lb{4.1}
\begin{aligned}
& u_t=\Delta u+f(x,u,\omega)\qquad && \text{ on }(0,\infty)\times \bbR^d,\\
& u(0,\cdot,\omega;\calH_e^-+le)=u_{0,\calH_e^-+le}\qquad && \text{ on } \bbR^d.
\end{aligned}
\end{equation} 
Here $u_{0,\calH_e^-+le}$ 
is the initial data used in the definition of $\calU_f'$ in Lemma \ref{L.4.0'} (note that $f$ satisfies \textbf{(H2')} in both cases under consideration, due to Lemma \ref{L.4.0}).  Hence, $u(\cdot,\cdot,\omega;\calH_e^-+le)\in\calU_f'$
 for all $(e,l,\omega)\in \bbS^{d-1}\times\bbR\times\Omega$ and \eqref{2.1} holds for it.

We will now prove that $\frac 1l \bbE \left[ T(le,\cdot\,;\calH_e^-)\right]$
converges as $l\to\infty$, with $T(le,\cdot\,;\calH_e^-)$ from \eqref{d.3.2} (stationarity shows that the expectation is the same if $le$ is replaced by any $y_l$ with $y_l\cdot e=l$).  Note that the reciprocal of this limit can then be considered the asymptotic mean speed of propagation of the solutions $u$ in direction $e$
(this mean is technically harmonic).

We also note that all constants in this section will be uniform in $e$, and
recall that all constants with $C$ in them 
 depend on \eqref{const}, with any other dependence explicitly indicated, and may vary from line to line.


\begin{proposition}\lb{P.4.5}
For each $e\in \bbS^{d-1}$ there is $ \bar{T}(e)\in [\frac 1{c_1},\frac 1{c_0}]$ (depending also on $f$) and for each $\delta>0$ there is $C_\delta\geq 1$  such that the following hold.
If $f$ satisfies \textbf{(H2')} and has range of dependence at most $\rho\in [1,\infty)$, 
then with $\be_1$ from \eqref{3.13} 
we have for all $l\geq 1$,
\[
\left|\frac {\bbE[T(le,\cdot\,;\calH^-_e)]} {l} -\bar{T}(e) \right|\leq C_\delta\sqrt\rho\, l^{-1+\frac12 {(1+\be_1)}+\delta}.
\]
If instead $f$ satisfies \textbf{(H3)} and  \textbf{(H4')},
then with $\be_3<1$ from \eqref{3.1} we have for all $l\geq 1$, 
\[
\left|\frac {\bbE[T(le,\cdot\,;\calH^-_e)]} {l} -\bar{T}(e)\right|\leq C_\delta\, l^{-1+\frac12 {(1+\be_3)}+\delta}.
\]
%
%
\end{proposition}

We will fix $e$ in the rest of this section and, for the sake of convenience, we will sometimes (but not always) drop $\calH_e^-$ from the notation in \eqref{e.3.1} and \eqref{d.3.2} when $S=\calH_e^-$.  Hence we let
\beq\lb{5.444}
T(x,\omega):=T(x,\omega;\calH_e^-)\qquad\text{and}\qquad u(t,x,\omega):=u(t,x,\omega;\calH_e^-) .
\eeq
We will also prove both claims in Proposition \ref{P.4.5} at the same time, with the notation
\beq\lb{b.1'}
    \beta:=\frac{1+\be_1}{2} \qquad\text{and}\qquad {\bar C}_\rho:=C\sqrt\rho \qquad\text{and}\qquad \phi(l):=0
\eeq
if $f$ satisfies \textbf{(H2')} and has range of dependence at most $\rho\in[1,\infty)$, and
\beq\lb{b.1''}
    \beta :=\frac{1+\be_3}{2}  \qquad\text{and}\qquad {\bar C}_\rho:=C \qquad\text{and}\qquad \phi(l):=\nu(l^{\be_3})
\eeq
 if $f$ satisfies \textbf{(H3)} and \textbf{(H4')}, with $\nu$ from \eqref{3.1q}.  Here again, $C\ge 1$ will be a constant depending only on \eqref{const}, which may vary from line to line.  In particular, Propositions \ref{P.3.4} and \ref{P.3.1} show that in both cases we have for all $e\in\bbS^{d-1}$, $\lambda>0$, $x\in\bbR^d$ with $x\cdot e\geq 1$ (and some $C>0$ defining ${\bar C}_\rho$),
\beq\lb{5.7}
\bbP\left( \big| T(x,\cdot\,;\calH_e^-)-\bbE [T(x,\cdot\,;\calH_e^-)] \big| \geq \lambda\right)\leq 2\exp \left( -{\bar C}_\rho^{-2}\,\lambda^2 \,(x\cdot e)^{-2\beta} \right) + \phi(x\cdot e).
\eeq

We start with the following simple result.

\begin{lemma}\label{L.4.1}
If $x\cdot e\le l$ for some $(e,l,x)\in \bbS^{d-1}\times\bbR\times\bbR^d$, then
\[
\bbE \left[ T(x,\cdot\,;\calH_e^-)\right] \le \bbE \left[ T(l e,\cdot\,;\calH_e^-)\right] + \kappa_0+2c_0^{-1}R_0. 
\]
\end{lemma}

\begin{proof}
Since $u(0,\cdot,\omega;\calH_e^-)\leq u(\tau_0,\cdot,\omega;\calH_e^-+le-x)$ for $\tau_0:=\kappa_0+\frac {2R_0}{c_0}$ by Lemma \ref{L.2.2} and $x\cdot e\le l$,
we have 
\[
T(le,\cdot\,;\calH_e^-)\geq T( le,\cdot\,;\calH_e^-+le-x)-\tau_0.
\]
Therefore,
\[
\bbE[T(le,\cdot\,;\calH_e^-)]\geq \bbE[ T( le,\cdot\,;\calH_e^-+le-x)] -\tau_0=\bbE[ T( x,\cdot\,;\calH_e^-)]-\tau_0
\]
because $f$ is stationary (if we assume \textbf{(H4')}, this follows from Lemma \ref{L.4.0})
\end{proof}


The next result is an immediate consequence of Lemma \ref{L.2.2} and Corollary \ref{C.2.1}.

\begin{lemma} \lb{L.5.3}
There is $t_0\geq 1$, depending only on \eqref{const}, such that for all $t\geq t_0$ and $\omega\in\Omega$,
\[
u(t,x,\omega;\calH_e^-)\geq 1-\theta^* \qquad \text{when $x\cdot e\leq \frac{c_0t}{2}$},
\]
\[
u(t,x,\omega;\calH_e^-)< 1-\theta^* \qquad \text{when $x\cdot e\geq {2c_1t}$}.
\]
In particular, for all $x\in\bbR^d$ with $x\cdot e\geq l_0:=2c_1 t_0$ and $\omega\in\Omega$ we have
\beq\lb{4.18}
T(x,\omega;\calH_e^-)\in \left[ \frac{x\cdot e}{2c_1}, \frac{2\,x\cdot e}{c_0} \right].
\eeq
\end{lemma}


In order to prove Proposition \ref{P.4.5}, it will be necessary to simultaneously prove it for $T(x,\cdot\,;\calH^-_e)$ with other points $x$ satisfying $x\cdot e=l$.  The Infinite Monkey ``Theorem'' shows that in dimensions $d\ge 2$, this cannot involve all the points in the unbounded set $\{x\cdot e=l\}$, but we will be able to include all such points with $|x|\le O(l)$ (i.e., within a ball centered at $le$ and with linearly-in-time growing radius due to \eqref{4.18}).  This will be sufficient thanks to the speed of propagation of perturbations of solutions being finite (see Lemma \ref{L.2.1}).

This and Lemma \ref{L.5.3} motivate the definitions of the cylinders
\begin{align*}
\calC_e^-(R,l) &:=\left\{x\in\bbR^d\, \bigg| \, x\cdot e\in \left[\frac{c_0}{4c_1}l, l\right] \text{ and }\, |x-(x\cdot e)e|\leq R\right\}, \\
\calC_e^+(R,l) &:=\left\{x\in\bbR^d\, \bigg| \, x\cdot e\in \left[l, \frac{4c_1}{c_0}l \right] \text{ and }\, |x-(x\cdot e)e|\leq R\right\}
\end{align*}
and of the corresponding times
\[
\begin{aligned}
   T^-_e(R,l,\omega):&=\inf\left\{t\geq 0\, | \, u(t,\cdot,\omega;\calH_e^-)\geq 1-\theta^*\, \text{ on }\calC_e^-(R,l)\right\}\\
   &=\sup\left\{ T(x,\omega;\calH_e^-)\, | \, x\in \calC_e^-(R,l)\right\},\\
   T^+_e(R,l,\omega):&=\sup\left\{t\geq 0\, | \, u(t,\cdot,\omega;\calH_e^-)< 1-\theta^*\, \text{ on } \calC_e^+(R,l)\right\}\\
     &=\inf\left\{ T(x,\omega;\calH_e^-)\, | \, x\in \calC_e^+(R,l)\right\}.
\end{aligned}
\]
Obviously {$T^+_e(R,l,\cdot)\leq T^-_e(R,l,\cdot)$} because $u_t(t,\cdot,\cdot;\calH_e^-)> 0$ for all $t>0$.

Our next result shows that means of these times are sufficiently close to $\bbE\left[ T(le,\cdot\,;\calH_e^-)\right]$.

\begin{lemma}\lb{L.4.2}
There is $C>0$, and for each $\delta>0$ there is $C_\delta\geq 1$, such that with $l_0\geq 1$ from Lemma \ref{L.5.3} 
we have for all $ l\geq l_0$ and $R\ge 0$ (with ${\bar C}_\rho$ from \eqref{b.1'} resp.~\eqref{b.1''}),
\beq \lb{4.2}
\begin{aligned}
   \bbE\left[ T^-_e(R,l,\cdot\,) \right] 
   & \leq \bbE\left[ T(le,\cdot\,;\calH_e^-)\right]+ {\bar C}_\rho (C_\delta+\max\{R,l\}^{\delta}\,l^{\beta})  +C\max\{R,l\}^{d-1}l^2\, \phi \left( \frac{c_0}{4c_1} l \right), \\
    \bbE\left[T^+_e(R,l,\cdot\,) \right] 
    & \geq \bbE\left[ T(le,\cdot\,;\calH_e^-)\right] - {\bar C}_\rho (C_\delta+\max\{R,l\}^{\delta}\,l^{\beta}) -C\max\{R,l\}^{d-1}l^2 \, \phi(l).
\end{aligned}
\eeq
\end{lemma}

\begin{proof}
The definition of $T^\pm_e(R,l,\omega)$ shows that it suffices to consider the case $R\ge l\ge l_0$.

For any $\lambda>0$ we have
\[
     \bbP\left[ T^-_e(R,l,\omega)-\bbE\left[ T(le,\cdot\,) \right]\geq \lambda\right]
    = \bbP\left[ \sup\left\{T(x,\omega)-\bbE\left[ T(le,\cdot\,) \right]\, | \, x\in \calC_e^-(R,l)\right\}\geq \lambda\right].
\]
Since Lemma \ref{L.2.2} yields
\[
|T(x,\omega)-T(y,\omega)|\leq \kappa_0+\frac{2|x-y|}{c_0},
\]
for any $\lambda\geq 4\kappa_0+\frac{4(\sqrt{d}+R_0)}{c_0}=:C_1$ 
we obtain
\begin{align*}
\bbP\left[ T^-_e(R,l,\omega)-\bbE\left[ T(le,\cdot\,) \right]\geq \lambda\right]
    & \leq \bbP\left[ \sup\left\{T(x,\omega)-\bbE\left[ T(le,\cdot\,) \right]\, | \, x\in \calC_e^-(R,l)\cap \bbZ^d\right\}\geq \frac{\lambda}{2}\right]\\
     &\leq  \sum_{x\in \calC_e^-(R,l)\cap \bbZ^d}\bbP\left[ T(x,\omega)-\bbE\left[ T(le,\cdot\,) \right]\geq \frac{\lambda}{2} + \kappa_0+\frac {2R_0}{c_0}\right].
\end{align*}
Since $x\in \calC_e^-(R,l)\cap \bbZ^d$ implies $x\cdot e\leq l$ and $f$ is stationary, Lemma \ref{L.4.1} now yields
\[
\bbP\left[ T^-_e(R,l,\omega)-\bbE\left[ T(le,\cdot\,) \right]\geq \lambda\right] \le \sum_{x\in \calC_e^-(R,l)\cap \bbZ^d}\bbP\left[ T(x,\omega)-\bbE\left[ T(x,\cdot\,) \right]\geq \frac{\lambda}{2}\right].
\]
The number of terms in this sum is bounded by $C_2 R^{d-1}l$ for some $C_2>0$.
Since each $x\in \calC_e^-(R,l)$ satisfies $d(x,\calH_e^-)\geq \frac{c_0}{4c_1}l$, by \eqref{5.7} we have for each $\lambda\geq C_1$,
\[
\bbP\left[ T^-_e(R,l,\omega)-\bbE\left[ T(l e,\cdot\,) \right]\geq \lambda\right]\leq C_2R^{d-1}l\exp \left( -{\bar C}_\rho^{-2}\,{\lambda^2}\,l^{-2\beta} \right) + C_2R^{d-1}l \,\phi \left( \frac{c_0}{4c_1} l \right).
\]
Moreover, this probability is clearly 0 when $\lambda> \kappa_0 + \frac{2}{c_0}l$.
Thus,
 for each $\delta> 0$ we obtain
\begin{align*}
\bbE & \left[ T^-_e(R,l,\cdot\,)  -\bbE\left[ T(le,\cdot\,) \right]\right]
=\int_0^\infty \bbP\left[ T^-_e(R,l,\omega)-\bbE\left[ T(le,\cdot\,) \right]\geq \lambda\right] d\lambda\\
&\leq \max\{C_1, {\bar C}_\rho R^\delta l^{\beta} \} +C_2R^{d-1}l \int_{{\bar C}_\rho R^\delta l^{\beta}}^\infty \exp \left( -{\bar C}_\rho^{-2}{\lambda^2}{l^{-2\beta}} \right)d\lambda + C_3R^{d-1}l^2 \,\phi \left( \frac{c_0}{4c_1} l \right) \\
&\leq   \max\{C_1, {\bar C}_\rho R^\delta l^{\beta} \} +C_2 {\bar C}_\rho R^{d-1}  l^{1+\beta}\int_{ R^\delta}^\infty e^{ -s^2}d s + C_3R^{d-1}l^2 \,\phi \left( \frac{c_0}{4c_1} l \right) ,
\end{align*}
with some $C_3>0$.
Since
\[
\int_r^\infty e^{-s^2}ds=e^{-r^2}\int_0^\infty e^{-s^2-2sr}ds\leq Ce^{-r^2},
\]
it follows from $R\geq l$ that there is $C_{\delta}\ge 1$ such that
\[
    R^{d-1}  l^{1+\beta}\int_{ R^{\delta}}^\infty e^{ -s^2} d s \leq   R^{d+\beta}e^{-R^{2\delta}}\leq C_{\delta}.
\]
This proves the first inequality in \eqref{4.2}.  The proof of the second is analogous.
\end{proof}

We can now show that $\bbE\left[ T(le,\cdot\,;\calH_e^-) \right]$ is close to being linear in $l$.

\begin{proposition}\lb{P.4.4}
There is $C>0$, and for each $\delta>0$ there is $C_\delta\geq 1$, such that for all $l, m \geq 0$ we have (with ${\bar C}_\rho$ from \eqref{b.1'} resp.~\eqref{b.1''})
\[
\big|\bbE\left[ T((l+m)e,\cdot\,;\calH_e^-) \right]-\bbE\left[ T(le,\cdot\,;\calH_e^-) \right]-\bbE\left[ T(me,\cdot\,;\calH_e^-) \right]\big|\leq {\bar C}_\rho (C_{\delta}+(l+m)^{\beta+\delta}).
\]
\end{proposition}

\noindent {\it Remark.}  We will in fact only need the weaker upper bound ${\bar C}_\rho C_{\delta}(l+m)^{\beta+\delta}$.

\begin{proof}
Without loss of generality, we can assume that $l\geq m \geq 0$, and we also let $l_0$ be from Lemma \ref{L.5.3}. By Lemma \ref{L.2.2}, for all $m\leq \frac{4c_1}{c_0}l_0$ we have $ T(me,\cdot;\calH_e^-)\le\kappa_0+8c_1l_0$ and
\[
u( \tau_0,x,\cdot\,;\calH_e^--me)\geq u(0,x,\cdot\,;\calH_e^-),
\]
with $\tau_0:= \kappa_0+8c_1l_0 +2c_0^{-1}R_0$.
This and stationarity of $f$ yield
\[
\bbE\left[ T((l+m)e,\cdot\,;\calH_e^-) \right]=\bbE\left[ T(le,\cdot\,;\calH_e^--me) \right]\leq \tau_0+ \bbE\left[ T(le,\cdot\,;\calH_e^-) \right].
\]
All this and Lemma \ref{L.4.1} with $(l+m,le)$ in place of $(l,x)$ yield the claim  for all $m\leq \frac{4c_1}{c_0}l_0$ and $l\ge m$, with any $C_\delta\ge 2\tau_0$.

Now assume that $l\geq m\ge \frac{4c_1}{c_0}l_0$, and let us first prove the direction
\[
\bbE\left[ T((l+m)e,\cdot\,) \right]-\bbE\left[ T(le,\cdot\,) \right]-\bbE\left[ T(me,\cdot\,) \right] \leq {{\bar C}_\rho}(C_\delta+ l^{\beta+\delta}).
\]
Pick $R_m:= D_2(1+ \frac{2m}{c_0})$, with $D_2$ from Lemma \ref{L.3.1}, and denote $T_1(\omega):=T^-_e(R_m,l,\omega)$.
Then Lemma \ref{L.5.3} shows that for each $\omega\in\Omega$ we have ${T_1(\omega)}\in [\frac{l}{2c_1},\frac{2l}{c_0}]$ and hence also
\[
u_1(T_1(\omega),\cdot)\ge 1-\theta^* \qquad\text{on } \calH_e^- + (4c_1)^{-1}{c_0}l e 
\]
holds with $u_1:=u(\cdot,\cdot,\omega;\calH_e^-)$.  Therefore,
\[
u_1(T_1(\omega),\cdot)\ge 1-\theta^* \qquad\text{on }  \calC_e^-(R_m,l)\cup \left (\calH_e^- + (4c_1)^{-1}{c_0}l e \right). 
\]

Hence Lemma \ref{L.2.2} shows that with $u_2:=u(\cdot,\cdot,\omega;\calH^{-}_e+le)$ and ${\tau_0}:=\kappa_0+\frac{2R_0}{c_0}$, we have
\[
u_1({T_1(\omega)}+{\tau_0},\cdot)\geq (1-\theta^*)\chi_{\calH^{-}_e+(l+R_0)e} \geq u_2(0,\cdot) 
\qquad \text{ on } \big\{x\in\bbR^d \,\big|\, |x-(x\cdot e)e|\leq R_m \big\}
\]
%
Since the set above contains
$ B_{R_m}((l+m)e)$
and we also have
$T((l+m)e,\omega;\calH^{-}_e+le)\le \frac{2m}{c_0}$ by Lemma \ref{L.5.3}, we can apply  Lemma \ref{L.3.1} (see Lemma \ref{L.4.0'}) with $f_1=f_2=f$ and
\[
(u_1,u_2,\eta,y,t_0,R)=(u_1,u_2,0,(l+m)e,{T_1(\omega)}+{\tau_0},R_m)
\]
to obtain 
\beq\lb{5.5}
T((l+m)e,\omega;\calH^{-}_e)\leq T((l+m)e,\omega;\calH^{-}_e+le)+{T_1(\omega)}+{\tau_0}+2\kappa_*+\kappa_0.
\eeq
Taking expectations on both sides of this inequality and using stationarity of $f$ yields
\begin{align*}
    \bbE \left[T((l+m)e,\cdot\,;\calH^{-}_e)\right]&\leq \bbE \left[T((l+m)e,\cdot\,;\calH^{-}_e+le)\right]+\bbE \left[{T_1}(\cdot)\right]+{\tau_0}+2\kappa_*+\kappa_0\\
    &=\bbE \left[T(me,\cdot;\calH^{-}_e)\right]+\bbE \left[T_1(\cdot)\right]+{\tau_0}+2\kappa_*+\kappa_0.
\end{align*}
Since $R_m\le Cl$ and $ l^{d+1}\phi(\frac{c_0}{4c_1} l)\le C$ for some $C>0$ due to \textbf{(H4')} and $\be_3\ge\frac{d+1}{2d+1+m_4'}$, it now follows from Lemma \ref{L.4.2} that for any $\delta>0$ we indeed have
\[
\bbE \left[T((l+m)e,\cdot\,;\calH^{-}_e)\right]
\leq \bbE \left[T(me,\cdot\,;\calH^{-}_e)\right]+\bbE\left[T(le,\cdot\,;\calH^{-}_e)\right]+{\bar C}_\rho( C_\delta+l^{\beta+\delta}),
\]
with some $C,C_\delta$ (and ${\bar C}_\rho$ from \eqref{b.1'} resp.~\eqref{b.1''}).

Let us now turn to the other direction (again assuming $l\geq m\ge \frac{4c_1}{c_0}l_0$)
\[
\bbE\left[ T(le,\cdot\,;\calH^{-}_e) \right]+\bbE\left[ T(me,\cdot\,;\calH^{-}_e)-\bbE\left[ T((l+m)e,\cdot\,;\calH^{-}_e) \right] \right] \leq {\bar C}_\rho(C_\delta + l^{\beta+\delta}),
\]
the proof of which is a little more involved.  With $\be_1$ from \eqref{3.13}, let
\beq\lb{4.4}
\eta:=\min \left\{\frac{\theta^*}2,\,\frac1{2M_*},\,l^{\be_1-1} \right\} >0,
\eeq
and then (see Subsection 2.1 for the other constants)
\[ 
R_{l}:= \frac 1{\mu_*\eta^{m_2}} \left(1+\left(\frac{2l}{c_0}\right)^{\al_2}\right) \qquad\text{and}\qquad R_m:= D_2\left(1+\frac {2m}{c_0}\right).
\]
For any $\omega\in\Omega$, denote $T_1'(\omega):=T_e^{+}(R_{l}+R_m,l,\omega)$
and
\begin{align*}
    u'_1(t,x)&:=u(t+T_1'(\omega),x,\omega;\calH^{-}_e),\\
u'_2(t,x)&:=u(t,x,\omega;\calH^{-}_e+le).
\end{align*} 
Lemma \ref{L.5.3} yields $T_1'(\omega)\in \left[\frac{l}{2c_1},\frac{2l}{c_0}\right]$ and hence also
\[
u'_1(0,\cdot)\le 1-\theta^* \qquad\text{on } \calH_e^+ + {4c_1} c_0^{-1}l e, 
\]
where $\calH_e^+:=\bbR^d\setminus \calH_e^-$.  Therefore,
\[
u'_1(0,\cdot)\le 1-\theta^* \qquad\text{on } \calC_e^+(R_{l}+R_m,l) \cup (\calH_e^+ + {4c_1}c_0^{-1}l e). 
\]
From \eqref{2.1} and ${T_1'(\omega)}\leq \frac{2l} {c_0}$ we have
$R_l\geq  L_{u,\eta,1-\theta^*}({T_1'(\omega)})$,
and so
\beq\lb{4.16} 
u'_1(0,\cdot)\le \eta 
\qquad \text{ on } \big\{x\in\bbR^d \,\big|\, x\cdot e\ge l+R_l \text{ and } |x-(x\cdot e)e|\leq R_m \big\}.
\eeq

As for $u'_2$, Lemma \ref{L.2.2} shows that
\[
u'_2(\kappa_0+{2R_l}{c_0 }^{-1},\cdot)\ge 1-\theta^* \qquad\text{on } \calH_e^- + (l+R_l) e, 
\]
and Lemma \ref{Cor.2.1} then shows that for ${\tau_0^l}:= \kappa_0+{2R_l}{c_0 }^{-1} + D_1\eta^{1-{m_1}}$
we have
\beq\label{e.4.1}
u'_2(\tau_0^l,\cdot)\ge 1-\eta \qquad\text{on } \calH_e^- + (l+R_l) e.
\eeq 
Note also that there is $C>0$ such that
\beq\lb{5.333}
{\tau_0^l}\le C(l^{\al_2}\eta^{-{m_2}}+\eta^{1-m_1}) \le C l^{\be_1}
\eeq
because  \eqref{3.13} 
shows that
\[
\max\{{(1-\be_1)({m_1}-1)},{{\al_2}+(1-\be_1) {m_2}}\}= \be_1.
\]

From \eqref{4.16} and \eqref{e.4.1} we now have that
\[
u'_2({\tau_0^l},\cdot)\geq u'_1(0,\cdot)-\eta
\qquad \text{ on } \big\{x\in\bbR^d \,\big|\, |x-(x\cdot e)e|\leq R_m \big\}.
\]
Since the set above contains
$ B_{R_m}((l+m)e)$,
 we can apply  Lemma \ref{L.3.1} (see Lemma \ref{L.4.0'}) with $f_1=f_2=f$ and
\[
(u_1,u_2,\eta,y,t_0,R)=(u'_2,u'_1,\eta,(l+m)e, \tau_0^l,R_m)
\]
to obtain
\beq\lb{5.5'}
T((l+m)e,\omega;\calH^{-}_e+le) \le (1+{M_*}\eta) \big[T((l+m)e,\omega;\calH^{-}_e)  - T_1'(\omega) \big] +{\tau_0^l}+2\kappa_*+\kappa_0,
\eeq
provided we also have
\beq\lb{5.5''}
T((l+m)e,\omega;\calH^{-}_e)- T_1'(\omega)\le {2m}{c_0}^{-1}
\eeq
(notice that $T((l+m)e,\omega;\calH^{-}_e) \ge T_1'(\omega)$ because $l+m\le 2l\le \frac {4c_1}{c_0}l$).
But since Lemma \ref{L.2.2}  yields $T((l+m)e,\omega;\calH^{-}_e+le)\le \kappa_0+ \frac {2m}{c_0}$, 
\eqref{5.5'} obviously holds even if \eqref{5.5''} fails.
%
%
Since $T((m+l)e,\omega;\calH^{-}_e)\leq \frac{2(m+l)}{c_0}\leq \frac{4l}{c_0}$ by Lemma \ref{L.5.3}, we get from \eqref{4.4} and  \eqref{5.333},
\begin{align*}
    T((l+m)e,\omega;\calH^{-}_e+le)&\leq T((l+m)e,\omega;\calH^{-}_e) - {T_1'}(\omega) +{4l}{c_0}^{-1}{M_*}\eta +C l^{\be_1}\\
    &\leq T((l+m)e,\omega;\calH^{-}_e) - {T_1'}(\omega) +C l^{\be_1},
\end{align*}
for some $C>0$.
Taking expectations, using stationarity of $f$, $\be_1\le\beta$, 
as well as $R_l+R_m\le Cl$ and $ l^{d+1}\phi(l)\le C$ for some $C>0$ (due to \textbf{(H4')}, \eqref{3.13}, and $\be_3\ge\frac{d+1}{2d+1+m_4'}$),
and applying Lemma \ref{L.4.2} shows that for any $\delta>0$ we indeed have
\[ 
\begin{aligned}
    \bbE \left[T((l+m)e,\cdot\,;\calH^{-}_e)\right]&\geq \bbE \left[T((l+m)e,\cdot\,;\calH^{-}_e+le)\right]+\bbE \left[{T_1'}(\cdot)\right] -C l^{\be_1}\\
    &\ge \bbE \left[T(me,\cdot;\calH^{-}_e)\right]+\bbE\left[T(le,\cdot\,;\calH^{-}_e)\right]-{\bar C}_\rho( C_\delta+l^{\beta+\delta}),
\end{aligned}
\]
with some $C,C_\delta$ (and ${\bar C}_\rho$ from \eqref{b.1'} resp.~\eqref{b.1''}).
%
%
%
\end{proof}

Now we are ready to prove the main result of this section. 

\subsection{Proof of Proposition \ref{P.4.5}}
Let $G(l):=\frac{1}{l}\bbE[T(le,\cdot\,)]\ge 0$ and $\gamma:=\beta+\delta$. It follows from Lemma \ref{L.2.2} and Corollary \ref{C.2.1} that 
\[
\frac{1}{c_1}\leq \liminf_{l\to\infty} G(l)\le \limsup_{l\to\infty} G(l) \leq \frac{1}{c_0}.
\]
It therefore suffices to show that there is $C_\delta>0$ such that with either ${\bar C}_{\delta,\rho}:=C_\delta\sqrt\rho$ (when $f$ satisfies \textbf{(H2')}) or  ${\bar C}_{\delta,\rho}:=C_\delta$ (when $f$ satisfies \textbf{(H3)} and \textbf{(H4')}), we have
 for all $l\ge m\geq 2$,
\beq\lb{4.12}
|G(l)-G(m)|\leq  {\bar C}_{\delta,\rho}  m^{\gamma-1}.
\eeq
Since $\beta<1$, we only need to consider $\delta>0$ such that $\gamma<1$.

By Lemma \ref{L.2.2}, $G(l)$ is no more than
\begin{equation}\lb{4.20}
    \frac{1}{l}\left(\frac{2l}{c_0}+\kappa_0\right)\leq \frac{2}{c_0}+\kappa_0 
\end{equation}
for all $l\ge 1$,
and we also have $|T(le,\cdot)-T(me,\cdot)|\leq \frac{4}{c_0}+\kappa_0$ when $|l-m|\leq 2$. Therefore, there exists $C_0>0$ such that for all $l,m$ satisfying $m+2\geq l\geq m\geq 2$ we have
\begin{equation}\lb{4.28}
    |G(l)-G(m)|\leq \left|\frac{m-l}{m} G(l)\right|+\frac{1}{m} \bbE[|T(le,\cdot)-T(me,\cdot)|]\leq C_0m^{-1}.
\end{equation}
Using Proposition \ref{P.4.4}, we also see that there is $C_\delta\ge 1$ such that for any $l\ge m\geq 2$ (and with ${\bar C}_{\delta,\rho} \geq 1$ given above),
\beq\lb{4.11}
\begin{aligned}
    |G(l)-G(m)|&= \frac{1}{l}\left|\bbE[T(le,\cdot\,)]-\bbE[T(me,\cdot\,)]-\frac{l-m}{m}\bbE[T(me,\cdot\,)]\right|\\
    &\leq \frac{1}{l}\Big|\bbE[T(le,\cdot\,)]-\bbE[T(me,\cdot\,)]-\bbE[T((l-m)e,\cdot\,)]\Big|\\
    &\qquad +\frac{l-m}{l}\left|\frac{1}{m}\bbE[T(me,\cdot\,)]-\frac{1}{l-m}\bbE[T((l-m)e,\cdot\,)]\right|\\
    &\leq  {\bar C}_{\delta,\rho} l^{\gamma-1}+\frac{l-m}{l}|G(m)-G(l-m)|.
\end{aligned}
\eeq

 Now assume that for some $p\in \{2,3,4,...\}$, there is $N_p\ge 2C_0$ such that for all $m\in[2,2^p]$ and $l\geq m$ we have
\beq \lb{4.29}
|G(l)-G(m)|\leq N_p m^{\gamma-1}.
\eeq
This is in fact true for $p=2$, because  \eqref{4.20} shows that \eqref{4.29} holds for all $m\in[2,4]$ and $l\geq m$ with $N_2=\max\{4(\frac{2}{c_0}+\kappa_0),2C_0\}$.  We will then extend \eqref{4.29} to all $m\in [2,2^{p+1}]$ and $l\ge m$, with a relevant new constant $N_{p+1}\ge N_p$.

First, note that for any $m\in[2,2^{p+1}]$ and $l\in[m, \frac{3}{2}m]$ we have
\beq \lb{4.13}
    |G(l)-G(m)|\leq ({\bar C}_{\delta,\rho}+3^{-\gamma}N_p)l^{\gamma-1}.
\eeq
Indeed, this holds for $l\in[m,m+2] $ due to \eqref{4.28} and $\frac{C_0}m\le \frac{N_p}{m+2}\le 3^{-\gamma} N_p (m+2)^{\gamma-1}$.
And if instead $l\in [m+2,\frac{3}{2}m]$, then $l-m\in[2,\min\{2^p,\frac l3\}]$, so it follows from \eqref{4.11} and the induction hypothesis \eqref{4.29} that
\[
    |G(l)-G(m)| \leq {\bar C}_{\delta,\rho} l^{\gamma-1}+\frac{l-m}{l}N_p(l-m)^{\gamma-1}
    \leq ({\bar C}_{\delta,\rho}+3^{-\gamma}N_p) l^{\gamma-1}.
\]

Let us now consider $m\in (2^p,2^{p+1}]$ and $l\geq \frac{3}{2}m$.
Let $l_k:=2^{-k}l$ for $k=0,1,...,j$, with $j$ chosen so that $l_j\in (\frac{3}{4}m,\frac{3}{2}m]$. Since $l_k=l_{k-1}-l_k$, from \eqref{4.11} we obtain
\[
    |G(l)-G(l_j)|\leq \sum_{k=1}^j |G(l_{k-1})-G(l_{k})| \leq \sum_{k=1}^j {\bar C}_{\delta,\rho} \,l_{k-1}^{\gamma-1} \leq  {\bar C}_{\delta,\rho} \,l_j^{\gamma-1}\sum_{k=1}^j \left(2^{j-k+1}\right)^{\gamma-1} \leq {\bar C}_{\delta,\rho}' m^{\gamma-1},
\]
where ${\bar C}_{\delta,\rho}' :={\bar C}_{\delta,\rho}(\frac{4}{3})^{1-\gamma}\frac{2^{\gamma-1}}{1-2^{\gamma-1}}<\infty$ (recall that $\gamma<1$).

If  $l_j\in [m,\frac{3}{2}m]$, then \eqref{4.13} yields
\[
|G(l_j)-G(m)|\leq ({\bar C}_{\delta,\rho}+3^{-\gamma} N_p)l_j^{\gamma-1}\leq ({\bar C}_{\delta,\rho}+3^{-\gamma} N_p)m^{\gamma-1}.
\]
If instead $  l_j\in \left(\frac{3}{4}m,m\right)$, then $l_j\in [2,2^{p+1}]$ and $m\in [l_j,\frac{4}{3}l_j]$.
Hence \eqref{4.13} again yields
\[
\begin{aligned}
|G(m)-G(l_j)|
   \leq  ({\bar C}_{\delta,\rho}+3^{-\gamma}N_p)m^{\gamma-1}.
\end{aligned}
\]
In either case we obtain 
\[
|G(l)-G(m)|\leq  |G(l)-G(l_j)|+ |G(l_j)-G(m)|
\leq \left({\bar C}_{\delta,\rho}+{\bar C}_{\delta,\rho}' +3^{-\gamma}N_p\right) m^{\gamma-1}.
\]

This, \eqref{4.29}, and \eqref{4.13} now prove \eqref{4.29} with $p+1$ in place of $p$ (so it holds for all  $m\in[2,2^{p+1}]$ and $l\geq m$) and with 
\[
N_{p+1}={\bar C}_{\delta,\rho}+{\bar C}_{\delta,\rho}' +3^{-\gamma}N_p \le \max \left\{N_p,\frac{3^\gamma}{3^\gamma-1}({\bar C}_{\delta,\rho}+{\bar C}_{\delta,\rho}' ) \right\}.
\]
  Since \eqref{4.29} holds for $p=2$ with $N_2=\max\{4(\frac{2}{c_0}+\kappa_0),2C_0\}$, it follows that it holds for any $p=2,3,\dots$ with  $N_p= \max \left\{4(\frac{2}{c_0}+\kappa_0),2C_0,\frac{3^\gamma}{3^\gamma-1}({\bar C}_{\delta,\rho}+{\bar C}_{\delta,\rho}' ) \right\} =: {\bar C}_{\delta,\rho}''$.  This proves \eqref{4.12}.

\section{Deterministic Front Speeds and Proof of Theorem \ref{T.1.3}} \lb{S6}

We are now ready to prove Theorem \ref{T.1.3}.  This is because it was shown in \cite{zlatos2019} that such homogenization results for reaction-diffusion equations and related models follow from appropriate estimates on the dynamics of the solutions to \eqref{4.1} for all vectors $e\in \bbS^{d-1}$.  We will be able to obtain these estimates using the main results from Sections \ref{S3}--\ref{S5}.

Specifically, we will use Proposition \ref{P.4.5}, and either Proposition \ref{P.3.4} (when we assume \textbf{(H2')})  or Proposition \ref{P.3.1}  (when we assume \textbf{(H3)} and \textbf{(H4')}) in the proof.  We will handle both cases at once, using that either of the latter two propositions yields \eqref{5.7} above, with the notation from either \eqref{b.1'} in the first case or \eqref{b.1''} in the second. 

For us, the key result from \cite{zlatos2019} will be Theorem 5.4, which applies when for almost all $\omega\in\Omega$, the reaction $f(\cdot,\cdot,\omega)$ has  deterministic strong exclusive front speeds in all directions $e\in\bbS^{d-1}$.  We will first define these, following Definitions 1.3, 1.6, and Remark 3 after Hypothesis H' in \cite{zlatos2019}, and then prove their existence.  


\begin{definition}\lb{D.5.0}
Let $f$ satisfy \textbf{(H1)} and let $e\in \bbS^{d-1}$.  If there is $c^*(e)\in\bbR$ and $\Omega_e\subseteq\Omega$ with $\bbP(\Omega_e)=1$ such that for each $\omega\in\Omega_e$ and compact $K\subseteq\calH^{+}_e=\{x\in \bbR^d\,|\, x\cdot e>0\}$, 
\begin{align*}
&\lim_{t\to\infty} \,  \inf_{x\in(c^*(e)e-K)t}u(t,x,\omega;\calH_e^-)=1,\\
&\lim_{t\to\infty} \, \sup_{x\in(c^*(e)e+K)t}u(t,x,\omega;\calH_e^-)=0
\end{align*}
holds for the solution to \eqref{4.1} with $l=0$  and some $u_{0,\calH^{-}_e}$ satisfying \eqref{2.6} and \eqref{2.7} with $S=\calH^{-}_e$,
then we say that $c^*(e)$ is a \textit{deterministic front speed} in direction $e$ for \eqref{1.1}.

This speed is {\it strong} if for each such $\omega$ and  $K$, and each $\Lambda\ge 0$, we have
\beq\lb{6.111}
\begin{aligned}
&\lim_{t\to\infty}\inf_{|y|\leq \Lambda t} \, \inf_{x\in(c^*(e)e-K)t}u(t,x,\Upsilon_y\omega;\calH_e^-)=1,\\
&\lim_{t\to\infty}\sup_{|y|\leq \Lambda t} \, \sup_{x\in(c^*(e)e+K)t}u(t,x,\Upsilon_y\omega;\calH_e^-)=0.
\end{aligned}
\eeq

And if, in addition, for each such $\omega$ and  $K$ there is $\lambda_{K,\omega,e}:(0,1]\to (0,1]$ satisfying  $\lim_{a\to 0}\lambda_{K,\omega,e}(a)=0$ such that for each $\Lambda>0$ and $a\in (0,1]$ we have
\beq\lb{6.222}
\limsup_{t\to\infty} \sup_{|y|\leq \Lambda t} \, \sup_{x\in(c^*(e)e+K)t} w_{e,a}(t,x,\Upsilon_y\omega)\leq \lambda_{K,\omega,e}(a),
\eeq
where $w_{e,a}(\cdot,\cdot, \omega)$ solves \eqref{1.1} with initial data
\[
w_{e,a}(0,\cdot,\omega)=\chi_{\calH_e^-}+a \chi_{\calH_e^+},
\]
then $c^*(e)$ is a \textit{deterministic  strong exclusive front speed} in direction $e$ for \eqref{1.1}.
\end{definition}

\noindent {\it Remarks.}  
1.  Lemma \ref{L.2.2} and the comparison principle show that all these definitions are independent of the choice of $u_{0,\calH^{-}_e}$ satisfying \eqref{2.6} and \eqref{2.7}.  We could equivalently choose $u_{0,\calH^{-}_e}:=(1-\theta^*)\chi_{\calH^{-}_e}$ here, but having solutions with $u_t\ge 0$ will be more convenient.
\smallskip

2.  We will show that in Theorem \ref{T.1.3},  $\lambda_{K,\omega,e}(a)=a$ for all $(K,\omega,e)$ as above and all small enough $a>0$ (depending on $ M, \theta_1,m_1,\al_1, \mu_*,K$).
\smallskip

Let us first show that the reactions we consider here have deterministic strong front speeds, and then we will show that all these speeds are also exclusive.

\begin{proposition}\lb{P.5.3}
Assume that $f$ either satisfies \textbf{(H2')} and has a finite range of dependence, or satisfies \textbf{(H3)} and  \textbf{(H4')}. For each $e\in\bbS^{d-1}$, let $\bar{T}(e)$ be from Proposition \ref{P.4.5}. Then
$c^*(e):=\bar{T}(e)^{-1}\in[c_0,c_1]$ is the deterministic strong front speed in direction $e$ for \eqref{1.1}.
\end{proposition}

\begin{proof}
Fix any $e\in\bbS^{d-1}$, $\Lambda\ge 0$, and compact $K'\subseteq \calH_e^+$.  Let $K\subseteq \calH_e^+$ be compact and such that $K'\subseteq K^0$,  let ${d_K}:=d(K,\calH_e^-)>0$, and let
$A_K:=1+\diam (K)^d<\infty$.
Let us also use the notation \eqref{5.444},
and for any $\eta\in (0,\theta^*]$ and $t\ge 0$, let
\[
I_t^\eta(K,\Lambda):= \left\{ \omega\in\Omega\,\bigg|\, \inf_{|y|\leq \Lambda (t+1)} \, \inf_{x\in(c^*(e)e-K)t}u(t,x,\Upsilon_y\omega)< 1-\eta \right\}.
\]

Assume that $u(t,x,\Upsilon_y\omega)\geq 1-\theta^*$ for some $(t,x,y,\omega)\in [0,\infty)\times\bbR^{2d}\times\Omega$. Since for all $y'\in\bbR^d$ we have $u(\cdot,\cdot,\Upsilon_{y'}\omega )=u(\cdot,\cdot+y',\omega ;\calH_e^-+(y'\cdot e)e)$, Lemma \ref{L.2.2} and comparison principle yield for all $y'\in B_{\sqrt{d}}(y)$,
\[
u(\cdot+\tau_0,\cdot-y',\Upsilon_{y'}\omega )\geq u(\cdot,\cdot-y,\Upsilon_y\omega ),
\]
with $\tau_0:=\kappa_0+2c_0^{-1}(R_0+\sqrt{d})$. Applying Lemma \ref{L.2.2} again, we obtain for all $y'\in B_{\sqrt{d}}(y)$,
\[
u(t+2\tau_0,x,\Upsilon_{y'}\omega )\geq u(t+\tau_0,x+y'-y,\Upsilon_y\omega)\geq 1-\theta^*.
\]
Then Lemma \ref{Cor.2.1} shows that if $\eta\in (0,\theta^*]$ and $\tau_\eta:=2\tau_0+4c_0^{-1}\sqrt{d}+D_1\eta^{1-{m_1}}$, then
\beq\lb{6.2}
u(t+\tau_\eta,\cdot,\Upsilon_{y'}\omega)\geq (1-\eta)\chi_{B_{\sqrt{d}}(x)}
\eeq
for all $y'\in B_{\sqrt{d}}(y)$.
This shows that we can only have $\omega\in I_t^\eta(K,\Lambda)$ for some  $t\geq \tau_\eta$ if
\begin{align*}
\inf_{(x,y)\in Z_{K,\Lambda,t}}u(t-\tau_\eta,x,\Upsilon_y\omega)< 1-{\theta^*},
\end{align*}
where 
$
Z_{K,\Lambda,t}:=((c^*(e)e-K)t\cap \bbZ^d)\times (B_{\Lambda (t+1)}(0)\cap\bbZ^d).
$
From this we obtain
\begin{align*}
\bbP[I_t^\eta(K,\Lambda)]
   & \leq  \sum_{(x,y)\in Z_{K,\Lambda,t}}\bbP\left[ u(t-{\tau_\eta},x,\Upsilon_y\omega)< 1-\theta^* \right]\\
   & =  \sum_{(x,y)\in Z_{K,\Lambda,t}}\bbP\left[ T(x,\Upsilon_y\omega)> t-{\tau_\eta} \right] .
\end{align*}
If $t\ge 2(\tau_\eta+\kappa_0+\frac 2{c_0})$, then from Lemma \ref{L.2.2} we have $T(x,\Upsilon_y\omega)\le t-\tau_\eta$ whenever $x\cdot e\le 1 + \frac{c_0t}4$.  Hence, with $Z_{K,\Lambda,t}':=\{(x,y)\in Z_{K,\Lambda,t}\,|\, x\cdot e> 1 + \frac{c_0t}4 \}$ we obtain
\beq\lb{5.3}
\bbP[I_t^\eta(K,\Lambda)] \leq  \sum_{(x,y)\in Z_{K,\Lambda,t}'}\bbP\left[ T(x,\Upsilon_y\omega)> t-{\tau_\eta} \right] .
\eeq
Note also that  there is $C>0$ such that this sum has at most $C(1+\Lambda^d)A_K t^{2d}$ terms.

Consider any $(x,y)\in Z_{K,\Lambda,t}'$, where $t\ge 2(\tau_\eta+\kappa_0+\frac 2{c_0})$.
With the notation from \eqref{b.1'} resp.~\eqref{b.1''}, and $\beta':=\frac{1+\beta}{2}\in (0,1)$, Proposition \ref{P.4.5} and stationarity of $f$ yield
\[
\left|\frac{\bbE[T(x,\cdot\,)]}{x\cdot e}-\frac 1{c^*(e)} \right|\leq \bar{C}_\rho (x\cdot e)^{\beta'-1}.
\]
Since $x\cdot e\leq (c^*(e)-{d_K})t$  and $c^*(e)={\bar{T}(e)}^{-1}\leq c_1$, we obtain
\[
\bbE[T(x,\cdot\,)] \leq \frac{x\cdot e}{c^*(e)}+\bar{C}_\rho (x\cdot e)^{\beta'}
\leq t-\frac{{d_K}t}{c^*(e)}+\bar{C}_\rho((c^*(e)-{d_K})t)^{\beta'}
    \leq  t-\frac{{d_K}t}{2c_1}
\]
whenever 
\[
t\geq 
\max \left\{ ({2\bar{C}_\rho c_1^{1+\beta'} d_K^{-1}} )^{\frac{1}{1-{\beta'}}}, 2(\tau_\eta+\kappa_0+ 2{c_0}^{-1}),  4c_1d_K^{-1}\tau_\eta \right\}.
\]
Hence for such $t$, \eqref{5.7} yields $C>0$ (defining $\bar C_\rho$ via \eqref{b.1'} resp.~\eqref{b.1''}) such that
\begin{align*}
    \bbP\left[ T(x,\Upsilon_y\omega) > t-{\tau_\eta} \right]&\leq
    \bbP\left[ \left|T(x,\Upsilon_y\omega)-\bbE[T(x,\cdot\,)]\right|\geq \frac{ {d_K}t}{2c_1}-{\tau_\eta} \right]\\
&\leq 2\exp\left(-\bar{C}_\rho^{-2}{\left( \frac{ {d_K}t}{2c_1}-{\tau_\eta}\right)^2} (x\cdot e)^{-2{\beta}} \right) + \phi(x\cdot e)\\
& \leq 2\exp\left(-\bar{C}_\rho^{-2} d_K^2 t^{2-2{\beta}}\right) + C t^{-\be_3(2d+1+m_4')}
\end{align*}
when  $(x,y)\in Z_{K,\Lambda,t}'$, where we also used that $\frac{c_0t}4\le x\cdot e\leq c_1t$ and $\tau_\eta\le \frac{d_K}{4c_1}t$ (recall that $C$ can change from line to line).  This and \eqref{5.3} show that for all large enough $t$ we have 
\[
\bbP\left[I_t^\eta(K,\Lambda)\right]\leq C(1+\Lambda^d) A_K t^{2d} \left( \exp\left(-\bar{C}_\rho^{-2}d_K^2 t^{2-2{\beta}}\right) +  t^{-\be_3(2d+1+m_4')} \right) .
\]

Then $\sum_{n\geq 1}\bbP[I_n^\eta(K,\Lambda)]<\infty$ since $\be_3> \frac{2d+1}{2d+1+m_4'}$, so the Borel-Cantelli Lemma shows that for a.e.~$\omega\in\Omega$, there is $N_\omega$ such that $\omega\notin \bigcup_{n\ge N_\omega} I_n^\eta(K,\Lambda)$.  
But since $K'\subseteq K^0$ means there is $\tau$ such that 
\[
(c^*(e)e-K')t \subseteq (c^*(e)e-K)\lfloor t \rfloor
\]
 for all $t\ge \tau$, from $u_t\ge 0$ and the definition of $I_t^\eta(K,\Lambda)$ we obtain 
\[
\inf_{|y|\leq \Lambda t} \, \inf_{x\in(c^*(e)e-K')t}u(t,x,\Upsilon_y\omega)\ge  1-\eta
\]
for all such $\omega$ and all $t\ge\max\{N_\omega,\tau\}$.  Applying this argument with $\eta=\frac 1n$, $\Lambda=n$, and $K'=\{x\in\bbR^d\,|\, x\cdot e\in[\frac 1n,n]\text{ and }|x-(x\cdot e)e|\le n\}$ for each $n\in\bbN$ yields $\Omega_1\subseteq\Omega$ with $\bbP(\Omega_1)=1$ for which the first statement in \eqref{6.111} holds.


%

It  remains to prove the the second statement for some $\Omega_2$ with $\bbP(\Omega_2)=1$, as we can then take $\Omega_e:=\Omega_1\cup\Omega_2$.  The proof is similar to that of the first statement.  With the setup from the start of its proof, let now
\[
I_t^\eta(K,\Lambda):=\bigg\{\omega\in\Omega\,\bigg|\, \sup_{|y|\leq \Lambda t}\sup_{x\in(c^*(e)e+K)t}u(t,x,\Upsilon_y\omega)> \eta \bigg\}.
\]

Assume that $u(t,x,\Upsilon_y\omega)> \eta$ for some $(t,x,y,\omega)\in [0,\infty)\times\bbR^{2d}\times\Omega$.
Then \eqref{2.1} shows that with $L_{t,\eta}:= \mu_{*}^{-1}(1+t^{\al_2} )\eta^{-{m_2}}$, there is $x'\in B_{L_{t,\eta}}(x)$ such that
\[
u(t,x',\Upsilon_y\omega)\geq 1-\theta^*.
\] 
Lemma \ref{L.2.2} now shows that 
\[
u(t+ \kappa_0+2c_0^{-1}L_{t,\eta},x ,\Upsilon_y\omega)\geq 1-\theta^*.
\]
In the same way as we obtained \eqref{6.2} (but using Lemma \ref{L.2.2} instead of Lemma \ref{Cor.2.1}), we now get for all $y'\in B_{\sqrt{d}}(y)$ and with $\tau_{t,\eta}:=3\kappa_0+\frac 2{c_0}(L_{t,\eta}+R_0+2\sqrt d)$,
\[
u(t+\tau_{t,\eta},\cdot,\Upsilon_{y'}\omega )\geq (1-\theta^*)\chi_{B_{\sqrt d}(x)}.
\]
So similarly to \eqref{5.3}, with $Z_{K,\Lambda,t}:=((c^*(e)e+K)t\cap  \bbZ^d )\times (B_{\Lambda t}(0)\cap\bbZ^d)$
 we get for all $t\ge 0$,
\beq\lb{6.3}
\bbP[I_t^\eta(K,\Lambda)] \leq  \sum_{(x,y)\in Z_{K,\Lambda,t}} \bbP\left[ T(x,\Upsilon_y\omega)\leq t+\tau_{t,\eta} \right].
\eeq
And again, there is $C>0$ such that this sum has at most $C(1+\Lambda^d)A_K t^{2d}$ terms.


Let $d_K':=d_H(K,\calH_e^-)$ and consider any $t\ge \frac 1{c_0}$ and $(x,y)\in Z_{K,\Lambda,t}$. Then $c^*(e)t\ge 1$ and 
\beq\lb{6.5}
x\cdot e\in [  (c^*(e)+{d_K})t,\,(c^*(e)+d_K') t].
\eeq
So Proposition \ref{P.4.5} and $c^*(e)=\bar{T}(e)^{-1}$ imply as above (with $\beta'=\frac{1+{\beta}}{2}$), 
\begin{align*}
\bbE[T(x,\cdot\,)]\geq  \frac{x\cdot e}{c^*(e)}-\bar{C}_\rho (x\cdot e)^{{\beta'}} \ge t+\frac {d_Kt}{c^*(e)} - 
\bar{C}_\rho ((c^*(e)+d_K') t)^{{\beta'}} \ge t+\frac {d_K t}{2c_1}
\end{align*}
whenever 
\[
t\geq 
\max \left\{ ({2\bar{C}_\rho (c_1+d_K')^{1+\beta'} d_K^{-1}} )^{\frac{1}{1-{\beta'}}}, c_0^{-1}, C_{K,\eta} \right\},
\]
where $C_{K,\eta}$ is such that $\tau_{t,\eta}\le \frac {d_K t}{4c_1}$ for all $t\ge C_{K,\eta}$ (this exists because $\alpha_2<1$, and will be used next).
Hence for such $t$, \eqref{5.7} yields $C>0$ (defining $\bar C_\rho$ via \eqref{b.1'} resp.~\eqref{b.1''}) such that
\begin{align*}
    \bbP\left[ T(x,\Upsilon_y\omega)\leq t+\tau_{t,\eta} \right]&\leq
    \bbP\left[ \left|T(x,\Upsilon_y\omega)-\bbE[T(x,\cdot\,)]\right|\geq\frac{d_K t}{2c_1}-{\tau}_{t,\eta} \right]\\
&\leq 2\exp\left(-\bar{C}_\rho^{-2}{ \left(\frac{d_K t}{2c_1}-{\tau}_{t,\eta} \right)^2} (x\cdot e)^{-2{\beta}}\right) + \phi(x\cdot e) \\
&\leq 2\exp\left(-\bar{C}_\rho^{-2}d_K^2(1+d_K')^{-2\beta}t^{2-{2\beta}} \right) + C t^{-\be_3(2d+1+m_4')}
\end{align*}
when  $(x,y)\in Z_{K,\Lambda,t}$, where  we  also used  \eqref{6.5} and $\tau_{t,\eta}\le \frac {d_K t}{4c_1}$ in the last inequality.
This and \eqref{6.3} show that for all large enough $t$ we have
\[
\bbP[I_t^\eta(K,\Lambda)]\leq C(1+\Lambda^d) A_K t^{2d} \left( \exp\left(-\bar{C}_\rho^{-2}d_K^2(1+d_K')^{-2\beta} t^{2-2{\beta}}\right) + t^{-\be_3(2d+1+m_4')} \right) .
\]

We can now conclude the proof of the second statement in \eqref{6.111} as we did the proof of the first statement, this time using that
\[
(c^*(e)e+K')t \subseteq (c^*(e)e+K)\lceil t \rceil
\]
for all large enough $t$.
\end{proof}

\noindent {\it Remark.}  This proof shows that \eqref{6.111} holds with $\Lambda t$ replaced by $\exp(t^\gamma)$ for any $\gamma<2-2\beta$.



\begin{proposition}\lb{P.6.4}
Under the hypotheses of Proposition \ref{P.5.3}, for each $e\in\bbS^{d-1}$, the speed $c^*(e)$ is also a deterministic strong  exclusive front speed in direction $e$ for \eqref{1.1}.
\end{proposition}

\begin{proof}
Having Proposition \ref{P.5.3}, this proof is now similar to the one of  \cite[Theorem 1.7(i)]{zlatos2019}. 

For any $(e,a,\omega)\in\bbS^{d-1}\times(0,1)\times\Omega$, let $u(\cdot,\cdot,\omega;\calH_e^-)$ and $ w_{e,a}(\cdot,\cdot,\omega)$ be from Definition \ref{D.5.0}, and let $\tau_a:=1+D_1a^{1-m_1}$. 
Lemma \ref{L.3.1} (see Lemma \ref{L.4.0'}) shows that if $a\in (0,\frac{1}{2}\min\{\theta^*,M_*^{-1}\}]$, then
\[
u_+(t,x):=u((1+M_*a)t+\tau_a,x,\omega;\calH_e^-)+a
\]
is a supersolution to \eqref{1.1} on $(\kappa_*,\infty)\times\bbR^d$. 
Moreover, $u_t\ge 0$ and Lemma \ref{Cor.2.1}  show that
\[
u_+(\kappa_*,\cdot)\ge u(\tau_a,\cdot,\omega;\calH_e^-)+a \geq w_{e,a}(0,\cdot,\omega).
\]
The comparison principle now yields for all $t\ge 0$,
\[
u_+(t+\kappa_*+\tau_a,\cdot)\geq w_{e,a}(t,\cdot,\omega).
\]

It now follows from Proposition \ref{P.5.3} that \eqref{6.222} holds with $\lambda_{K,\omega,e}(a)=a$ for all $\Lambda\ge 0$ and all compact $K\subseteq \calH_e^++M_*a\,c^*(e)e$.  This is true for all $a\in (0,\frac{1}{2}\min\{\theta^*,M_*^{-1}\}]$, so the result follows after letting $\lambda_{K,\omega,e}(a):=1$ for all $a\in (\frac{1}{2}\min\{\theta^*,M_*^{-1}\},1]$.
\end{proof}

\subsection{Proof of Theorem \ref{T.1.3}}
 
 For any $e\in\bbS^{d-1}$, let $\Omega_e\subseteq\Omega$ with $\bbP(\Omega_e)=1$ be the set from Definition \ref{D.5.0} and let $c^*(e)$ be the corresponding deterministic strong exclusive front speed for \eqref{1.1} from Proposition \ref{P.6.4}. Let $A\subseteq\bbS^{d-1}$ be a dense countable set and let $\Omega_0:=\bigcap_{e\in A} \Omega_e$.  Then $\bbP(\Omega_0)=1$, and for each $\omega\in\Omega_0$, \eqref{1.1} with this fixed $\omega$ has a strong exclusive front speed $c^*(e)$ in each direction $e\in A$ (i.e., \eqref{6.111} and \eqref{6.222} hold for this fixed $\omega$ and each $e\in A$, $\Lambda\ge 0$, and compact $K\subseteq\calH_e^+$).  
 
 Then \cite[Theorem 4.4(i)]{zlatos2019} shows that \eqref{1.1} with this fixed $\omega$ has a strong exclusive front speed $c^*_\omega(e)$ in each direction $e\in \bbS^{d-1}$, and $c^*_\omega$ is Lipschitz with Lipschitz constant only depending on $M$.  But then $c^*_\omega(e)$ must be independent of $\omega\in\Omega_0$ for each $e\in\bbS^{d-1}$ (instead of just all $e\in A$), and hence equals $c^*(e)$ from Proposition \ref{P.6.4} because $\bbP(\Omega_0)=1$.
 
Theorem \ref{T.1.3} now follows directly from \cite[Theorem 5.4]{zlatos2019} applied separately to each $\omega\in \Omega_0$ (see also the remarks after Hypothesis H' in \cite{zlatos2019}).

\section{Proof of Theorem \ref{T.1.7}} \lb{S7}

In this section we will show how to extend the above analysis to the cases considered in Theorem \ref{T.1.7}.
We can obviously assume  $\alpha_2'> 0$ without loss, and all constants with $C$ in them may depend on \eqref{const} as well as on $\al_2'$.

%
%

Since we now replace \textbf{(H2')} by \textbf{(H2'')}, the estimates \eqref{2.1} instead become 
\beq
\begin{aligned}\lb{7.0}
    \sup_{t\ge0 \,\&\,\eta>0} \frac{L_{u,\eta+a,1-\theta^*}(t)}{(1+t^{\al_2}) \eta^{-{m_2}}} & \le \mu_*^{-1},
\\    \inf_{\substack{(t,x)\in[\kappa_*,\infty)\times \bbR^d\\
    u(t,x)\in [{\theta^*},1-{\theta^*}]}}  u_t(t,x) t^{\al_2'} &\geq \mu_*
    \end{aligned}
\eeq
 for all $a\in[0,a_2]$ and either all $u\in\calU_{f,a}$ (when assuming \textbf{(H2'')})  or all $u\in\bigcup_{n\ge n_4} \calU_{f_n,a}$ (when assuming \textbf{(H3)} and \textbf{(H4'')}), again with some $\mu_*,\kappa_*>0$.

We will first assume without loss that $a_2=0$.
Then of course also $a=0$ and $\calU_{f,a}=\calU_f$ above, so \eqref{7.0} is just \eqref{2.1} with the extra factor of $t^{\al_2'}$ in the second estimate.  We will now show how the results in Sections \ref{S2}--\ref{S6} and their proofs change due to this.

Of the results in Section \ref{S2}, clearly only  Lemmas \ref{L.3.1} and \ref{L.4.3} are affected by this change.  They will instead become the following two results.

\begin{lemma}
\label{L.7.1}
Let  $f_1$ satisfy \textbf{(H2'')} and $f_2$ satisfy \textbf{(H1')}, and let $M_*:=2^{\al_2'}\frac{1+M}{\mu_*}$, with $\mu_*,\kappa_*$ from \eqref{7.0} for all $u\in\calU_{f_1}$. 
Fix some $\omega\in\Omega$ and let $u_1,u_2:[0,\infty)\times\bbR^d\to [0,1]$ solve \eqref{1.1} with   $f_1,f_2$ in place of $f$, respectively.    If $u_1\in\calU_{f_1}$, $t_0\ge 0$, $T\ge 2\kappa_*$, and for some $\eta\in [0, \frac 12\min\{\theta^*,M_*^{-1}(T+t_0)^{-\al_2'}\} ]$ we have
\[
f_1(x,u,\omega)=f_2(x,u,\omega)\qquad \text{whenever }u_1(t_0,x)<1-\eta\text{ and }u\in [0,1],
\]
then
\[
    u_+(t,x) :=u_1((1+M_*(T+t_0)^{\al_2'}\eta)t+t_0,x)+\eta
\]
is a supersolution to \eqref{1.1} with $f_2$ in place of $f$ on $(\kappa_*,T)\times \bbR^d$, and
\[
u_{-}(t,x) :=u_1((1-M_*(T+t_0)^{\al_2'}\eta)t+t_0,x)-\eta
\]
is a subsolution to \eqref{1.1} with $f_2$ in place of $f$ on 
$
(2\kappa_*,T)\times \{ x\in\bbR^d \,|\,u_1(t_0,x)<1-\eta\}.
$


Moreover, there is $D_2=D_2(M,\theta_1,m_1,\al_1)\ge 1$ 
such that if also ${T}_{u_2}(y)\le T$, and
\[
\sup_{x\in B_R(y)} (u_2(0,x)- u_1(t_0,x))\le \eta 
\]
for some $y\in\bbR^d$ and
$R\ge D_2(1+{T}_{u_2}(y)),$
then
\[
{T}_{u_2}(y)\geq  \left(1+ {M_*}(T+t_0)^{\al_2'} \eta \right)^{-1}(T_{u_1}(y)-t_0- 2\kappa_*-\kappa_0).
\]
\end{lemma}

\begin{proof}
The proof is the same as that of Lemma \ref{L.3.1}, replacing \eqref{2.1} by \eqref{7.0} and using $\tau_\pm(t):=(1\pm M_*(T+t_0)^{\al_2'}\eta)t+t_0$.
In particular, we use in it that for $t\le T$ we have
\[
{M_*}(T+t_0)^{\al_2'} \eta \mu_*((1+M_*(T+t_0)^{\al_2'}\eta)t+t_0)^{-\al_2'} \ge {M_*}(T+t_0)^{\al_2'} \eta \mu_*(2T+t_0)^{-\al_2'} \ge (1+M)\eta.
\]
We also have $D_2:=2\sqrt{Md}\ln \frac{4d}{\theta^*}$ 
as before.
\end{proof}

\begin{lemma}\lb{L.7.2}
Let  $f_1$ satisfy \textbf{(H2'')} and $f_2$ satisfy \textbf{(H1')}, with at least one satisfying \textbf{(H3)} with $\al_3\le 1$, and let $M_*,D_2$ be from Lemma \ref{L.7.1}. 
Fix some $\omega\in\Omega$ and let $u_1,u_2:[0,\infty)\times\bbR^d\to [0,1]$ solve \eqref{1.1} with $f_1,f_2$ in place of $f$, respectively. If $u_1\in\calU_{f_1}$, for some $y\in\bbR^d$, $R\ge D_2(1+{T}_{u_2}(y))$, and
$
\eta\in [0, \frac 12\min\{\theta^*,M_*^{-1}(\max\{{T}_{u_2}(y),2\kappa_*\}+t_0)^{-\al_2'}\} ]
$ we have
\[
f_1(x,u,\omega)\geq f_2(x,u,\omega)-\al_3\eta^{m_3}\qquad \text{for all $(x,u)\in B_R(y)\times [0,1]$},
\]
and $u_2(0,\cdot)\leq u_1(t_0,\cdot)$ for some $t_0\ge 0$ and all $x\in B_R(y)$, then
\[
{T}_{u_2}(y)\geq \left(1+ {M_*} (\max\{{T}_{u_2}(y),2\kappa_*\}+t_0)^{\al_2'} \eta \right)^{-1} \left(T_{u_1}(y)-t_0-2\kappa_*- \kappa_0\right).
\]
\end{lemma}

\begin{proof}
The proof is the same as that of Lemma \ref{L.4.3}, replacing \eqref{2.1} and Lemma \ref{L.3.1} by \eqref{7.0} and Lemma \ref{L.7.1} with $T:=\max\{{T}_{u_2}(y),2\kappa_*\}$, and using $\tau_+(t):=(1+ M_* (T+t_0)^{\al_2'}\eta)t+t_0$.
\end{proof}

We can now extend all of Sections \ref{S3}--\ref{S5} to \textbf{(H2'')} in place of \textbf{(H2')}, and obtain the following analog of Proposition \ref{P.5.3}.

\begin{proposition}\lb{T.7.3}
Assume that $f$ either satisfies \textbf{(H2'')} and has a finite range of dependence, or satisfies \textbf{(H3)} and  \textbf{(H4'')}. Then for each $e\in\bbS^{d-1}$, \eqref{1.1} has a deterministic strong front speed $c^*(e)\in [c_0,c_1]$ in direction $e$. 
\end{proposition}

\begin{proof}
In the whole proof, we assume without loss that $a_2=0$ (and so $a=0$ as well).
When extending results from Section \ref{S3}, we assume that $f$ satisfies \textbf{(H2'')} and has range of dependence at most $\rho\in[1,\infty)$; in Section \ref{S4} we instead assume \textbf{(H3)} and \textbf{(H4'')}; and in Sections \ref{S5} and \ref{S6} we assume either of these two cases, as before.

Most of Section \ref{S3} is unchanged, with  \eqref{3.13} replaced by
\beq\lb{7.111}
\be_1:=\max\left\{ \frac{(1+\alpha_2')(m_1-1)}{m_1}, \frac{(1+\alpha_2')m_2+\alpha_2}{m_2+1} \right\},
\eeq
which is still in $(0,1)$ thanks to $\alpha_2'<\min\{ \frac 1{m_1-1}, \frac {1-\alpha_2}{m_2} \}$.
The first adjustment is required in Lemma \ref{L.3.3}, where we used the second claim in \eqref{2.1} to obtain \eqref{e.3.4}.   
We have here  $u(\tau(x,\omega),y,\omega)\ge\theta^*$ for some $y\in\overline{B_1(x)}$, and can instead use Lemma \ref{L.2.2} to get
\[
u(\tau(x,\omega)+ \kappa_0 + 2c_0^{-1} (L_{u,\theta^*,1-\theta^*}(\tau(x,\omega))+1),x,\omega)\geq 1-\theta^*.
\]
Since
the first claim in \eqref{7.0} (with $a=0$ and $\eta=\theta^*$) yields $C>0$ such that
\[
\kappa_0 + 2c_0^{-1} (L_{u,\theta^*,1-\theta^*}(\tau(x,\omega))+1) \le C(1+\tau(x,\omega)^{\alpha_2})
\]
(recall that $u\in\calU_f$) and Lemma \ref{L.2.2} also implies $\tau(x,\omega)\le C(1+t)$ whenever $\omega\in F_{t,x}$, we have
$T(x,\cdot)\le \tau(x,\cdot)+C(1+t^{\alpha_2})$ on $F_{t,x}$.  Hence the first claim of  Lemma \ref{L.3.3} becomes
\[
   \left|\bbE[T(x,\cdot)\chi_{F_{t,x}}\,|\,\calG_{t}]-T(x,\cdot)\chi_{F_{t,x}}\right|\leq C(1+t^{\al_2}) \qquad \text{ on }\Omega,
\]
while the second holds only for $s\in[0,t-C(1+t^{\al_2})]$.

As for Proposition \ref{P.3.0}, instead of \eqref{3.26} we let 
\[
\eta:=  C_1^{-1}(\rho+d(x_0,S))^{-\gamma},
\]
where $C_1>0$ will be chosen shortly and 
\beq\lb{7.6}
\gamma:=\min\left\{ \frac{1+\al_2'}{m_1} , \frac{1+\al_2'-{\al_2}}{m_2+1} \right\}>0.
\eeq
Note that then, similarly to \eqref{3.13'}, we have
\beq\lb{7.5}
\max\{ \gamma(m_1-1), \al_2+\gamma m_2,1+\al_2'-\gamma\}=\be_1<1
\eeq
(in particular, $\gamma>\alpha_2'$), and \eqref{3.21} and \eqref{3.2} continue to hold. Now we pick $C_1$ so that with with  $T:=\max\{T(x_0,\omega), {T}_i(x_0,\omega),2\kappa_*\}$ we have
\[
\eta\leq \min \left\{ \frac{\theta^*}2,\frac { ( T+ t_0+t_3)^{-\gamma}} {2M_*} \right\},
\]
which is possible due to  $\max\{T(x_0,\omega),{T}_i(x_0,\omega),t_0\}\leq C(1+d(x_0,S))$ and \eqref{3.2}.
Then
we can use Lemma \ref{L.7.1} with this $T$ and  $\eta$ (instead of Lemma \ref{L.3.1}) to see that \eqref{3.9} becomes
\[
    T(x_0,\omega)-t_0-{T}_i(x_0,\omega)\leq 
    {M_*}\eta \big(T+t_0+t_3\big)^{\al_2'} {T}_i(x_0,\omega) +2\kappa_*+\kappa_0+t_3
    \leq C(\rho+d(x_0,S)^{\be_1})
\]
because $\be_1\ge 1+\alpha_2'-\gamma$.  This ends the first half of the proof. 
Using again Lemma \ref{L.7.1} instead of Lemma \ref{L.3.1} in the second half of it, with $\eta$ as above, shows that
\[
u_{-}(t,x) :=u((1-M_*\eta( T+t_0)^{\al_2'})t+t_0,x)-\eta
\]
is a subsolution to \eqref{1.1} on
$
 (2\kappa_*,T)\times (\bbR^d\backslash \Gamma_{u,1-\eta}(t_0,\omega))$. The rest of the proof does not use the second claim in \eqref{2.1} and is unchanged (using this $u_-$ and also \eqref{7.5}), with  \eqref{3.10} becoming
\begin{align*}
 {T}_i(x_0,\omega)+t_0-T(x_0,\omega)\leq t_4+M_*\eta (T+t_0)^{\al_2'} T_i(x_0,\omega)\leq C(\rho+ d (x_0,S)^{\be_1}).
\end{align*}
This finishes the proof.


The proof of Lemma \ref{L.3.6} remains the same.   In the proof of Proposition \ref{P.3.3},  the change in Lemma \ref{L.3.3} turns the $C$ in  \eqref{3.22} and \eqref{3.18a} into $C(1+t^{\al_2})$ (recall that $s\le t$ in the argument), which is then added to the right-hand sides of \eqref{3.30} and \eqref{3.18}.  Then \eqref{3.18a} shows that there is $C>0$ such that $X_t=T(x,\cdot)$ for all $t\ge C(1+d(x,S))$,  and we now pick $N$ to be the smallest integer with $N\tau\ge C(1+d(x,S))$.  The estimate $N\leq C_1d(x,S)(\rho+d(x,S)^{\be_1})^{-1}$ now still holds (recall that $d(x,S)>\rho\geq 1$ here), and \eqref{3.555} remains unchanged because the term added to \eqref{3.18} is estimated by $C(1+(N\tau)^{\al_2})\le C(1+d(x,S)^{\be_1})$ because $\alpha_2<\be_1$.
%
The rest of the proof of Proposition \ref{P.3.3} remains the same.

Lemmas \ref{L.4.0'} and \ref{L.4.0} are unchanged except for replacement of \textbf{(H2')}, \textbf{(H4')}, $\calU_f$, $\calU_f'$, Lemma~\ref{L.3.1}, and Lemma \ref{L.4.3} in their statements and proofs by \textbf{(H2'')}, \textbf{(H4'')}, $\calU_{f,a}$, $\calU_{f,a}'$, Lemma \ref{L.7.1}, and Lemma \ref{L.7.2}, respectively (here we can even allow any fixed $a_2\in[0,\frac 12\theta^*]$ in \textbf{(H2'')}, although $a_2=0$ is sufficient).
The proof of Proposition \ref{P.3.4} also remains the same, using Lemma \ref{L.7.1} (with $T=\infty$ because $\eta=0$) instead of Lemma \ref{L.3.1}.

In Proposition \ref{P.3.1}, we replace \eqref{3.1} by
\beq\lb{7.333}
\be_3:=\max\left\{ \be_1,  \frac{(1+2\al_2')m_3}{m_3+2m_4} , \frac{2d+2}{2d+2+m_4'}\right\},
\eeq
which is still in $(0,1)$ thanks to $\alpha_2'<\frac{m_4}{m_3}$.  Then when we use Lemma \ref{L.4.3} in the proof, we replace it by Lemma \ref{L.7.2} with the same $\eta:=(\frac{\al_4}{\al_3}n^{-m_4})^{1/m_3}$, but now we need to pick $n$ so that
\[
\al_4^{1/m_3} \al_3^{-1/m_3} n^{-\frac{m_4}{m_3}}\leq \frac{1}{2}\min\{\theta^*,\,M_*^{-1}(T+\kappa_0+2c_0^{-1}R_0)^{-\al_2'}\},
\]
with $T:=\max\{T(x,\omega;S),T_n(x,\omega;S),2\kappa_*\}$.  Since $T \leq C(1+d(x,S))$ due to Lemma \ref{L.2.2}, and $\be_3\frac{m_4}{m_3}>\alpha_2'$ due to $\alpha_2'<\frac{m_4}{m_3}$ and \eqref{7.333}, there is again $C_1>0$ such that it suffices to let $n$ be the smallest integer for which
\eqref{4.333} holds.
Then a double application of Lemma \ref{L.7.2} replaces \eqref{4.40} by
\[
    |T_n(x,\cdot\,;S)-T(x,\cdot\,;S)|\leq Cn^{-\frac{m_4}{m_3}}\left(1+T\right)^{1+\al_2'}+C \leq  Cn^{-\frac{m_4}{m_3}}\left(1+d(x,S)\right)^{1+\al_2'}+C.
\]
Since there is again $C_2>0$ such that
\begin{align*}
4Cn^{-\frac{m_4}{m_3}}\left(1+d(x,S)\right)^{1+\al_2'}+4C \leq C_2(1+d(x,S)^{1+\al_2'-\be_3\frac{m_4}{m_3}}) \leq C_2(1+d(x,S)^{\frac{1+\be_3}2})
\end{align*}
 because $ 1+\al_2' - \be_3\frac{m_4}{m_3} \leq \frac{1+\be_3}{2}$, the rest of the proof of Proposition \ref{P.3.1} is unchanged.

Most of Section \ref{S5} is also unchanged,
with the only two adjustments needed in the proof of Proposition \ref{P.4.4}. 
We used Lemma \ref{L.3.1} when proving \eqref{5.5}, and we can just replace it by Lemma \ref{L.7.1} without any other change  because there we had $\eta=0$.  We also used Lemma \ref{L.3.1} when proving \eqref{5.5'}, and the change to Lemma \ref{L.7.1} now requires us to replace \eqref{4.4} by
\[
\eta:=C_1^{-1}\min\left\{\theta^*,M_*^{-1} l^{-\gamma}\right\}
\]
with $\gamma$ from \eqref{7.6} (recall that \eqref{7.5} shows $\gamma\ge 1+\alpha_2'-\be_1\ge 1-\be_1$, so $\eta\le l^{\be_1-1}$ as well; in fact, we could have chosen this $\eta$ in \eqref{4.4} as well) and $C_1\ge 2$ such that with $T=\max\{\kappa_0+\frac{2(l+m)}{c_0},2\kappa_*\}$ we have $\eta\le \frac 1{2M_*}(T+\tau_0^l)^{-\al_2'}$.  This is possible because of \eqref{5.333} and $\gamma\ge \al_2'$.
 Then
replacing Lemma \ref{L.3.1} by Lemma~\ref{L.7.1} with this $T$ yields
\begin{align*}
{T}((l+m)e,\omega;\calH_e^-+le) \leq (1+{M_*}\eta(T+\tau_0^l)^{\al_2'}) \big[ {T}((l+m)e,\omega;\calH_e^-)-T_1'(\omega) \big]+{\tau_0^l}+2\kappa_*+\kappa_0
\end{align*}
instead of \eqref{5.5'}. But the addition of $(T+\tau_0^l)^{\al_2'}$ here does not require further changes because
\[
{M_*}\eta(T+\tau_0^l)^{\al_2'} {T}((l+m)e,\omega;\calH_e^-) \le Cl^{1+\al_2'-\gamma}\le Cl^{\be_1}
\]
by \eqref{7.5} (recall that we assume here $l\ge m\ge \frac {4c_1}{c_0}l_0$).

The proofs of Propositions \ref{P.4.5} and \ref{P.5.3} then remain unchanged, finishing the proof.
\end{proof}

We are only able to obtain an \textbf{(H2'')}-version of Proposition \ref{P.6.4} when $a_2>0$, and we do so below.  But even without that, we can already prove Theorem \ref{T.1.7}(ii).

\begin{proof}[Proof of Theorem \ref{T.1.7}(ii)]
This is identical to the proof of Theorem \ref{T.1.3} above, with the word ``exclusive'' and \eqref{6.222} dropped, and using Proposition \ref{T.7.3} and \cite[Theorem 1.4(iii)]{zlatos2019} instead of Proposition \ref{P.6.4} and \cite[Theorem 5.4]{zlatos2019}, respectively.  Note that  $f$ is also  stationary ergodic in \cite[Theorem 1.4(iii)]{zlatos2019}, 
 but this is only used in the first paragraph of its proof to show existence of deterministic strong front speeds for all $e\in\bbS^{d-1}$ (which we proved in Proposition~\ref{T.7.3}), so that result extends to the case at hand.
\end{proof}



To prove Theorem \ref{T.1.7}(i), we  need to show that the deterministic strong front speeds from Proposition~\ref{T.7.3} are exclusive.  For this, we will need some uniform estimates on the reactions
\[
f_a(x,u,\omega):=\frac{f(x,(1-a)u+a,\omega)}{1-a}
\]
for $(x,u,\omega)\in \bbR^d\times[0,1]\times \Omega$.  Note that the transformation $u\mapsto\frac{u-a}{1-a}$ turns solutions $u$ to \eqref{1.1} for which $a\le u\le 1$ into solutions to \eqref{1.1} with $f_a$ in place of $f$ for which $0\le u\le 1$.



\begin{proposition}\lb{P.7.6}
Assume that  $f$ either satisfies \textbf{(H2'')} with $a_2>0$ and has a finite range of dependence, or satisfies \textbf{(H3)} and  \textbf{(H4'')} with $a_2>0$. Then for each $(a,e)\in (0,a_2]\times \bbS^{d-1}$, \eqref{1.1} with $f_a$ in place of $f$ has a deterministic strong front speed $c^*_a(e)\in [c_0,c_1]$ in direction $e$,
and $\lim_{a\to 0} c^*_a(e)=c^*(e)$ holds uniformly in $e\in \bbS^{d-1}$ (with $c^*(e)$ from Proposition \ref{T.7.3}). 
\end{proposition}

\begin{proof}
Clearly $f_a$ satisfies \textbf{(H1)} with the same $M$ and $m_1$, and $\theta_1$ and $\al_1$ replaced by $\frac 12\theta_1$ and $\al_1(1-\frac18 {\theta_1})^{m_1-1}$, respectively (recall that $a\le a_2\le\frac18{\theta_1}$).
%

If we now assume \textbf{(H2'')} and finite range of dependence of $f$, and let $u^a:=\frac {u-a}{1-a}$ for some $u\in \calU_{f,a}$ with initial datum $u_{0,k,a}$, then $u^a_t=\Delta u^a+f_a(x,u^a,\omega)$ on $(0,\infty)\times\bbR^d$ (with the same $\omega$) and its initial datum $\frac1{1-a}(u_{0,k,a}-a)$ satisfies \eqref{2.7'} and \eqref{2.6} (with $F$ now defined via $f_a$). 
Moreover, \eqref{7.000} and $a_2\le\frac 12$ show that we also have
\[
\begin{aligned}
&\limsup_{t\to\infty} \sup_{a\in[0,a_2]} \sup_{u\in\calU_{f,a}} \sup_{\eta>0}\frac{L_{u^a,\eta,1-2\theta^*}(t)}{t^{\al_2} \eta^{-{m_2}}}<\infty, \\
&    \liminf_{t\to\infty} \inf_{a\in[0,a_2]} \inf_{u\in\calU_{f,a}} \inf_{ u^a(t,x)\in [\theta^*,1-2\theta^*]}u^a_t(t,x)t^{\al_2'}>0.
\end{aligned}
\]
Hence for each such $f_a$ we have \textbf{(H2'')} with $a_2=0$, $\theta^*$ replaced by $2\theta^*$,  the above constants in \textbf{(H1)}, and $\calU_{f_a}:=\{u^a\,|\,u\in\calU_{f,a}\}$ (and the same $m_2,\alpha_2,\al_2'$).  Moreover, 
there are $\mu_*,\kappa_*>0$ such that \eqref{7.0} holds for all $a\in[0,a_2]$ and $u\in\calU_{f_a}$, with $\theta^*$ replaced by $2\theta^*$.  This and the remark after \eqref{2.10} (which shows that replacing $\theta^*$ by $2\theta^*$ in \eqref{7.0} does not change any of the above proofs) now show that Proposition \ref{T.7.3} holds for all the $f_a$, and all constants in its proof are uniform in $a\in[0,a_2]$. 
In particular, \eqref{7.0} holds with the same $\mu_*,\kappa_*>0$ (and $\theta^*$ replaced by $2\theta^*$) for all $a\in[0,a_2]$ and $u\in\calU_{f_a}'$ (see Lemma \ref{L.4.0'}), and the first claim in Proposition \ref{P.4.5} holds for $f_a$ with $\be_1$ from \eqref{7.111} and $C_\delta$ uniform in $a\in[0,a_2]$.  

The same argument applies when we assume \textbf{(H3)}+\textbf{(H4'')}, where the passage to $f_a$ and 
$f_{n,a}$ also 
replaces $\al_3 $ by $\al_3(1-\frac{\theta_1}8)^{m_3-1}$ in \textbf{(H3)}
and $\al_4$ by $\al_4 (1-\frac{\theta_1}8)^{-1}$ in \textbf{(H4'')}.
Again we obtain Proposition \ref{T.7.3}  for all the $f_a$, as well as that \eqref{7.0} holds with some $\mu_*,\kappa_*>0$ (and with $\theta^*$ replaced by $2\theta^*$) for all $a\in[0,a_2]$ and $u\in\calU_{f_a}'$, and the second claim in Proposition~\ref{P.4.5} holds for $f_a$ with $\be_3$ from \eqref{7.333} and $C_\delta$ uniform in $a\in[0,a_2]$.  

It therefore remains to prove the last claim, with the above deterministic strong front speeds denoted $c^*_a(e)$ (where clearly $c^*_0(e)=c^*(e)$).  To achieve this, we will use  uniformity of the estimates in Proposition~\ref{P.4.5} in $a\in[0,a_2]$.  We  therefore denote by $v_{e,a}(\cdot,\cdot,\omega)$
the solution to \eqref{4.1} with $f$ replaced by $f_a$, some $(e,\omega)\in\bbS^{d-1}\times\Omega$, and $l=0$ (then of course $v_{e,a}(\cdot,\cdot,\omega)\in \calU_{f_a}'$).  We also let
\beq\lb{7.444}
u_{e,a}(\cdot,\cdot,\omega):=(1-a)v_{e,a}(\cdot,\cdot,\omega)+a\in\calU_{f,a}',
\eeq
with $\calU_{f,a}'$ obtained from $\calU_{f,a}$ as in Lemma \ref{L.4.0'}, and for any $x\in\bbR^d$,
\begin{align*}
T_{e,a}(x,\omega) & :=\inf\{t\geq 0\,|\, v_{e,a}(t,x,\omega)\geq 1-\theta^*\}=\inf\{t\geq 0\,|\, u_{e,a}(t,x,\omega)\geq 1-(1-a)\theta^*\}, \\
T_{e,a}'(x,\omega) & :=\inf\{t\geq 0\,|\, u_{e,a}(t,x,\omega)\geq 1-\theta^*\}.
\end{align*}
These definitions, Lemma \ref{L.2.2}, and $a_2\le \frac 12\theta^*$ show that there is $C$ such that 
\beq\lb{7.777}
T_{e,a}'\le T_{e,a}\le T_{e,a}'+C. 
\eeq
We will treat both cases \textbf{(H2'')}+finite range and \textbf{(H3)}+\textbf{(H4'')} at once, using the notation from either \eqref{b.1'} in the first case or \eqref{b.1''} in the second. Then the claims from Proposition~\ref{P.4.5}, with $\delta:=\frac{1-\beta}2$ and $\bar C'_\rho:=C_{(1-\beta)/2} \bar C_\rho$ independent of $(e,a,\omega)$ become
\beq\lb{7.555}
\left|\frac {\bbE[T_{e,a}(le,\omega)]} {l} - \frac 1{c_a^*(e)} \right|\leq \bar C'_\rho\, l^{-({1-\beta})/2}
\eeq
for all $(e,a,\omega)\in\bbS^{d-1}\times[0,a_2]\times\Omega$  and all $l\ge 1$.


Lemma \ref{L.2.2} shows that $u_{e,a}(\tau_0,\cdot,\cdot)\geq u_{e,0}(0,\cdot,\cdot)$ with $\tau_0:=\kappa_0+\frac{2R_0}{c_0}$, hence the comparison principle yields $u_{e,a}(\tau_0+t,\cdot,\cdot)\geq u_{e,0}(t,\cdot,\cdot) $ for all $t\geq 0$.  This and \eqref{7.444} immediately imply
\[
T_{e,a}'\le T_{e,0}+\tau_0,
\]
so $c_a^*(e)\geq c^*(e)$ for all $(e,a)\in\bbS^{d-1}\times[0,a_2]$ by \eqref{7.777} (this also shows that $c_a^*(e)\geq c_0$).

Since $T_{e,a}'(le,\cdot)\leq Cl$ for all $l\geq 1$ by Lemma \ref{L.2.2},
 Lemma \ref{L.7.1} with $f_1=f_2=f$ and 
\[
(u_1,u_2,\eta,t_0,T,R)=(u_{e,0},u_{e,a}, a,\tau_0,\max\{T_{e,a}'(le,\cdot),2\kappa_*\},\infty)
\]
yields
\[
T_{e,0}(le,\cdot)\leq T_{e,a}'(le,\cdot) + C(1+ l^{1+\al_2'}a).
\]
as long as $l \in[1,(C'a)^{-1/\al_2'}]$ (for some $C,C'>0$).  It follows by \eqref{7.777} that for such $l$ we have
\[
T_{e,0}(le,\cdot)\le T_{e,a}(le,\cdot) + C(1+ l^{1+\al_2'}a).
\]
Picking $l:=a^{-1/2\al_2'}$ and using \eqref{7.555} now yield for all small enough $a$ (depending only on \eqref{const} and $\alpha_2$),
\[
c^*(e)^{-1}\le c^*_a(e)^{-1} + C a^{1/2} + \bar C'_\rho\, a^{{(1-\beta)}/4\al_2'}.
\]
Since $\beta<1$, $\al_2'>0$, and $c_a^*(e)\geq c^*(e)$, 
the uniform convergence claim follows.
%
\end{proof}

We can now extend Proposition \ref{P.6.4} to the case $a_2>0$.

\begin{proposition}\lb{P.7.4}
Under the hypotheses of Proposition \ref{P.7.6}, for each $e\in\bbS^{d-1}$, the speed $c^*(e)$ is a deterministic strong  exclusive front speed in direction $e$ for \eqref{1.1}.
\end{proposition}

\begin{proof}
Fix any $e\in\bbS^{d-1}$, and let $u_{e,a}$ be from \eqref{7.444} and  $c_a^*(e)$ from Proposition \ref{P.7.6}.
From that proposition  and $f(\cdot,a,\cdot)\equiv 0$ for all $a\in[0,a_2]$ we know that for almost all $\omega\in\Omega$ we have 
\beq\lb{7.9}
\lim_{t\to\infty}\sup_{|y|\leq \Lambda t} \, \sup_{x\in(c_a^*(e)e+K)t}u_{e,a}(t,x,\Upsilon_y\omega)=a
\eeq
for each $a\in[0,a_2]\cap\bbQ$,  $\Lambda>0$, and compact $K\subseteq\calH^{+}_e=\{x\in \bbR^d\,|\, x\cdot e>0\}$.  Fix any such $\omega$, and then any compact $K\subseteq\calH^{+}_e$.
Let $K'\subseteq \calH_e^+$ be compact and such that $K\subseteq (K')^0$. 
Proposition \ref{P.7.6} then yields $\delta\in (0,a_2]$ 
such that for all $a\in[0, \delta]$ we have 
\beq\lb{7.11}
c^*(e)e+K\subseteq (c_a^*(e)e+K')(1+2\delta).
\eeq
Since $\alpha_2'<\min\{ \frac 1{m_1-1}, \frac {1-\alpha_2}{m_2} \}$, there exists $T_0\geq \kappa_*$ such that for all
 $T\geq T_0$ we have
 \beq\lb{7.13}
\eta_T\leq \frac{\theta^*}2\qquad\text{and}\qquad \max\left\{ M_*(T+\kappa_*+\tau_T)^{\al_2'}\eta_T,(\tau_T+2\kappa_*) T^{-1} \right\}\leq {\delta},
\eeq
where
 $\eta_T:=T^{-\gamma}$ with $\gamma:=\frac{1}{2}(\al_2'+\frac{1}{m_1-1})$, and $\tau_T:=1+D_1\eta_T^{1-m_1}$ with $D_1$ from Lemma~\ref{Cor.2.1}. 
 
 Now fix any $\Lambda>0$.
We see from \eqref{7.9} that for each $a\in [0,a_2]\cap\bbQ$, there is a function $\varphi_{a}:[0,\infty)\to [0,\infty)$ 
such that $\lim_{t\to\infty}\varphi_{a}(t)=0$ and
\beq\lb{7.99}
\sup_{t\geq T}\sup_{|y|\leq \Lambda t}\sup_{x\in (c_a^*(e)e+K')t}u_{e,a}(t,x,\Upsilon_y\omega)\leq a+\varphi_{a}(T).
\eeq
Pick any $T\ge T_0$ and $a\in[0, \delta]\cap\bbQ$, and let $w_{e,a}$ be from Definition \ref{D.5.0}.
Then Lemma \ref{Cor.2.1} yields $u_{e,a}(\tau_T,\cdot,\cdot)\geq 1-\eta_T$ on $\calH_e^-$, so from $u_{e,a}\ge a$ and  $(u_{e,a})_t\geq 0$ we see that
\beq\lb{7.16}
u_{e,a}(t+\tau_T,\cdot,\cdot)+\eta_T \geq w_{e,a}(0,\cdot,\cdot)
\eeq
for all $t\ge 0$.  Since Lemma \ref{L.7.1} and  \eqref{7.13} show that
\[
u_+(t,x,\cdot) :=u_{e,a}((1+M_*(T+\kappa_*+\tau_T)^{\al_2'}\eta_T)t+\tau_T,x,\cdot )+\eta_T
\]
is a supersolution to \eqref{1.1} on $(\kappa_*,T+\kappa_*)\times\bbR^d$, the comparison principle and \eqref{7.16} yield 
\[
u_+(t+\kappa_*,x,\cdot)\geq w_{e,a}(t,x,\cdot)
\]
for all $(t,x)\in [0,T]\times\bbR^d$.  This, \eqref{7.13},  \eqref{7.11}, \eqref{7.99}, and $(u_{e,a})_t\geq 0$ now show that
\begin{align*}
    \sup_{|y|\leq \Lambda T} & \sup_{x\in (c^*(e)e+K)T} w_{e,a}(T,x,\Upsilon_y\omega) \\
    &\leq \sup_{|y|\leq \Lambda T}\sup_{x\in (c^*(e)e+K)T} u_{e,a}((1+M_*(T+\kappa_*+\tau_T)^{\al_2'}\eta_T)(T+\kappa_*)+\tau_T,x,\Upsilon_y\omega)+\eta_T\\
        &\leq \sup_{|y|\leq \Lambda T}\sup_{x\in (c^*(e)e+K)T} u_{e,a}((1+\delta)T+\tau_T+2\kappa_*,x,\Upsilon_y\omega)+\eta_T\\
    &\leq \sup_{|y|\leq \Lambda T}\sup_{x\in (c^*(e)e+K)T} u_{e,a}((1+2\delta)T,x,\Upsilon_y\omega)+\eta_T \\
    &\leq \sup_{|y|\leq \Lambda(1+2\delta) T} \, \sup_{x\in (c_a^*(e)e+K')(1+2\delta)T} u_{e,a}((1+2\delta)T,x,\Upsilon_y\omega)+\eta_T \\
    &\leq a+\varphi_{a}((1+2\delta)T)+T^{-\gamma}.
\end{align*}
 Hence
\[
\lim_{t\to\infty}\sup_{|y|\leq \Lambda t}\sup_{x\in (c^*(e)e+K)t}w_{e,a}(t,x,\Upsilon_y\omega)\le  a
\]
for all $a\in[0,\delta]\cap\bbQ$ (and the previously fixed $(\omega,K,\Lambda)$).  
Since $w_{e,a}$ is non-decreasing in $a$, this  extends to all $a\in [0,\delta]$. 
And since $\delta$ does not depend on $\Lambda$, we obtain \eqref{6.222} with $\lambda_{K,\omega,e}(a)=a+(1-a)\chi_{(\delta,1]}(a)$, so the result follows. 
\end{proof}

\begin{proof}[Proof of Theorem \ref{T.1.7}(i)]
This is now identical to the proof of Theorem \ref{T.1.3} in Section \ref{S6}, using Proposition \ref{P.7.4} in place of Proposition \ref{P.6.4}.
\end{proof}

\medskip

\appendix

\section{Proof of Lemma \ref{L.2.4}}

If $F_0$ is the function defined before Lemma \ref{L.2.2}, then we have 
\beq\lb{A.5}
\delta=\delta(M,\theta_1,m_1,\al_1):= \min_{u\in [1-2\theta_1/3,1-\theta^*]} {F_0}(u)>0.
\eeq

We now claim that for each ${L}\geq 1$, there is $R_{L}:=R_L(M,\theta_1,m_1,\al_1)$ and a smooth function $u_{L}:\bbR^d\to\bbR$ such that 
\begin{align}
   (1-\theta^*)\chi_{S} \leq u_{L} & \leq (1-\theta^*)\chi_{B_{R_{L}}(S)},\lb{A.6}\\
|\Delta u_{L}|+|\nabla u_{L}|^2 & \leq \frac \delta L   \lb{A.4}
\end{align}
hold on $\bbR^d$, and for each $x\in\bbR^d$ with $u_L(x)< \frac{1-\theta^*}{3}$ we have
\beq \lb{A.1}
\Delta u_{L}(x)\geq 0.
\eeq
Note that if we also had $1-\frac 23 \theta_1\leq \frac{1-\theta^*}{3}$ (which is not the case), 
then \eqref{A.5} and \eqref{A.4} would show that for such $u_L$ (with $L\ge 1$) we have
\[
{F_0}\left(u_L(x)\right)\geq \delta\geq -\Delta u_{L} (x)
\]
whenever $u_L(x)\in\left[\frac{1-\theta^*}{3},1-\theta^*\right]$,
so this and \eqref{A.1} would yield
\[
\Delta u_{L}+{F_0}(u_{L})\geq 0
\]
on $\bbR^d$.  Hence the result would follow with $u_{0,S}:=u_{1}$ and $R_0:=R_{1}$ because $F_0\le F$.

Let us now prove the claim.
For any $a\in (0,\frac{1}{8})$, let $0\not\equiv \xi_a:\bbR^d\to\bbR$ be a smooth, radially symmetric, non-negative function supported in $B_{a}(0)$, and define 
\[
\varphi_a:=\frac{\zeta*\xi_a}{\|\zeta*\xi_a\|_{L^1}},
\]
where \begin{equation*}
    {\zeta}(x):=\left\{\begin{aligned}
    &\left({|x|^{2-d}}-2^{d-2}\right)_+&&\text{ if }d\geq 3,\\
    &\ln_- (2|x|)&&\text{ if }d=2.
    \end{aligned}\right.
\end{equation*}
Notice that ${\zeta}$ is sub-harmonic on $\bbR^d\backslash \{0\}$, and it is supported  and integrable in $B_{1/2}(0)$. Therefore it is not hard to see that
\beq\lb{A.i}
\lim_{a\to\infty}\int_{B_a(0)}\varphi_a(x)dx=0.
\eeq
And since $\xi_a$ is supported in $B_a(0)$, we also have
\[
\Delta \varphi_a(x)=\int_{\bbR^d}\Delta\zeta(x-y)\xi_a(y)dy \geq 0
\]
for all $x\in \bbR^d\setminus B_a(0)$.
Thus, for any $R\geq 1$, the function 
\[
 \varphi_{a,R}(x):=R^{-d} \varphi_a (R^{-1}x)
\]
 satisfies
\beq\lb{A.ii}
\Delta \varphi_{a,R}\geq 0
\eeq
on $\bbR^d\setminus B_{aR}(0)$.

Next, for some $N\geq 1$ (to be determined later), take 
\[
u=u_{a,R,N,S}:=(1-\theta^*)\chi_{B_{NR}(S)}*\varphi_{a,R}.
\]
Direct computations then yield
\begin{align*}
    |\nabla u(x)|  & \leq \int_{\bbR^d}|\nabla \varphi_{a,R}(y)| dy
    = R^{-1}\int_{B_{1}(0)}|\nabla \varphi_a (y)|dy, \\
|\Delta u(x)| & \leq \int_{\bbR^d}\left|\Delta \varphi_{a,R}(y)\right| dy
    = R^{-2}\int_{B_{1}(0)}\left|\Delta \varphi_a (y)\right|dy
\end{align*}
because $\varphi_a$ is supported in $B_1(0)$.  Hence \eqref{A.4} will hold with $u_L:=u$ provided $R=R(a,\delta,{L})$ is chosen  large enough.  And then $N\ge 1$ shows that \eqref{A.6} will also hold as long as we pick $R_L\ge (N+1)R$ (given this $R$, as well as some yet to be determined $a$ and $N$).

It remains to show \eqref{A.1} when $u(x)< \frac{1-\theta^*}{3}$.  If $d(x,S)\ge (N+a)R$, then \eqref{A.ii} yields
\[
\Delta u (x)= (1-\theta^*)\left(\chi_{B_{NR}(S)}*\Delta\varphi_{a,R} \right)(x)\geq 0,
\] 
so \eqref{A.1} holds. If $d(x,S) \le (N+a)R$, let $z\in\overline S\cap\overline{B_{(N+a)R}(x)}$.  Then
\[
 (\chi_{B_{NR}(S)}*\varphi_{a,R})(x)
 \geq  \int_{ B_{NR}(z)}\varphi_{a,R}(x-y)dy= \int_{ B_{N}(z')}\varphi_{a}(y)dy,
\]
with $z':=\frac{z-x}R$, so $|z'|\le N+a$.
From \eqref{A.i} and radial symmetry of $\varphi_a$, we get
\[
\lim_{N\to\infty}\lim_{a\to 0}\int_{ B_{N}((N+a)(1,0,\dots,0))}\varphi_{a}(y)dy=\frac{1}{2},
\]
so there are universal $a\in(0,\frac 18)$ and $N\ge 1$ such that the last integral is at least $\frac 13$.  Then
\begin{align*}
    u(x)=(1-\theta^*)(\chi_{B_{NR}(S)}*\varphi_{a,R})(x)\geq \frac{1-\theta^*}{3}
\end{align*}
holds when $d(x,S)\ge (N+a)R$, so  \eqref{A.1} holds when $u(x)< \frac{1-\theta^*}{3}$ and the claim is proved.

Next, to prove the lemma, recall that $1-\frac 23 \theta_1\in( \frac{1-\theta^*}{3},1-\theta^*)$ and
take $u_{0,S}:=\psi(u_{L})$, for some $L\ge 1$ and some $\psi:[0,1-\theta^*]\to[0,1-\theta^*]$ satisfying the following:
\begin{enumerate}[(i)]
    \item $\psi $ is smooth and non-decreasing on $[0,1-\theta^*]$;
    \smallskip
    \item $\psi(0)=0$, $\psi(1-\theta^*)=1-\theta^*$ and $\psi (\frac{1-\theta^*}3 )=1-\frac 23 \theta_1$,
    \smallskip
    \item $\psi''= 0$ on $[0,\frac{1-\theta^*}3]$.
\end{enumerate}
From (i,ii) and \eqref{A.6} we clearly have
\[
(1-\theta^*)\chi_{S}\leq u_{0,S}\leq (1-\theta^*)\chi_{B_{R_{L}}(S)},
\]
so it suffices to take $R_0:=R_L$ and verify \eqref{2.6}. 

When $ u_{L}(x)< \frac{1-\theta^*}{3}$,  \eqref{A.1} and (i,iii) yield
\[
\Delta \psi(u_{L}(x))+{F_0}(\psi(u_{L}(x)))\geq \psi'(u_L(x)) \Delta u_{L}(x)+\psi'' (u_L(x)) |\nabla u_{L}(x)|^2\geq 0.
\]
When $u_{L}(x)\geq \frac{1-\theta^*}{3}$, \eqref{A.5} and (ii) yield $ {F_0}(\psi(u_{L}(x)))\geq \delta$.
Hence with
\[
{L}=L(M,\theta_1,m_1,\al_1):=\max\{\|\psi'\|_\infty,\,\|\psi''\|_\infty\},
\]
we get 
\begin{align*}
\Delta \psi(u_{L}(x))+{F_0}(\psi(u_{L}(x)))&\geq \psi'(u_L(x)) \Delta u_{L}(x)+\psi'' (u_L(x)) |\nabla u_{L}(x)|^2 +\delta\\
& \geq \delta-{L}(|\Delta u_{L}(x)|+|\nabla u_{L}(x)|^2).
\end{align*}
So \eqref{2.6} follows from \eqref{A.4} and $F_0\le F$, concluding the proof.

\section{Proof of Lemma \ref{L.3.1}}



Let us drop $\omega$ from the notation.  Also recall that we extend the reactions by 0 to $u\notin[0,1]$.


Let us start with four estimates involving the reactions where $u_1(t,x)\notin(\theta^*,1-\theta^*)$.
From \eqref{2.10} and $\eta\le\frac{\theta^*}2$ we get ${\theta^*}+\eta\leq \theta_1$. Hence \textbf{(H1)} shows that for $u\leq {\theta^*}$ we have 
\beq \lb{8.1}
f_1(\cdot,u)\equiv f_2(\cdot,u\pm \eta) \equiv 0
\eeq
 on $\bbR^d$, while for  $u\geq 1-\eta$ we have
\beq \lb{8.2}
f_1(\cdot,u)\geq 0 \equiv  f_2(\cdot,u+{\eta})
\eeq
 on $\bbR^d$.  If $u_1(t,x)\in [1-\theta^*,1-\eta)$ for some $(t,x)\in[t_0,\infty)\times\bbR^d$, then $(u_1)_t\geq 0$ shows that $u_1(t_0,x)<1-\eta$, so \textbf{(H1)} and \eqref{2.4} yield
\beq \lb{8.3}
f_1(x,u_1(t,x))=f_2(x,u_1(t,x))\geq f_2(x,u_1(t,x)+{\eta}).
\eeq
Finally, if $u_1(t,x)\in [1-\theta^*,1]$ for some $(t,x)\in[t_0,\infty)\times\bbR^d$ and $u_1(t_0,x)<1-\eta$, then \textbf{(H1)} and \eqref{2.4} again yield
\beq \lb{8.4}
f_1(x,u_1(t,x))\leq f_1(x,u_1(t,x)-\eta)= f_2(x,u_1(t,x)-{\eta}).
\eeq

Denote $\tau_\pm(t):=(1\pm {M_*}\eta)\,t+t_0$, so that $u_\pm(t,x)=u_1(\tau_\pm(t),x)\pm\eta$.  If now  $u_1(\tau_+(t),x)\notin (\theta^*,1-\theta^*)$ for some $(t,x)\in(0,\infty)\times\bbR^d$, then \eqref{8.1}, \eqref{8.2}, and \eqref{8.3} yield
\begin{align*}
    [(u_+)_t-&\Delta u_+-f_2(\cdot,u_+)](t,x)\\
    \geq\,& (1+{M_*}\eta)(u_1)_t(\tau_+(t),x)-\Delta u_1(\tau_+(t),x)-f_1(x,u_1(\tau_+(t),x))\\
    = \,&{M_*} \eta (u_1)_t(\tau_+(t),x) \\
     \geq\,& 0.
\end{align*}
Similarly,  if $u_1(\tau_-(t),x)\notin (\theta^*,1-\theta^*)$ and $u_1(t_0,x)<1-\eta$, then  \eqref{8.1}, \eqref{8.2}, and \eqref{8.4} yield
\[
[(u_-)_t-\Delta u_--f_2(\cdot,u_-)](t,x)\leq 0.
\]

Let us now consider those $(t,x)\in (\kappa_*,\infty)\times\bbR^d$ for which $u_1(\tau_+(t),x)\in (\theta^*,1-\theta^*)$.  Then  $(u_1)_t(\tau_+(t),x) \geq \mu_*$ by  \eqref{2.1}, so 
$|f_2(x,u_+(t,x))-f_2(x,u_1(\tau_+(t),x))|\leq  {M}\eta$ yields
\begin{align*}
  [(u_+)_t- &\Delta u_+-f_2(\cdot,u_+)](t,x)\\ 
  \geq\,& (1+{M_*}\eta)(u_1)_t(\tau_+(t),x)-\Delta u_1(\tau_+(t),x)-f_2(x,u_1(\tau_+(t),x))- M\eta\\
  \ge \,& (u_1)_t(\tau_+(t),x)-\Delta u_1(\tau_+(t),x)-f_1(x,u_1(\tau_+(t),x)) + {M_*}\eta\mu_*- M\eta\\
    \geq\, & 0,
\end{align*}
where we again used \eqref{2.4} due to $u_1(t_0,x)<1-\theta^*<1-\eta$.
Similarly if $u_1(\tau_-(t),x)\in (\theta^*,1-\theta^*)$ for some $(t,x)\in ( 2\kappa_*,\infty)\times\bbR^d$ (so $\tau_-(t)>\kappa_*$ because $M_*\eta\le \frac 12$), 
we obtain
\[
[(u_-)_t-\Delta u_--f_2(\cdot,u_-)](t,x)\leq 0.
\]
This proves the claims about $u_+$ and $u_-$.

If now $u_2(0,\cdot)\leq u_1(t_0,\cdot)+\eta$ on $B_R(y)$, from $(u_1)_t\geq 0$ we also obtain $u_2(0,\cdot)\leq u_+(\kappa_*,\cdot)$ there.
Since $u_+$ is a supersolution to \eqref{1.1} with $f_2$ in place of $f$ on $(\kappa_*,\infty)\times B_R(y)$, Lemma~\ref{L.2.1} yields
\[
u_2(t,y)\leq u_+(t+\kappa_*,y)+2de^{2M t -  \sqrt{M/d\,}R}
\]
 for all 
$t\ge 0$.  Hence,
\[
u_2(t,y)\leq u_+(t+\kappa_*,y)+\frac{\theta^*}{2}
\]
for all 
$t\in [0,{T}_{u_2}(y)]$ as long as 
\[
R \geq  2\sqrt{Md\,}{T}_{u_2}(y)+\sqrt{d/M}\ln\frac {4d}{\theta^*},
\]
which will be guaranteed by taking $D_2:=2\sqrt{Md}\ln\frac{4d}{\theta^*}$.

It follows from $\eta\le\frac{\theta^*}2$ and the definition of $T_{u_2}(y)$ that,
\[
    u_1(\tau_+(T_{u_2}(y)+\kappa_*),y) \ge u_2(T_{u_2}(y), y)-\eta-\frac{\theta^*}{2}
\ge 1-2\theta^*.
\]
By Lemma \ref{L.2.2},
we have
\[
u_1(\tau_+(T_{u_2}(y)+\kappa_*)+\kappa_0,y)\geq 1-\theta^*.
\]
Therefore
\beq \lb{8.5}
T_{u_1}(y)\leq \tau_+(T_{u_2}(y)+\kappa_*)+\kappa_0\leq (1+{M_*}\eta)T_{u_2}(y)+2\kappa_*+\kappa_0 + t_0.
\eeq


\section{Proof of Lemma \ref{L.4.3}}

We again have \eqref{8.1} and \eqref{8.2}.
For $(x,u)\in B_R(y)\times [1-\theta^*,1-\eta]$, \textbf{(H3)} and \eqref{3.3} yield either 
\begin{align*}
    f_1(x,u)\geq f_1(x,u+\eta)+\al_3\eta^{m_3}\geq f_2(x,u+{\eta})
\end{align*}
(if $f_1$ satisfies \textbf{(H3)}) or 
\begin{align*}
    f_1(x,u)\geq f_2(x,u)-\al_3\eta^{m_3}\geq f_2(x,u+{\eta})
\end{align*}
(if $f_2$ does), replacing \eqref{8.3}.
Hence, as in the proof of Lemma \ref{L.3.1} and  with $\tau_+,u_+$ from it, 
we have 
that if $u_1(\tau_+(t),x)\notin (\theta^*,1-\theta^*)$ for some $(t,x)\in(0,\infty)\times B_R(y)$, then
\[
    [(u_+)_t- \Delta u_+-f_2(\cdot,u_+)](t,x)\ge 0.
\]

Let us now consider those $(t,x)\in (\kappa_*,\infty)\times B_R(y)$ for which $u_1(\tau_+(t),x)\in (\theta^*,1-\theta^*)$.  Then  $(u_1)_t(\tau_+(t),x) \geq \mu_*$ by  \eqref{2.1}, so 
\[
f_2(x,u_+(t,x))-f_1(x,u_1(\tau_+(t),x)) \leq  \al_3 \eta^{m_3}+{M}\eta\leq (1+M)\eta
\]
yields
\begin{align*}
   [(u_+)_t-& \Delta u_+-f_2(\cdot,u_+)](t,x)\\
  \geq\,& (1+{M_*}\eta)(u_1)_t(\tau_+(t),x)-(\Delta u_1)(\tau_+(t),x)-f_1(x,u_1(\tau_+(t),x))-(1+ M)\eta\\
    \geq\,& {M_*} \eta \mu_*-(1+ M)\eta\\
    =\,&0.
\end{align*}

Hence $u_+$ is a supersolution to \eqref{1.1} with $f_2$ in place of $f$ on $(\kappa_*,\infty)\times B_R(y)$.  Since we also have $u_2(0,\cdot)\le u_+(\kappa_*,\cdot)$ on $B_R(y)$ due to $(u_1)_t\ge 0$, 
\eqref{8.5} follows via Lemmas \ref{L.2.1} and \ref{L.2.2}
as at the end of the proof of Lemma \ref{L.3.1}.






\end{document}